\documentclass[11pt]{article}
\usepackage{amsthm,amsmath,mathrsfs,enumitem,amssymb,mathtools,mathrsfs}
\usepackage{xcolor}
\usepackage{hyperref}
\usepackage[english]{babel}
\usepackage{graphicx}
\usepackage[hypcap]{caption}
\usepackage{enumitem}

\theoremstyle{plain}
\newtheorem{theorem}{Theorem}[section]
\newtheorem{lemma}{Lemma}[section]
\newtheorem{proposition}{Proposition}[section]

\theoremstyle{definition}
\newtheorem{definition}{Definition}[section]
\theoremstyle{remark}
\newtheorem{remark}{Remark}[section]
\numberwithin{equation}{section}

\setlist{nosep}

\newcommand{\N}{\mathbb{N}}
\newcommand{\B}{\mathcal{B}}
\newcommand{\C}{\mathcal{C}}
\newcommand{\X}{\mathcal{X}}
\newcommand{\R}{\mathbb{R}}
\newcommand{\M}{\mathcal{M}}
\newcommand{\Z}{\mathcal{Z}}
\newcommand{\E}{\mathcal{E}}
\newcommand{\K}{\mathscr{K}}
\newcommand{\Space}{\mathcal{S}}
\newcommand{\LN}{\mathscr{L}^N}
\newcommand{\CBZ}{C_0(\Z)}
\newcommand{\MZ}{\mathcal{M}_1(\mathcal{Z})}
\newcommand{\MNZ}{\mathcal{M}_1^N(\mathcal{Z})}
\newcommand{\projinv}{\sigma^{-1}_0}
\newcommand{\Msi}{\M_{1,\psi}(\X)}
\newcommand{\DZ}{D([0,T],\Z)}
\newcommand{\MDZ}{\mathcal{M}_1(\DZ)}
\newcommand{\DMZ}{D([0,T],\mathcal{M}_1(\mathcal{Z}))}

\newcommand{\EQ}{\mathcal{E}_Q}
\newcommand{\EW}{\mathcal{E}_W}

\newcommand{\indf}{\mathbf{1}_}
\newcommand{\PN}{\mathbb{P}^N}
\newcommand{\PbarN}{\bar{\mathbb{P}}^N}

\newcommand{\EbarN}{\mathbb{\bar{E}}^N}
\newcommand{\PQN}{\mathcal{P}^N}
\newcommand{\PbarQN}{\bar{\mathcal{P}}^N}

\newcommand{\emp}{\text{emp}}

\title{A sufficient condition for the quasipotential to be the rate function of the invariant measure of countable-state mean-field interacting particle systems}
\author{Sarath Yasodharan\thanks{Most of this work was competed when SY was a PhD student at the Indian Institute of Science.} and Rajesh Sundaresan 
\\ {\small Brown University and Indian Institute of Science}}

\begin{document}
\maketitle
\begin{abstract}
This paper considers the family of invariant measures of Markovian mean-field interacting particle systems on a countably infinite state space and studies its large deviation asymptotics. The Freidlin-Wentzell quasipotential is the usual candidate rate function for the sequence of invariant measures indexed by the number of particles. The paper provides two counterexamples where the quasipotential is not the rate function. The quasipotential arises from finite horizon considerations. However there are certain barriers that cannot be surmounted easily in any finite time horizon, but these barriers can be crossed in the stationary regime. Consequently, the quasipotential is infinite at some points where the rate function is finite. After highlighting this phenomenon, the paper studies some sufficient conditions on a class of interacting particle systems under which one can continue to assert that the Freidlin-Wentzell quasipotential is indeed the rate function.

\vspace{10pt}

\end{abstract}

\section{Introduction}
For a broad class of Markov processes such as small-noise diffusions, finite-state mean-field models, simple exclusion processes, etc., it is well-known that the Freidlin-Wentzell quasipotential is the rate function that governs the large deviation principle (LDP) for the family of invariant measures \cite{freidlin-wentzell,sowers-90,borkar-sundaresan-12,farfan-etal-18}. The quasipotential is the minimum cost (arising from the rate function for a process-level large deviation principle) associated with trajectories of arbitrary but finite duration, with fixed initial and terminal conditions. We begin this paper with two counterexamples of independently evolving countable-state particle systems for which the quasipotential is not the rate function for the family of invariant measures. The family of invariant measures of these counterexamples satisfy the LDP with a suitable relative entropy as its rate function, and we show that the quasipotential is not the same as this relative entropy. Specifically, we show that there are points in the state space where the  rate function is finite, but the quasipotential is infinite. These points cannot be reached easily via trajectories of arbitrary but finite time duration. However the barriers to reach these points are surmounted in the stationary regime. There are however some sufficient conditions, at least on a family of such countable-state interacting particle systems, where the Freidlin-Wentzell quasipotential is indeed the correct rate function; this will be the main result of this paper. Intuitively, the sufficient conditions cut-down the speed of outward excursions and ensure that the insurmountable barriers for the finite horizon trajectories continue to be insurmountable in the stationary regime.

Before we describe the counterexamples and the main result, let us introduce some notations and describe the model of a countable-state mean-field interacting particle system. Let $\Z$ denote the set of non-negative integers and let $(\Z, \E)$ denote a directed graph on $\Z$. Let $\MZ$ denote the space of probability measures on $\Z$ equipped with the total variation metric (which we denote by $d$). For each $N \geq 1$, let $\MNZ \subset \MZ$ denote the set of probability measures on $\Z$ that can arise as empirical measures of $N$-particle configurations on $\Z^N$. For each $N \geq 1$, we consider a Markov process with  the infinitesimal generator acting on functions $f$ on $\MNZ$:
\begin{align}
\LN f(\xi) \coloneqq  \sum_{(z,z^\prime) \in \E} N\xi(z) \lambda_{z,z^\prime}(\xi) \left[f\left(\xi+\frac{\delta_{z^\prime}}{N} - \frac{\delta_z}{N}\right) - f(\xi)\right],\, \xi \in \MNZ;
\label{eqn:LN}
\end{align}
here $\lambda_{z,z^\prime}:\MZ \to \R_+$, $(z,z^\prime) \in \E$, are given functions that describe the transition rates and $\delta$ denotes the Dirac measure. Such processes arise as the empirical measure of  weakly interacting Markovian mean-field particle systems where the evolution of the state of a particle depends on the states of the other particles only through the empirical measure of the states of all the particles. Under suitable assumptions on the model, the martingale problem for $\LN$ is well posed and the associated Markov process possesses a unique invariant probability measure $\wp^N$. This paper highlights certain nuances associated with  the large deviation principle for the sequence $\{\wp^N, N \geq 1\}$ on $\MZ$.

Fix $T > 0$ and let $\mu^N_{\nu_N}$ denote the Markov process with initial condition $\nu_N \in \MNZ$ whose  infinitesimal generator is $\LN$. Its sample paths are elements of $D([0,T], \MNZ)$, the space of $\MNZ$-valued functions on $[0,T]$ that are right-continuous with left limits equipped with the Skorohod topology. Such processes have been well studied in the past. Under mild conditions on the transition rates, when $\nu_N \to \nu$ in $\MZ$ as $N \to \infty$, it is well-known that the family $\{\mu^N_{\nu_N}, N \geq 1\}$ converges in probability, in $\DMZ$, as $N \to \infty$ to the {\em mean-field limit}\footnote{See McKean~\cite{mckean-67} in the context of interacting diffusions and Bordenave et al.~\cite{bordenave-etal-12} in the context of countable-state mean-field models.}:
\begin{align}
\dot{\mu}(t) = \Lambda_{\mu(t)}^* \mu(t), \mu(0) = \nu,\, t \in [0,T];
\label{eqn:mve}
\end{align}
here $\dot{\mu}(t)$ denotes the derivative of $\mu$ at time $t$, $\Lambda_{\xi}$, $\xi \in \MZ$, denotes the rate matrix when the empirical measure is $\xi$ (i.e., $\Lambda_{\xi}(z,z^\prime) = \lambda_{z,z^\prime}(\xi)$ when $(z,z^\prime) \in \E$, $\Lambda_{\xi}(z,z^\prime) = 0$ when $(z,z^\prime) \notin \E$, and $\Lambda_{\xi}(z,z) = -\sum_{z^\prime \neq z} \lambda_{z,z^\prime}(\xi)$), and $\Lambda^*_{\xi}$ denotes the transpose of $\Lambda_\xi$. The above dynamical system on $\MZ$ is called the McKean-Vlasov equation. This mean-field convergence allows one to view the process $\mu^N_{\nu_N}$ as a small random perturbation of the dynamical system~\eqref{eqn:mve}. The starting point of our study of the asymptotics of $\{\wp^N, N \geq 1\}$ is the process-level LDP for $\{\mu^N_{\nu_N}, \nu_N \in \MNZ, N \geq 1\}$, whenever $\nu_N$ converges to $\nu$ in  $\MZ$. This LDP was established by  L\'eonard~\cite{leonard-95} when the initial  conditions are fixed, and by Borkar and Sundaresan~\cite{borkar-sundaresan-12} when the initial conditions converge\footnote{Often, as done in~\cite{borkar-sundaresan-12}, one lets $\nu_N$ be random, and only requires $\nu_N \rightarrow \nu$ in distribution, where $\nu$ is deterministic. For simplicity, we restrict $\nu_N$ to be deterministic.} in $\MZ$. The rate function of this LDP is governed by ``costs" associated with trajectories on $[0,T]$ with initial condition $\nu$, which we denote by $S_{[0,T]}(\varphi | \nu)$, $\varphi \in \DMZ$ (see~\eqref{eqn:rate-function-finite-duration} for its definition).

We assume that $\xi^*$ is the unique globally asymptotically stable equilibrium of~\eqref{eqn:mve}. Define the Freidlin-Wentzell quasipotential
\begin{align}
V(\xi) \coloneqq \inf \{ S_{[0,T]}(\varphi|\xi^*) : \varphi(0) = \xi^*, \varphi(T) = \xi, T > 0 \}, \, \xi \in \MZ.
\label{eqn:v}
\end{align}
From the theory of large deviations of the invariant measure of Markov processes~\cite{freidlin-wentzell,sowers-90,cerrai-rockner-05,borkar-sundaresan-12}, $V$ is a natural candidate for the rate function of the family $\{\wp^N, N \geq 1\}$.
\subsection{Two counterexamples}
\label{subsection:counterexamples-introduction}
We begin with two counterexamples for which  $V$ is not the rate function for the family of invariant measures. 
\subsubsection{Non-interacting M/M/1 queues}
\label{subsection:counterexamples-introduction-mm1}
\setlength{\unitlength}{1mm}
\begin{figure}
\centering
\begin{picture}(90,30)(5,0)
\thicklines
\put(-20,15){\circle{9}}
\put(0,15){\circle{9}}
\put(20,15){\circle{9}}
\put(80,15){\circle{9}}
\put(-21,14){$0$}
\put(-1,14){$1$}
\put(19,14){$2$}
\put(79,14){$z$}

\qbezier(-17,18)(-10,25)(-3,18)
\put(-8,21.5){\vector(1,0){0.1}}
\qbezier(3,18)(10,25)(17,18)
\put(12,21.5){\vector(1,0){0.1}}
\qbezier(23,18)(30,25)(37,18)
\put(32,21.5){\vector(1,0){0.1}}
\qbezier(63,18)(70,25)(77,18)
\put(72,21.5){\vector(1,0){0.1}}
\qbezier(-17,12)(-10,5)(-3,12)
\put(-11,8.5){\vector(-1,0){0.1}}
\qbezier(3,12)(10,5)(17,12)
\put(9,8.5){\vector(-1,0){0.1}}
\qbezier(23,12)(30,5)(37,12)
\put(29,8.5){\vector(-1,0){0.1}}
\qbezier(63,12)(70,5)(77,12)
\put(69,8.5){\vector(-1,0){0.1}}

\qbezier(83,18)(90,25)(97,18)
\put(91,21.5){\vector(1,0){0.1}}

\qbezier(83,12)(90,5)(97,12)
\put(89,8.5){\vector(-1,0){0.1}}

\put(42,15){\circle*{1}}
\put(50,15){\circle*{1}}
\put(58,15){\circle*{1}}
\put(102,15){\circle*{1}}
\put(110,15){\circle*{1}}
\put(118,15){\circle*{1}}

\put(-11,23){$\lambda_f$}
\put(9,23){$\lambda_f$}
\put(29,23){$\lambda_f$}
\put(69,23){$\lambda_f$}
\put(89,23){$\lambda_f$}
\put(-13,5){$\lambda_b$}
\put(6,5){$\lambda_b$}
\put(27,5){$\lambda_b$}
\put(67,5){$\lambda_b$}
\put(87,5){$\lambda_b$}
\end{picture}
\caption{Transition rates of an M/M/1 queue}
\label{fig:mm1}
\end{figure}
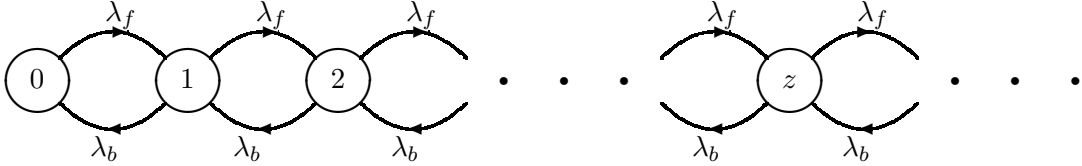
Consider the graph $(\Z, \EQ)$ whose edge set $\EQ$ consists of forward edges $\{(z,z+1), z \in \Z\}$ and backward edges  $\{(z,z-1), z \in \Z \setminus \{0\}\}$  (see Figure~\ref{fig:mm1}). Let $\lambda_f$ and $\lambda_b$ be two positive numbers. Consider the generator $L^Q$ acting on functions $f$ on $\Z$ by
\begin{align*}
L^Q f(z) \coloneqq \sum_{z^\prime:(z,z^\prime) \in \EQ} \lambda_{z,z^\prime}(f(z^\prime) - f(z)),\, z \in \Z,
\end{align*}
where $\lambda_{z,z+1} = \lambda_f$ for each $z \in \Z$ and $\lambda_{z,z-1} = \lambda_b$ for each $z \in \Z \setminus  \{0\}$. 
When $\lambda_f < \lambda_b$,  the invariant probability  measure associated with this  Markov process is
\begin{align*}
\xi^*_Q(z) \coloneqq \left(1-\frac{\lambda_f}{\lambda_b}\right) \left(\frac{\lambda_f}{\lambda_b} \right)^z,\, z \in \Z.
\end{align*}

For each $N \geq 1$, we consider $N$ particles, each of which evolves independently as a Markov process on $\Z$ with the infinitesimal generator $L^Q$. That is, the particles are independent M/M/1 queues. It is easy to check that the empirical measure of the system of particles is also a Markov process on the state space $\MNZ$ and it   possesses a unique invariant probability measure, which we denote by $\wp^N_Q$. 

On one hand, it is straightforward to see that the family $\{\wp^N_Q, N \geq 1\}$ satisfies the LDP on $\MZ$. Indeed, under stationarity, the state of each particle is distributed as $\xi^*_Q$. As a consequence, $\wp^N_Q$ is the law of the random variable
$\frac{1}{N}\sum_{n=1}^N \delta_{\zeta_n}$ on $\MZ$, where $\zeta_1, \ldots, \zeta_N$ are independent and identically distributed (i.i.d.) as $\xi^*_Q$. Therefore, by Sanov's theorem~\cite[Theorem~6.2.10]{dembo-zeitouni}, $\{\wp^N_Q, N \geq 1 \}$ satisfies the LDP with the rate function $I(\cdot \| \xi_Q^*)$,  where $I : \MZ \times \MZ \to [0, \infty]$ is the relative entropy defined by\footnote{We use the convention $0 \log 0 = 0$.}
\begin{align}
I(\zeta \| \nu) \coloneqq \left\{
\begin{aligned}
&\sum_{z \in \Z} \zeta(z)  \log\left(\frac{\zeta(z)}{\nu(z)}\right), &  \text{ if } \zeta \ll \nu, \\
&\infty, & \text{ otherwise.}
\end{aligned}
\right.
\label{eqn:entropy}
\end{align}
On the other hand, it is natural to conjecture that the rate function for the family $\{\wp^N_Q, N \geq 1\}$ is given by the quasipotential~\eqref{eqn:v} with $\xi^*$  replaced by $\xi_Q^*$. However, as discussed in the next paragraph,  the quasipotential is not the same as $I(\cdot \| \xi_Q^*)$. Hence, from the uniqueness of the large deviations rate function~\cite[Lemma~4.1.4]{dembo-zeitouni}, the quasipotential does not govern the rate function for  the family $\{\wp^N_Q, N \geq 1\}$.

We now provide some intuition on why the quasipotential is not the rate function in the example under consideration.  For a formal proof, see Section~\ref{section:counterexamples}.  We first introduce some notation. Let $\mathbb{R}^\infty$ denote the infinite product of $\R$ equipped with the product topology. We view  $\MZ$ as the subset $\{x \in \R^\infty: x_i \geq 0\, \forall i, \sum_{i \geq 0} x_i = 1\}$ of $\R^\infty$ with the subspace topology (e.g., see \cite[Chapter 3, Section 2]{durrett-13}). If $\xi, f \in \R^\infty$, we define
\begin{align}
\langle \xi, f \rangle \coloneqq \lim_{m \to \infty} \sum_{z = 0}^m \xi(z) f(z),
\label{eqn:inner-prod}
\end{align}
whenever the limit exists. Also, define $\vartheta: \Z \to \R_+$ by
\begin{align}
\vartheta(z) \coloneqq z \log z, \, z \in \Z,
\label{eqn:theta}
\end{align}
with the convention that $0 \log 0 = 0$, and define $\iota(z) \coloneqq z$, $z \in \Z$. Using the fact that $\xi^*_Q$ has geometric decay, it can be checked that $I(\xi \| \xi^*_Q)$ is finite if and only if the first moment of $\xi$ (i.e., $\langle \xi, \iota \rangle$) is finite. However it turns out that $V(\xi)$ (i.e., the quantity in~\eqref{eqn:v} with $\xi^*$ replaced by $\xi^*_Q$) is finite if and only if the $\vartheta$-moment of $\xi$ (i.e.,  $\langle \xi, \vartheta \rangle$) is finite. In particular, if we consider a $\xi \in \MZ$ whose first moment is finite but $\vartheta$-moment is infinite then $ V(\xi) \neq I(\xi \| \xi^*_Q)$. 
Let $\varepsilon > 0$, $\xi \in \MZ$ be such that $\langle \xi, \iota \rangle < \infty$ but $\langle \xi, \vartheta \rangle = \infty$, and consider the $\varepsilon$-neighbourhood of $\xi$ in $\MZ$. By Sanov's theorem, the  probability of this neighbourhood under $\wp^N_Q$ is of the form $\exp\{-N (I(\xi \| \xi^*_Q) + o(1))\}$. For a fixed $T > 0$, let us now try to estimate the probability of $\mu^N_{\nu_N}(T)$  being in this neighbourhood when $\nu_N$ is in a small neighbourhood of $\xi^*_Q$. If the process $\mu^N$ is initiated at a $\nu_N$ near $\xi^*_Q$, then the  probability that the random variable $\mu^N_{\nu_N}(T)$ is in the $\varepsilon$-neighbourhood of $\xi$ is at most 
\begin{align*}
\exp\left\{-N \left(\inf_{\{\xi^\prime:d(\xi,\xi^\prime) \leq \varepsilon\}} V(\xi^\prime)  +o(1)\right)\right\}.
\end{align*}
Since $V$ is lower semicontinuous (we prove this in Lemma~\ref{lemma:v-compactness-level-sets}), we must have
\begin{align*}
\inf_{\{\xi^\prime:d(\xi,\xi^\prime) \leq \varepsilon\}} V(\xi^\prime) \to \infty \text{ as } \varepsilon \to 0.
\end{align*}
Hence we can choose an  $\varepsilon$ small enough so that $\inf_{\{\xi^\prime:d(\xi,\xi^\prime) \leq \varepsilon\}} V(\xi^\prime) > 2 I(\xi \| \xi^*_Q)$. For this $\varepsilon$, the probability that $\mu^N_{\nu_N}(T)$ lies is the $\varepsilon$-neighbourhood of $\xi$ is upper bounded by $\exp\{-N \times (2I(\xi \| \xi^*_Q) + o(1))\}$, which is smaller than $\exp\{-N (I(\xi \| \xi^*_Q) + o(1))\}$, even in the exponential scale, for large enough $N$. That is, for any arbitrary but fixed $T$, we can find a small neighbourhood of $\xi$ such that the probability that $\mu^N_{\nu_N}(T)$ lies in that neighbourhood is smaller than  what we expect to see in the stationary regime. In other words, there are some barriers in $\MZ$ that cannot be surmounted in any finite time, yet these barriers can be crossed  in the stationary regime. These barriers indicate that, to obtain the correct stationary regime probability of a small neighbourhood of $\xi$ using the dynamics of $\mu^N_{\nu_N}$, one should wait longer than any fixed time horizon. That is, one should consider the random variable $\mu^N_{\nu_N}(T(N))$, where  $T(N)$ is a suitable function of $N$, and estimate the probability that $\mu^N_{\nu_N}(T(N))$ belongs to a small neighbourhood of $\xi$. However it is not straightforward to obtain such estimates from the process-level large deviation estimates of $\mu^N_{\nu_N}$ since the latter are usually available for a fixed time duration.

There are natural barriers in the context of finite-state mean-field models when the limiting dynamical system has multiple (but finitely many) stable equilibria~\cite{mypaper-1}. In such situations, passages from a neighbourhood of one equilibrium to a neighbourhood of another take place over time durations of the form $\exp\{N\times O(1)\}$ where $N$ is the number of particles\footnote{$O(1)$ refers to a bounded sequence, and $\omega(1)$ refers to a sequence that goes to $\infty$.}. Interestingly, these barriers can be surmounted using trajectories of finite time durations; i.e., for any fixed $T$, the probability that the empirical measure process reaches a neighbourhood of an  equilibrium at time $T$ when it is  initiated in a  small neighbourhood of another equilibrium  is of the form $\exp\{-N \times O(1)\}$. In contrast, in the case of the above counterexample, the barriers cannot be surmounted in finite time durations; for any fixed $T$, the probability that   $\mu^N(T)$  reaches a small neighbourhood of a point in $\MZ$ with finite first moment but infinite $\vartheta$-moment when it is initiated from a neighbourhood of $\xi^*_Q$ is of the form $\exp\{-N \times \omega(1)\}$.  Hence we anticipate that the barriers that we encounter in the above counterexample are somehow more difficult to surmount than those that arise in the case of  finite-state  mean-field models with  multiple stable equilibria. 
\subsubsection{Non-interacting nodes in a wireless network}
\label{subsection:counterexamples-introduction-wlan}
\setlength{\unitlength}{1mm}
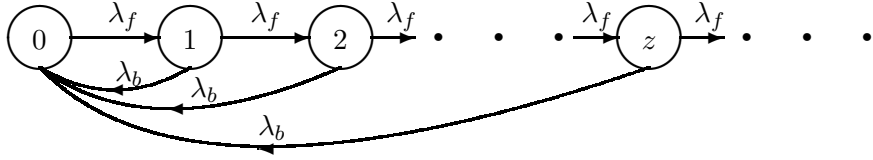
\begin{figure}
\centering
\begin{picture}(60,20)(0,0)
\thicklines
\put(-20,15){\circle{9}}
\put(0,15){\circle{9}}
\put(20,15){\circle{9}}
\put(61,15){\circle{9}}
\put(-21,13.5){$0$}
\put(-1,13.5){$1$}
\put(19,13.5){$2$}
\put(60,13.5){$z$}
\put(-16,15){\vector(1,0){12}}
\put(4,15){\vector(1,0){12}}
\put(24,15){\vector(1,0){6}}
\put(65,15){\vector(1,0){6}}
\put(51,15){\vector(1,0){6}}
\qbezier(0,11)(-10,5)(-20,11)
\qbezier(20,11)(-5,0)(-20,11)
\qbezier(61,11)(0,-10)(-20,11)
\put(-11,8){\vector(-1,0){0.1}}
\put(-3,5.5){\vector(-1,0){0.1}}
\put(9,0.2){\vector(-1,0){0.1}}
\put(33,15){\circle*{1}}
\put(41,15){\circle*{1}}
\put(49,15){\circle*{1}}
\put(74,15){\circle*{1}}
\put(82,15){\circle*{1}}
\put(90,15){\circle*{1}}

\put(-11,17){$\lambda_f$}
\put(8,17){$\lambda_f$}
\put(26,17){$\lambda_f$}
\put(52,17){$\lambda_f$}
\put(67,17){$\lambda_f$}
\put(-10,9.2){$\lambda_b$}
\put(0,7){$\lambda_b$}
\put(9,2){$\lambda_b$}
\end{picture}
\caption{Transition rates of a wireless node}
\label{fig:wlan}
\end{figure} 
We provide another counterexample where the issue is similar. Consider the graph $(\Z, \EW)$ whose edge set $\EW$ consists of forward edges $\{(z,z+1), z \in \Z\}$ and backward edges $\{(z,0), z \in \Z \setminus \{0\}\}$ (see Figure~\ref{fig:wlan}). Let $\lambda_f$ and $\lambda_b$ be positive numbers. Consider the generator $L^W$ acting on functions $f$ on $\Z$ by
\begin{align*}
L^W f(z) \coloneqq \sum_{z^\prime: (z,z^\prime) \in \EW}  \lambda_{z,z^\prime}(f(z^\prime) - f(z)), z \in \Z,
\end{align*}
where $\lambda_{z,z+1} = \lambda_f$ for each $z \in \Z$ and $\lambda_{z,0} = \lambda_b$ for each $z \in \Z \setminus \{0\}$.
The invariant probability measure associated with this Markov process is
\begin{align*}
\xi^*_W(z) \coloneqq \frac{\lambda_b}{\lambda_f + \lambda_b} \left(\frac{\lambda_f}{\lambda_f + \lambda_b} \right)^z,\, z \in \Z.
\end{align*} 

Similar to the previous example, for each $N \geq 1$, we consider $N$ particles, each of which evolves independently as a Markov process on $\Z$ with the infinitesimal generator $L^W$. It is easy to check that the empirical measure of the system of particles possesses a unique invariant probability measure, which we denote by $\wp_W^N$.  
Under stationarity, the state of each particle is distributed as $\xi^*_W$. As a consequence, $\wp^N_W$ is the law of the random variable $\frac{1}{N}\sum_{n=1}^N \delta_{\zeta_n}$ on $\MZ$,  where $\zeta_1, \ldots, \zeta_N$ are i.i.d. $\xi^*_W$. Hence, by Sanov's theorem, the family $\{\wp^N_W, N \geq 1\}$ satisfies the LDP with the rate function $I(\cdot \| \xi^*_W)$. As we show in Section~\ref{section:counterexamples}, in this example too, the quasipotential~\eqref{eqn:v} with $\xi^*$ replaced by $\xi_W^*$ is not the same as $I(\cdot \| \xi^*_W)$. As in the previous example, there are points $\xi$ where $V(\xi) = \infty$ but $I(\xi \| \xi_Q^*) < \infty$, points $\xi$ that have a finite first moment but infinite $\vartheta$-moment. Once again, the quasipotential does not govern the rate function for the family $\{\wp^N_W, N \geq 1\}$.
\subsection{Assumptions and main result}
We now provide some assumptions on the model of countable-state mean-field interacting particle systems  that ensure that the barriers in $\MZ$ that are insurmountable using trajectories of arbitrary but finite time duration remain insurmountable in the stationary regime as well. Under these assumptions, we prove the main result of this paper, i.e., the sequence  of invariant measures $\{\wp^N, N \geq 1\}$ satisfies the LDP with rate function $V$.
\subsubsection{Assumptions}
Our first set of assumptions is on the mean-field interacting particle system (i.e., on the generator $\LN$ defined in~\eqref{eqn:LN}).
\begin{enumerate}[label=({A\arabic*})]
\item The edge set is given by $\mathcal{E} = \{(z,z+1), z \in \Z \} \cup \{(z,0), z \in \Z \setminus \{0\}\}.$ \label{assm:a1}
\item There exist positive constants $\overline{\lambda}$ and $\underline{\lambda}$ such that
\begin{align*}
\frac{\underline{\lambda}}{z+1}\leq \lambda_{z,z+1}(\xi) \leq \frac{\overline{\lambda}}{z+1}, \text{ and }  \underline{\lambda} \leq \lambda_{z,0}(\xi) \leq \overline{\lambda},
\end{align*}
for all $\xi \in \MZ$. \label{assm:a2}
\item The functions $(z+1) \lambda_{z,z+1}(\cdot)$, $z \in \Z,$ and $\lambda_{z,0}(\cdot)$, $z \in \Z \setminus\{0\}$, are uniformly Lipschitz continuous on $\MZ$.  \label{assm:a3}
\end{enumerate}
Note that assumption~\ref{assm:a1} considers a specific transition graph  (Figure~\ref{fig:wlan}) for each particle. This graph arises  in the contexts of random backoff algorithms for  medium access in wireless local area networks~\cite{kumar-etal-06} and   decentralised control of loads in a smart grid~\cite{meyn-etal-15}. Assumption \ref{assm:a2} ensures that the forward transition rates at state $z$ decays as $1/z$. This key assumption cuts down the speed of outward excursions and  enables us to overcome the issue described in the counterexamples. To highlight this, consider a modified example of Section~\ref{subsection:counterexamples-introduction-wlan} where $\lambda_{z,z+1} = \lambda_f/(z+1)$, $z \in \Z$; the rest of the description remains the same.  Let $\tilde{\xi}_W \in \MZ$ denote the invariant probability measure associated with one particle. It can be checked  that $\tilde{\xi}_W(z)$ is of the order of $\exp\{-\vartheta(z)\}$, unlike $\xi^*_W$ which has geometric decay. As a consequence, $I(\xi \| \tilde{\xi}_W)$ is finite if and only if the $\vartheta$-moment of $\xi$ is finite. Hence, by imposing~\ref{assm:a2}, we have ensured that the barriers  in $\MZ$ that are insurmountable for finite time duration trajectories continue to remain  insurmountable in the stationary regime; this is the key property that enables us to prove the main result of this paper.  
Assumption \ref{assm:a3} is a uniform Lipschitz continuity property for the transition rates which is required for the process-level LDP for $\mu^N_{\nu_N}$ to hold and for the the McKean-Vlasov equation~\eqref{eqn:mve} to be well-posed.

Our second set of assumptions is on the McKean-Vlasov equation~\eqref{eqn:mve}. Let $\mu_\nu$, $\nu \in \MZ$, denote the solution  to the limiting dynamics~\eqref{eqn:mve} with initial condition $\nu \in \MZ$. Recall the function $\vartheta$. Define $\K_M \coloneqq \{\xi \in \MZ: \langle \xi, \vartheta \rangle \leq M\}$, $M > 0$.

\begin{enumerate}[label=({B\arabic*})]
\item There exists a unique globally asymptotically stable equilibrium $\xi^*$ for the McKean-Vlasov equation~\eqref{eqn:mve}. \label{assm:gase} 
\item $\langle \xi^*, \vartheta \rangle < \infty$ and  $ \lim_{t \to \infty} \sup_{\nu \in \K_M} \langle \mu_\nu(t), \vartheta \rangle =  \langle \xi^*, \vartheta \rangle$ for each $M > 0$. \label{assm:zlogz-convergence}
\end{enumerate}
The first assumption above asserts that all the  trajectories of~\eqref{eqn:mve} converge to $\xi^*$ as time becomes large. 
The proof of the LDP upper and lower bounds for the family $\{\wp^N, N \geq 1\}$ involves construction of trajectories that start at suitable compact sets, reach the stable equilibrium $\xi^*$ using arbitrarily small cost, and then terminate at a desired point in $\MZ$ starting from $\xi^*$. All these are enabled by assumption~\ref{assm:gase} (see more remarks about this assumption in Section~\ref{subsection:discussion}). The second assumption asserts  that the $\vartheta$-moment of the solution to the limiting dynamics converges uniformly over initial conditions lying in sets of bounded $\vartheta$-moment. In the case of a non-interacting system that satisfies \ref{assm:a1} but with constant forward transition rates (for example, see $L^W$ in Section~\ref{subsection:counterexamples-introduction-wlan}), the analogue of this assumption can easily be verified: the first moment of the solution to the limiting dynamics converges uniformly over initial conditions lying in  sets of bounded first moment. In fact, one can explicitly write down the first moment of the solution to the limiting dynamics in this case and verify this assumption easily. Assumption~\ref{assm:zlogz-convergence} is the analogous statement for our mean-field system that satisfies the $1/z$-decay of the forward transition rates in assumption~\ref{assm:a2}.

\subsubsection{Main result}
We now state the main result of this paper, namely the LDP for the family of invariant measures $\{\wp^N, N \geq 1\}$ under the assumptions \ref{assm:a1}--\ref{assm:a3} and \ref{assm:gase}--\ref{assm:zlogz-convergence}.

We first assert the existence and uniqueness of the invariant measure $\wp^N$ for $\LN$ for each $N \geq 1 $, and the exponential tightness of the family $\{\wp^N, N \geq 1\}$.
\begin{proposition}
Assume \ref{assm:a1} and \ref{assm:a2}. For each $N \geq 1$, $\LN$ admits a unique  invariant probability measure $\wp^N$. Further, the family $\{\wp^N, N \geq 1\}$ is exponentially tight in $\MZ$.
\label{prop:invariant-measure-exp-tightness}
\end{proposition}

Recall the quasipotential $V$ defined in~\eqref{eqn:v}. We now state the main result of this paper. 
\begin{theorem}
\label{thm:main-result}
Assume \ref{assm:a1}, \ref{assm:a2}, \ref{assm:a3}, \ref{assm:gase}, and \ref{assm:zlogz-convergence}. Then the family of probability measures $\{\wp^N, N \geq 1\}$ satisfies the large deviation principle on $\MZ$ with rate function $V$.
\end{theorem}

The proof of this result is carried out in Sections~\ref{section:ldp-lower-bound}--\ref{section:proof-completion}. We begin with the process-level uniform LDP for $\mu^N_{\nu_N}$ over compact subsets of $\MZ$; this uniform LDP gives us the large deviation estimates for the process $\mu^N_{\nu_N}$  uniformly over the initial conditions $\nu_N$ lying in a given compact set (see Definition~\ref{def:uniform-ldp} and Theorem~\ref{thm:uniform-ldp-mun}).  We prove the LDP for the family $\{\wp^N, N \geq 1\}$ by transferring this process-level uniform LDP for $\mu^N_{\nu_N}$ over compact subsets of $\MZ$  to the stationary regime. The proof of the LDP lower bound (in Section~\ref{section:ldp-lower-bound}) considers specific trajectories and lower bounds the probability of small neighbourhoods of points in $\MZ$ under $\wp^N$ using the probability that the process $\mu^N_{\nu_N}$ remains close to these trajectories. For the proof of the upper bound, we require certain regularity properties of the quasipotential. These properties are established in Section~\ref{section:v-properties}. We first show a controllability\footnote{This terminology is from Cerrai and R\"ockner~\cite{cerrai-rockner-05}.} property for $V$:  $V(\xi)$ is finite if and only if $\langle \xi, \vartheta \rangle < \infty$. Using the lower bound proved in Section~\ref{section:ldp-lower-bound}, we then show that the level sets of $V$ are compact subsets of $\MZ$.  Since  $\MZ$ is not locally compact and $V$ has compact lower level sets,  we do not expect $V$ to be continuous on $\MZ$. Indeed, if $\xi \in \MZ$ is such that $V$ is continuous at $\xi$ and $V(\xi) < \infty$, given $\varepsilon>0$  there exists a $\delta > 0$ such that $d(\xi^\prime, \xi) < \delta$ implies that $|V(\xi^\prime) - V(\xi)| < \varepsilon$. In particular, $\{\xi^\prime \in \MZ : V(\xi^\prime) \leq V(\xi) + \varepsilon\} \supset B(\xi, \delta)$.  Since $\{\xi^\prime \in \MZ : V(\xi^\prime) \leq V(\xi) + \varepsilon\}$ is compact in $\MZ$, this shows that $\xi$ has a relatively compact neighborhood in $\MZ$, which is a contradiction. This shows that, for any $\xi \in \MZ$ such that $V(\xi) < \infty$, $V$ is discontinuous at $\xi$. However we show the following  {\em small cost connection} property:  whenever $\xi_n \to  \xi^*$ in $\MZ$ and $\langle \xi_n, \vartheta \rangle \to \langle \xi^*, \vartheta \rangle$ as $n \to \infty$, we have $\lim_{n \to \infty}V(\xi_n) = V(\xi^*) = 0$. These properties of the quasipotential are then used to transfer the process-level uniform LDP upper bound for $\mu^N_{\nu_N}$ (uniform over compact subsets of $\MZ$) to the LDP upper bound for the family of invariant measures. The proof of the upper bound is carried out in Section~\ref{section:ldp-upper-bound}. Finally, we complete the proof of the theorem in Section~\ref{section:proof-completion}.

While the proofs of our lower and upper bounds follow the general methodology of Sowers~\cite{sowers-90}, there are significant model-specific difficulties that arise in our context. The main novelty in the proof of Theorem~\ref{thm:main-result} is to establish the small cost connection property  of the quasipotential $V$ under assumptions \ref{assm:a1}--\ref{assm:a3} and \ref{assm:gase}--\ref{assm:zlogz-convergence}. That is, we can find trajectories of small cost that start  at $\xi^*$ and end at points in $\MZ$ whose $\vartheta$-moment is not very far from that of $\xi^*$. In the work of Sowers~\cite{sowers-90}, this has been carried out by considering the ``straight-line" trajectory that connects the attractor to the nearby point under consideration. Such a trajectory may not have small cost in our case since the mass transfer is restricted to the edges in $\E$. We overcome this difficulty by considering a piecewise constant velocity mass transfer via the edges in $\E$. We then carefully estimate the cost of this trajectory and prove the necessary {\em small cost connection}  property. We also simplify the proof of the compactness of the  lower level sets of $V$; while Sowers~\cite[Proposition~7]{sowers-92} studies  the minimisation of the costs of trajectories over the infinite-horizon, we arrive at it by  using the LDP lower bound and the exponential tightness of the family $\{\wp^N, N \geq 1\}$.
We also remark that the methodology of Sowers~\cite{sowers-90} has been used by Cerrai and R\"ockner~\cite{cerrai-rockner-05} in the context of  stochastic  reaction diffusion equations and by Cerrai and Paskal~\cite{cerrai-paskal-20} in the context of two-dimensional  stochastic Navier-Stokes equations.
\subsection{Discussion and future directions}
\label{subsection:discussion}
The main result and the counterexamples suggest that in order for the family of invariant measures of a Markov process to satisfy the large derivation principle with rate function governed by the  Freidlin-Wentzell quasipotential, one must have some {\em good}  properties on the model under consideration. In the case of our main result, this goodness property was achieved by the $1/z$-decay of the forward transition rates from  assumption~\ref{assm:a2}. We use this assumption to show the  exponential tightness of the invariant measure over compact subsets with bounded $\vartheta$-moments. It also enables us to show the necessary regularity  properties of the quasipotential required to transfer the process-level large deviation result to the stationary regime. However a general treatment of the LDP for the family of  invariant measures of Markov processes (that encompasses the cases of~ \cite{sowers-90,cerrai-rockner-05,cerrai-paskal-20,borkar-sundaresan-12,farfan-etal-18}),  especially when the ambient state space is not locally compact, is missing in the literature.

One of the assumptions that plays a significant role in the proof of our main result is the existence of a unique globally asymptotically stable equilibrium for the limiting dynamics (assumption~\ref{assm:gase})\footnote{In the works of Sowers~\cite{sowers-90},  Cerrai and R\"ockner~\cite{cerrai-rockner-05}, and Cerrai and Paskal~\cite{cerrai-paskal-20}, their model assumptions ensure that~\ref{assm:gase} holds.}. In general, the limiting dynamical system~\eqref{eqn:mve} could possess multiple $\omega$-limit sets. In that case the approach of our proofs breaks down. A well-known approach to study large deviations of the invariant measures in such cases is to focus on  small neighbourhoods of these $\omega$-limit sets and then  analyse the  discrete time Markov chain that evolves on these neighbourhoods. The LDP then follows from the estimates of the invariant measure of this discrete time chain (see Freidlin and Wentzell~\cite[Chapter~6,~Section~4]{freidlin-wentzell}). However this approach requires the {\em uniform LDP over open subsets} of $\MZ$, which is not yet  available for our mean-field model.  If this can be established, along with the regularity properties of the quasipotential established in Section~\ref{section:v-properties}, one can not only use  the above  idea to extend our main result to the case when the limiting dynamical system possesses multiple $\omega$-limit sets but also to study  exit problems and metastability phenomena in our mean-field model.

Another definition of the quasipotential appears in the literature. It is given by the minimisation of costs  of the form $S_{(-\infty, 0]}(\varphi)$
over infinite-horizon trajectories $\varphi$ on $(-\infty, 0]$ such that the terminal time condition $\varphi(0)$ is fixed and $\varphi(t) \to \xi^*$ as $t \to -\infty$ (see Sowers~\cite{sowers-90}, Cerrai and R\"ockner~\cite{cerrai-rockner-05}). While it is clear that the above definition of the quasipotential is a lower bound for $V$ in~\eqref{eqn:v}, unlike in  Sowers~\cite{sowers-90} and  Cerrai and R\"ockner~\cite{cerrai-rockner-05}, we are not able to show that the two definitions are the same. A proof of this equality, or otherwise, will add more insight on the general case.

We remark that assumption~\ref{assm:a3} does not play a role in the proof of our main result. It is used to invoke the process-level LDP for $\mu^N_{\nu_N}$  (see Theorem~\ref{thm:uniform-ldp-mun}) and the well-posedness of the limiting dynamical system~\eqref{eqn:mve}. If these two properties are established through some other means then the proof of Theorem~\ref{thm:main-result} holds verbatim without the need for assumption~\ref{assm:a3}.

Finally, we mention that a time-independent variational formula for the quasipotential is available for some non-reversible models in statistical mechanics, see Bertini et al.~\cite{bertini-etal-02,bertini-etal-03}. It is not clear if the quasipotential $V$ in~\eqref{eqn:v} admits a time-independent variational form. This would be an interesting direction to explore.

\subsection{Related literature}
Process-level large deviations of small-noise diffusion processes have been well studied in the past. For finite-dimensional large deviation problems, see Freidlin and Wentzell~\cite[Chapter~5]{freidlin-wentzell}, Liptser~\cite{liptser-96}, Veretennikov~\cite{veretennikov-00}, Puhalskii~\cite{puhalskii-16}, and the references therein. For infinite-dimensional problems where the state space is not locally compact, see Sowers~\cite{sowers-92} and  Cerrai and R\"ockner~\cite{cerrai-rockner-04}. More recently, uniform large deviation principle (uniform LDP) for Banach-space valued stochastic differential equations over the class of bounded and open subsets of the Banach space have been studied by Salins~et~al.~\cite{salins-etal-19}. These have been used to study the exit times and metastability in such processes, see Salins and Spiliopoulos~\cite{salins-spiliopoulos-19}. While the above works focus on diffusion processes, our work focuses on the stationary regime large deviations of  countable-state mean-field models with jumps. In the spirit of the small-noise problems listed above, our process $\mu^N_{\nu_N}$ can be viewed as  a small random perturbation of the dynamical system~\eqref{eqn:mve} on $\MZ$.

In the context of interacting particle systems, Dawson and G\"artner~\cite{dawson-gartner-87} established the process-level LDP for weakly interacting diffusion processes, and L\'eonard~\cite{leonard-95} and Borkar and Sundaresan~\cite{borkar-sundaresan-12} extended this to mean-field interacting particle systems with jumps. In this work, we focus on the stationary regime large deviations of mean-field models with jumps when the state of each particle comes from a countable set. For small-noise diffusion process on Euclidean spaces and finite-state mean-field models, since the state space (on which the empirical measure process evolves) is locally compact, the process-level large deviation results have been  extended in a straightforward manner to the uniform LDP over the class of open subsets of the space. Such uniform large deviation estimates have been used to prove the large deviations of the invariant measure and the exit time estimates, see Freidlin and Wentzell~\cite[Chapter~6]{freidlin-wentzell} in the context of diffusion processes, Borkar and Sundaresan~\cite{borkar-sundaresan-12} and~\cite{mypaper-1} in the context of finite-state mean-field models. One of the key ingredients in these proofs is the continuity of the  quasipotential. However in our case, the state space $\MZ$ is infinite-dimensional and not locally compact. Therefore, since the quasipotential \eqref{eqn:v} is expected to have compact lower level sets, we do not expect it to be continuous on $\MZ$ unlike in the finite-dimensional problems mentioned above. Hence the ideas presented in~\cite{borkar-sundaresan-12} are not directly applicable to our context of the LDP for the family of invariant measures. 

Large deviations of the family of invariant measures for small-noise diffusion processes on non-locally compact spaces have also been studied in the past, see Sowers~\cite{sowers-90} and Cerrai and R\"ockner~\cite{cerrai-rockner-05}. They have a unique attractor for the limiting dynamics, and the proof essentially involves conversion of the uniform LDP over the finite-time horizon to the stationary regime. Martirosyan~\cite{martirosyan-17} studied a situation where the limiting dynamical system possesses multiple attractors. For the study of large deviations of the family of invariant measures for simple exclusion processes, see Bodineau and Giacomin~\cite{bodineau-giacomin-04} and Bertini et al.~\cite{bertini-etal-03}. More recently, Farf\'an et al.~\cite{farfan-etal-18} extended this to a simple exclusion process whose limiting hydrodynamic equation has multiple attractors. Their proof proceeds similar to the case of finite-dimensional diffusions in Freidlin and Wentzell~\cite[Chapter~6,~Section~4]{freidlin-wentzell} by first approximating the process near the attractors and then using the Khasminskii reconstruction formula~\cite[Chapter~4,~Section~4]{khasminskii}. In particular, it requires the uniform LDP to hold over open subsets of the state space. Since their state space, although infinite-dimensional, is compact, the proof of the uniform LDP over open subsets easily follows  from the process-level LDP. Also, the compactness of the state space simplifies the proofs of the {\em small cost connection} property from the attractors to nearby points, a property needed in the Khasminskii reconstruction. Although we restrict our attention to the case of a unique globally asymptotically stable equilibrium as in~\cite{sowers-90,cerrai-rockner-05}, the main novelty of our work is that we establish certain  regularity properties of the quasipotential for countable-state mean-field models with jumps  which were not done in the past. We then use these properties to prove the LDP for the family of invariant measures. Furthermore, we demonstrate two counterexamples where the stationary regime LDP's rate functions are not governed by the usual quasipotential. To the best of our knowledge, such examples where the LDP for the family of invariant measures hold but there rate functions are not governed by  the usual Freidlin-Wentzell quasipotential are new. These examples are constructed in a way that the particle systems do not possess the {\em small cost connection} property from the attractor to nearby points with finite first moment but infinite $\vartheta$-moment.

Large deviations of the family of invariant measure for a queueing network in a finite-dimensional setting has been studied by Puhalskii~\cite{puhalskii-19}. Finally, large deviations of the family of invariant measures for  a stochastic process  under some general conditions has been studied by Puhalskii~\cite{puhalskii-20}. One of their conditions is the small cost connection property between any two nearby points in the state space, which we do not expect to be true in our countable-state mean-field model since our state space is infinite-dimensional.
\subsection{Organisation}
This paper is organised as follows. In Section~\ref{section:prelims-results}, we provide preliminary results on the large deviations over finite time horizons. The proof of the main result is carried out in Sections~\ref{section:exp-tightness}--\ref{section:proof-completion}. In Section~\ref{section:exp-tightness}, we prove the existence, uniqueness, and exponential tightness of the family of invariant measures. In Section~\ref{section:ldp-lower-bound}, we prove the LDP lower bound for the family of invariant measures. In Section~\ref{section:v-properties}, we establish some regularity properties of the quasipotential $V$ defined in~\eqref{eqn:v}. In Section~\ref{section:ldp-upper-bound}, we prove the LDP upper bound for the family of invariant measures. In Section~\ref{section:proof-completion}, we complete the proof of the main result. Finally in Section~\ref{section:counterexamples}, we prove that the quasipotential differs from the relative entropy (with respect to the globally asymptotically stable equilibrium) for the two counterexamples discussed in Section~\ref{subsection:counterexamples-introduction}.
\section{Preliminaries}
\label{section:prelims-results}
\subsection{Frequently used notation}
\label{subsection:notation}
We first summarise the frequently used notation in the paper. Let $\Z$ denote the set of nonnegative integers and let $(\Z, \E)$ denote a directed graph on $\Z$. Let $\R^\infty$ denote the infinite product of $\R$ equipped with the topology of pointwise convergence. Let $\CBZ$ denote the space of functions on $\Z$ with compact support. Recall that $\MZ$ denotes the space of probability measure on $\Z$ equipped with the total variation metric (denoted by $d$). This metric generates the topology of weak convergence on $\MZ$. By Scheff\'e's lemma~\cite[Chapter~3,~Section~2]{durrett-13}, $\MZ$ can be  identified with the subset $\{x \in \R^\infty: x_i \geq 0\, \forall i, \sum_{i \geq 0} x_i = 1\}$ of $\R^\infty$ with the subspace topology. For each $N \geq 1$, recall that $\MNZ \subset \MZ$ denotes the space of probability measures on $\Z$ that can arise as empirical measures of $N$-particle configurations on $\Z^N$.  Recall $\vartheta$ defined in \eqref{eqn:theta}. Given $\alpha \in \CBZ$ and $g \in \R^\infty$, let the bracket $\langle \alpha, g \rangle$ denote $\sum_{z \in \Z} \alpha(z) g(z)$. Similarly, given $f, g \in\R^\infty$, let the bracket $\langle f, g \rangle$ denote $\lim_{n \to \infty}\sum_{k=0}^n f(k) g(k)$, whenever the limit exists. For $M > 0$, define  
\begin{align}
\K_M \coloneqq \left\{\xi \in \MZ: \langle \xi,  \vartheta \rangle \leq M\right\};
\label{eqn:KM}
\end{align}
by Prohorov's theorem, $\K_M$ is a compact subset of $\MZ$. Define $ \K \coloneqq \bigcup_{M \geq 1} \K_M$. Let $\xi^*  \in \MZ$ denote the globally asymptotically stable equilibrium for the McKean-Vlasov equation~\eqref{eqn:mve} (see assumption~\ref{assm:gase}). For each $\Delta > 0$, define
\begin{align}
K(\Delta) \coloneqq \{\xi \in \K : d(\xi^*, \xi) \leq \Delta \text{ and } |\langle  \xi^* , \vartheta \rangle - \langle \xi, \vartheta \rangle | \leq \Delta\};
\label{eqn:KD}
\end{align}
note that $K(\Delta)$ depends on $\xi^*$ as well (which we do not indicate for ease of readability). Define
\begin{align}
\tau(u) \coloneqq e^u-u-1,\, u \in \R.
\label{eqn:tau}
\end{align}
Note that $\tau$ is the log-moment generating function of the centred unit rate Poisson law, and define its convex dual
\begin{align}
\tau^*(u) \coloneqq \left\{
\begin{array}{ll}
\infty & \text{ if } u < -1, \\
1 & \text{ if } u = -1, \\
(u+1) \log (u+1) - u & \text{ if } u > -1,
\end{array}
\right.
\label{eqn:taustar}
\end{align}

For a complete and separable metric space $(\Space, d_0)$, $A \subset \Space$, and $x \in \Space$, let $d_0(x,A)$ denote $\inf_{y \in A} d_0(x,y)$.  For a set $A$ let $\sim \hspace{-0.4em} A$ denote the complement of $A$. For two numbers $a$ and $b$, let $a \vee b$ (resp. $a \wedge b$) denote maximum (resp. minimum)  of $a$ and $b$. Also, let $a^+ = \max\{a, 0\}$. For a metric space $\Space$, let $\B(\Space)$ denote the Borel $\sigma$-field on $\Space$. Finally, constants are denoted by $C$ and their values may be different in each occurrence.

\subsubsection{Notation related to the dynamics}
\label{subsection:dynamics-notation}
Let  $D([0,T], \Space)$ denote the space of $\Space$-valued functions on $[0,T]$ that are right continuous with left limits. It is equipped with the Skorohod topology which makes it a complete and separable metric space (see, for example, Ethier and Kurtz~\cite[Chapter~3]{ethier-kurtz}). Let $\rho$ denote a metric  on $D([0,T], \Space)$ that generates the Skorohod topology. An element of $D([0,T], \Space)$ is called a ``trajectory", and we shall refer to the process-level large deviations rate function evaluated on a trajectory as the ``cost" associated with that trajectory. For a trajectory $\varphi$, let both $\varphi_t$ and $\varphi(t)$ denote the evaluation of $\varphi$ at time $t$. For $N \geq 1$ and $\nu_N \in \MNZ$, let $\PN_{\nu_N}$ denote the solution to the $D([0,T], \MNZ)$-valued martingale problem for $\LN$ with initial condition $\nu_N \in \MNZ$ (whenever the martingale problem for $\LN$ is well-posed). Let $\mu^N_{\nu_N}$ denote the random element of $D([0,T], \MNZ)$ whose law is $\PN_{\nu_N}$. For each  $\xi \in \MZ$, let $L_\xi$ denote the generator acting on functions $f$ on $\Z$ by 
\begin{align*}
f \mapsto L_\xi(z) \coloneqq \sum_{z^\prime: (z,z^\prime) \in \E} \lambda_{z,z^\prime}(\xi) (f(z^\prime) - f(z)),\, z \in \Z,
\end{align*} 
i.e., the  generator of the single particle evolving on $\Z$ under the static mean-field $\xi$.

Let $C_0^1([0,T] \times \Z)$ denote the space of real-valued functions on $[0,T] \times \Z$ with compact support that are continuously differentiable in the first argument. Given a trajectory $\varphi \in \DMZ$ such that the mapping $[0,T] \ni t \mapsto \varphi_t \in \MZ$ is  absolutely continuous (see Dawson and G\"artner~\cite[Section~4.1]{dawson-gartner-87}), one can define $\dot{\varphi}_t \in \R^\infty$ for almost all $t \in [0,T]$ such that 
\begin{align*}
\langle \varphi_t, f_t \rangle = \langle \varphi_0, f_0 \rangle  + \int_{[0,t]} \langle \dot{\varphi}_u, f_u \rangle du+ \int_{[0,t]} \langle \varphi_u, \partial_u f_u \rangle du
\end{align*}
holds for each $f \in C_0^1([0,T]\times \Z)$ and $t \in [0,T]$.

Finally, let $\M_1(\DZ)$  denote the space of probability measures on $\DZ$  equipped with the usual weak topology. Also, let $\M_1(\M_1(\DZ))$ denote the space of probability measures on $\M_1(\DZ)$ equipped with the weak topology.

\subsection{Process-level large deviations}
We first recall the definition of the large deviation principle for a family of random variables indexed by one parameter.
\begin{definition}[Large deviation principle]
\label{def:ldp} 
Let $(\Space, d_0)$ be a metric space. We say that a family $\{X^N, N \geq 1\}$ of $\Space$-valued random variables defined on a probability space $(\Omega, \mathcal{F}, P)$ satisfies the large deviation principle  with rate function $I : S \to [0,\infty]$ if
\begin{itemize}
\item (Compactness of level sets). For any $s\geq 0$, $\Phi(s) \coloneqq \{x \in \Space: I(x) \leq s \}$ is a compact subset of $\Space$;
\item (LDP lower bound). For any $\gamma > 0$, $\delta > 0$, and $x \in \Space$, there exists $N_0 \geq 1$ such that
\begin{align*}
P(d_0(X^N, x) < \delta) \geq \exp\{-N(I(x) + \gamma)\}
\end{align*}
for any $N \geq N_0$;
\item (LDP upper bound). For any $\gamma > 0$, $\delta > 0$, and  $s >0$, there exists $N_0 \geq 1$ such that
\begin{align*}
P(d_0(X^N, \Phi(s)) \geq \delta) \leq \exp\{-N(s - \gamma)\}
\end{align*}
for any $N \geq N_0$.
\end{itemize}
\end{definition}
\noindent This definition is also used to study the large deviations of a family of probability measures. For each $N \geq 1$, let $P^N = P \circ (X^N)^{-1}$, the law of the random variable $X_N$ on $(\Space, d_0)$. We say that the family of probability measures $\{P^N, N \geq 1\}$ satisfies the LDP on $(\Space, d_0)$ with rate function $I$ if the sequence of $\Space$-valued random variables $\{X^N, N \geq 1\}$ satisfies the LDP with rate function $I$.

The LDP lower bound in the above definition is equivalent to the following statement~\cite[Chapter~3,~Section~3]{freidlin-wentzell}
\begin{align*}
\liminf_{N \to \infty} \frac{1}{N} \log P(X^N \in G) \geq -\inf_{x \in G} I(x), \text{ for all } G \subset S \text{ open}.
\end{align*} 
Similarly, under the compactness of the level sets of the  rate function $I$, the LDP upper bound above is equivalent to the following statement:
\begin{align*}
\limsup_{N \to \infty} \frac{1}{N} \log P(X^N \in F) \leq -\inf_{x \in F} I(x), \text{ for all } F \subset S \text{ closed}.
\end{align*}

To study the LDP for the family of invariant measures, we require estimates on the  probabilities of the process-level large deviations of $\mu^N_{\nu_N}$. In particular, we consider hitting times of $\mu^N_{\nu_N}$ on certain subsets of the state space $\MZ$ and apply the  process-level large deviation lower and upper bounds for $\mu^N_{\nu_N}$ starting at these subsets. Therefore, in addition to the scaling parameter $N$, we must  consider the process $\mu^N_{\nu_N}$   indexed by the initial condition $\nu_N \in \MNZ$. To study the process-level large deviations of such stochastic processes indexed by two parameters, we use the following  definition of  the uniform large deviation principle (see Freidlin and Wentzell~\cite[Chapter~3,~Section~3]{freidlin-wentzell}).
\begin{definition}[Uniform large deviation principle] We say that the family $\{\mu^N_{\nu_N}, \nu_N \in \MNZ, N \geq 1\}$ of $D([0,T],\MZ)$-valued random variables  defined on a probability space $(\Omega, \mathcal{F}, P)$ satisfies the uniform large derivation principle over the class $\mathcal{A}$ of subsets of $\MNZ$ with the family of rate functions $\{I_\nu, \nu \in \MZ\}$, $I_\nu : D([0,T],\MZ) \to [0,+\infty]$, $\nu \in \MZ$, if
\begin{itemize}
\item (Compactness of level sets). For each $K \subset \MZ$ compact and $s \geq 0$, $\bigcup_{\nu \in K}\Phi_\nu(s)$ is a compact subset of $D([0,T],\MZ)$, where $\Phi_\nu(s) \coloneqq \{\varphi \in D([0,T],\MZ): \varphi(0) = \nu, I_\nu(\varphi) \leq s\}$;
\item (Uniform LDP lower bound). For any $\gamma > 0$, $\delta > 0$, $s >0$, and $A \in \mathcal{A}$, there exists $N_0 \geq 1$ such that
\begin{align*}
P(\rho(\mu^N_{\nu_N}, \varphi) < \delta) \geq \exp\{-N(I_{\nu_N}(\varphi)+\gamma)\},
\end{align*}
for all $\nu_N \in A \cap \MNZ$, $\varphi \in \Phi_{\nu_N}(s)$, and $N \geq N_0$;
\item (Uniform LDP upper bound). For any $\gamma > 0$, $\delta > 0$, $s_0 >0$, and $A \in \mathcal{A}$, there exists $N_0 \geq 1$ such that
\begin{align*}
P(\rho(\mu^N_{\nu_N}, \Phi_{\nu_N}(s))\geq \delta) \leq \exp\{-N(s-\gamma)\},
\end{align*}
for all $\nu_N \in A \cap \MNZ$, $s \leq s_0$, and $N \geq N_0$.
\end{itemize}
\label{def:uniform-ldp}
\end{definition}
\noindent Note that the initial conditions in the upper and lower bounds lie in $A \cap \MNZ$, unlike in the definition in \cite[Chapter 3, Section 3]{freidlin-wentzell}.

We now make some definitions. Recall $\tau$ defined in \eqref{eqn:tau}. For each $\nu \in \MZ$ and $T > 0$, define the functional $S_{[0,T]}(\cdot | \nu): \DMZ \to [0,\infty]$ by 
\begin{align}
S_{[0,T]}(\varphi|\nu) \coloneqq \int_{[0,T]} \sup_{ \alpha \in \CBZ} \biggr\{ \langle \alpha,\dot{\varphi}_t - \Lambda_{\varphi_t}^*\varphi_t \rangle- \sum_{(z,z^\prime) \in \mathcal{E}} \tau(\alpha(z^\prime) - \alpha(z)) \lambda_{z,z^\prime}(\varphi_t) \varphi_t(z) \biggr\} dt,
\label{eqn:rate-function-finite-duration}
\end{align}
whenever $\varphi(0) = \nu$ and the mapping $[0,T]\ni t \mapsto \varphi(t) \in \MZ$ is absolutely continuous;  $S_{[0,T]}(\varphi | \nu) = \infty$ otherwise. Define the lower level sets of the functional $S_{[0,T]}(\cdot | \nu)$ by
\begin{align*}
\Phi_\nu^{[0,T]}(s) \coloneqq \{\varphi \in \DMZ: \varphi(0) = \nu, S_{[0,T]}(\varphi|\nu) \leq s\},\, s > 0, \, \nu \in \MZ.
\end{align*}
The next lemma asserts that these level sets are compact in $\DMZ$ when the initial conditions belong to a compact subset of $\MZ$. The proof is deferred to Appendix~\ref{appendix:proofs}.
\begin{lemma}
\label{lemma:compactness-S}
For each $T > 0$, $s > 0$, and $K \subset \MZ$ compact,
\begin{align*}
\{\varphi \in \DMZ : \varphi(0) \in K, S_{[0,T]}(\varphi | \varphi(0)) \leq s \}
\end{align*} 
is a compact subset of $\DMZ$.
\end{lemma}

The starting point of our study of the invariant measure asymptotics is the following uniform large deviation principle for the family $\{\mu^N_{\nu_N}, \nu_N \in \MNZ, N \geq 1\}$ over the class of compact subsets of $\MZ$ with the family of rate functions $\{S_{[0,T]}(\cdot | \nu),\nu \in \MZ\}$. Its proof  uses the process-level LDP for $\mu^N_{\nu_N}$ studied in L\'eonard~\cite{leonard-95} for a fixed initial condition and its extension (when $\Z$ is a finite set) to the case when initial conditions converge to a point in $\MZ$ in Borkar and Sundaresan~\cite{borkar-sundaresan-12}. The proof can be found in Appendix~\ref{appendix:proofs}.
\begin{theorem}
\label{thm:uniform-ldp-mun}
Fix \, $T > 0$ and assume \ref{assm:a1}, \ref{assm:a2}, and \ref{assm:a3}. Then the family of $\DMZ$-valued random variables $\{\mu^N_{\nu_N}, \nu_N \in \MNZ, N \geq 1\}$ satisfies the uniform large deviation principle over the class of compact subsets of $\MZ$ with the family of rate functions $\{S_{[0,T]}(\cdot|\nu), \nu \in \MZ\}$.
\end{theorem}
\noindent The rate function $S_{[0,T]}(\cdot | \nu)$ admits a non-variational representation in terms of a minimal cost ``control" that modulates the transition rates across various edges in $\E$ so that the desired trajectory is obtained.  Recall $\tau^*$ defined in \eqref{eqn:taustar}. 
\begin{theorem}[Non-variational representation; L\'eonard~\cite{leonard-95-1}]
Let $\varphi \in \DMZ$ be such that $S_{[0,T]}(\varphi|\varphi(0)) < \infty$. Then there exists a measurable function $h_\varphi : [0,T] \times \E \to \R$    such that 
\begin{align}
\langle \varphi_t, f_t \rangle & = \langle \varphi_0, f_0 \rangle + \int_{[0,t]} \langle  \varphi_u, \partial_u f_u \rangle du \nonumber \\
& \qquad + \int_{[0,t]} \sum_{(z,z^\prime)\in \E} (f_u(z^\prime) - f_u(z)) (1+h_\varphi(u,z,z^\prime)) \lambda_{z,z^\prime}(\varphi_u) \varphi_u(z) du
\label{eqn:weak-equation}
\end{align}
holds for all $t \in [0,T]$ and all $f \in C_0^1([0,T] \times \Z)$, and $S_{[0,T]}(\varphi|\varphi(0))$ admits the non-variational representation
\begin{align}
S_{[0,T]}(\varphi|\varphi(0)) = \int_{[0,T]} \sum_{(z,z^\prime)\in \E} \tau^*(h_\varphi(t,z,z^\prime)) \lambda_{z,z^\prime}(\varphi_t) \varphi_t(z) dt.
\label{eqn:rate-function-nonvar}
\end{align}
\label{thm:weak-equation}
\end{theorem}

\begin{remark}
It can be shown that the rate function $S_{[0,T]}$ defined in~\eqref{eqn:rate-function-finite-duration} can also be expressed as
\begin{align}
S_{[0,T]}(\varphi|\nu) & = \sup_{f \in C_0^1([0,T]\times\Z)} \Biggr\{\langle \varphi_T , f_T \rangle  - \langle \varphi_0 , f_0 \rangle - \int_{[0,T]}\langle \varphi_u , \partial_u f_u \rangle du \nonumber \\
& \qquad  - \int_{[0,T]}\langle \varphi_u, L_{\varphi_u} f_u \rangle du - \int_{[0,T]} \sum_{(z,z^\prime) \in \E}\tau(f_u(z^\prime) - f_u(z)) \lambda_{z,z^\prime}(\varphi_u) \varphi_u(z) du \Biggr\},
\label{eqn:rate-function-finite-duration-counter-example}
\end{align}
$\varphi \in \DMZ$, see L\'eonard~\cite{leonard-95-1}. This form of the rate function will indeed be used in the proof of the counterexamples in Section~\ref{section:counterexamples}.
\end{remark}

\section{Invariant measure: Existence, uniqueness, and exponential tightness}
\label{section:exp-tightness}
In this section we prove Proposition~\ref{prop:invariant-measure-exp-tightness}, the existence and uniqueness of the invariant measure $\wp^N$ for $\LN$ for each $N \geq 1$, and the exponential tightness of the family of invariant measures $\{\wp^N, N \geq 1\}$. The proof relies on the standard Krylov-Bogolyubov argument and a coupling between the interacting particle system under consideration and a non-interacting system with maximal forward transition rates minimal backward transition rates.

We first introduce some notations for the non-interacting particle system. Let $\bar{L}$ denote the generator acting on functions $f$ on $\Z$ by
\begin{align}
\bar{L}f(z) = \sum_{z^\prime: (z,z^\prime) \in \E} \lambda_{z,z^\prime} (f(z^\prime) - f(z)), \, z \in \Z,
\label{eqn:Lbar}
\end{align}
where $\lambda_{z,z+1} = \overline{\lambda}/(z+1)$ and $\lambda_{z,0} = \underline{\lambda}$. For each $z \in \Z$, let $\bar{P}_z$ denote the solution to the $D([0,T], \Z)$-valued martingale problem for $\bar{L}$ with initial condition $z$. Integration with respect to $\bar{P}_z$ is denoted by $\bar{E}_z$. Let $\pi \in \MZ$ denote the unique  invariant probability measure for $\bar{L}$. Let $\bar{P}_\pi$ denote the solution to the martingale problem for $\bar{L}$ with initial law $\pi$. Integration with respect to $\bar{P}_\pi$ is denoted by $\bar{E}_\pi$. By solving the detailed balance equations for $\bar{L}$, we see that 
\begin{align*}
\pi(z) \leq \pi(0) \left(\frac{\overline{\lambda}}{\underline{\lambda}}\right)^z \prod_{k=1}^z \frac{1}{k}, \, \,  z \geq 1.
\end{align*}
In particular, $\pi(z)$ has superexponential decay in $z$, and $\bar{E}_{\pi}(\exp\{\beta \vartheta(X)\}) < \infty$ for small enough $\beta>0$, where $\vartheta$ is defined in \eqref{eqn:theta}. Finally, for each $N \geq 1$, let $\PbarN_{\nu_N}$ denote the solution to the $D([0,T], \MNZ)$-valued martingale problem for $\LN$ with initial condition $\nu_N \in \MNZ$,  $\lambda_{z,z+1}(\zeta)$ replaced by $\overline{\lambda}/(z+1)$ and $\lambda_{z,0}(\zeta)$ replaced by $\underline{\lambda}$ in~\eqref{eqn:LN}, respectively, for each $\zeta \in \MZ$. Integration with respect to $\PbarN_{\nu_N}$ is denoted by $\EbarN_{\nu_N}$. Also, recall $\PN_{\nu_N}$, $\nu_N \in \MNZ$,  from Section \ref{subsection:dynamics-notation}. We are now ready to prove Proposition~\ref{prop:invariant-measure-exp-tightness}.
\begin{proof}[Proof of Proposition~\ref{prop:invariant-measure-exp-tightness}]
Fix $N \geq 1$. We first show the existence and uniqueness of the invariant probability measure for $\LN$. Consider the family of probability measures $\{\eta^N_T, T \geq 1\}$ on $\MZ$ defined by
\begin{align*}
\eta^N_T(A) \coloneqq \frac{1}{T}\int_0^T \PN_{\delta_0}(\mu^N(t) \in A) dt, \, A \in \B(\MZ), \, T \geq 1.
\end{align*}
Let $X_n^N(t)$ denote the state of the $n$th particle at time $t$. Recall the compact sets $\K_M$, $M > 0$,  defined in \eqref{eqn:KM}.
We first couple the laws $\PN_{\delta_0}$ and $\PbarN_{\delta_0}$. For $\mathbf{z}^N \in \Z^N$, define  $\emp(\mathbf{z}^N):= \frac{1}{N} \sum_{n=1}^N\delta_{z_n^N} \in \MNZ$. Let $\mathbf{e}_n^N$ denote the $N$-length vector with a $1$ in the $n$th position and $0$ everywhere else. Consider the Markov process on $\Z^N  \times \Z^N$ with the infinitesimal generator acting on functions $f$ on $\Z^N \times \Z^N$ by
\begin{align*}
(\mathbf{z}^N, \bar{\mathbf{z}}^N) \mapsto &  \sum_{n=1}^N \biggr[  \left(f(\mathbf{z}^N + \mathbf{e}_n^N, \bar{\mathbf{z}}^N+\mathbf{e}_n^N) - f(\mathbf{z}^N, \bar{\mathbf{z}}^N)\right) \left( \lambda_{z_n^N, z_n^N+1}(\emp(\mathbf{z}^N)) \wedge   \frac{\overline{\lambda}}{\bar{z}_n^N + 1} \right) \\
& \qquad + \left(f(\mathbf{z}^N + \mathbf{e}_n^N, \bar{\mathbf{z}}^N) - f(\mathbf{z}^N, \bar{\mathbf{z}}^N)\right) \left( \lambda_{z_n^N, z_n^N+1}(\emp(\mathbf{z}^N)) -    \frac{\overline{\lambda}}{\bar{z}_n^N + 1} \right)^+ \\
& \qquad  + \left(f(\mathbf{z}^N, \bar{\mathbf{z}}^N+\mathbf{e}_n^N) - f(\mathbf{z}^N, \bar{\mathbf{z}}^N)\right) \left( \frac{\overline{\lambda}}{\bar{z}_n^N + 1} - \lambda_{z_n^N, z_n^N+1}(\emp(\mathbf{z}^N)) \right)^+ \\
& \qquad + \left(f(\mathbf{z}^N - z^N_n \mathbf{e}_n^N, \bar{\mathbf{z}}^N - \bar{z}_n^N \mathbf{e}_n^N) - f(\mathbf{z}^N, \bar{\mathbf{z}}^N)\right) \left( \lambda_{z_n^N, 0}(\emp(\mathbf{z}^N)) \wedge   \underline{\lambda} \right) \indf{\{z_n^N > 0, \bar{z}_n^N > 0\}} \\
& \qquad + \left(f(\mathbf{z}^N - z^N_n \mathbf{e}_n^N, \bar{\mathbf{z}}^N) - f(\mathbf{z}^N, \bar{\mathbf{z}}^N)\right)  \left( \lambda_{z_n^N, 0}(\emp(\mathbf{z}^N)) -    \underline{\lambda} \right)^+ \indf{\{z_n^N > 0\}} \\
& \qquad  + \left(f(\mathbf{z}^N, \bar{\mathbf{z}}^N - \bar{z}^N_n\mathbf{e}_n^N) - f(\mathbf{z}^N, \bar{\mathbf{z}}^N)\right) \left( \underline{\lambda} - \lambda_{z_n^N,0}(\emp(\mathbf{z}^N)) \right)^+ \indf{\{\bar{z}_n^N > 0\}} \biggr].
\end{align*}
Such couplings were studied for continuous-time Markov chains, see, e.g., \cite{mufa1994optimal}. Note that, under the above Markov process, for any two initial conditions $\nu_N, \bar{\nu}_N \in \MNZ$,  the empirical measure flow associated with the first (resp. second) marginal has law $\PN_{\nu_N}$ (resp. $\PbarN_{\bar{\nu}_N}$). Therefore, for any $t >1$, $M > 1$, and $\beta > 0$, we have 
\begin{align}
\PN_{\delta_0}(\mu^N(t) \notin \K_M) & \leq \PbarN_{\delta_0}(\mu^N(t) \notin \K_M) \nonumber \\
& = \PbarN_{\delta_0} \left(\sum_{n=1}^N \vartheta(X_n^N(t)) > NM\right)\nonumber \\
& \leq \exp\{-NM\beta\}\EbarN_{\delta_0}\left(\exp\left\{\beta \sum_{n=1}^N \vartheta(X_n^N(t)) \right\}\right) \nonumber \\
& = \exp\{-NM\beta\} (\bar{E}_0(\exp\{\beta \vartheta(X_1^N(t))\}))^N,
\label{eqn:exp-tight-series}
\end{align}
where the first inequality follows from the above coupling since (i) the $n$th particle under $\PbarN_{\delta_0}$ moves from $z$ to $z+1$ whenever it does so under $\PN_{\delta_0}$, and (ii) the $n$th particle under $\PN_{\delta_0}$ moves to $0$ (i.e., a $z$ to $0$ transition for some $z$) whenever it does so under $\PbarN_{\delta_0}$. The second inequality in~\eqref{eqn:exp-tight-series} is a consequence of Chebyshev's inequality. Recall $\pi$, and the laws $\bar{P}_\pi$ and $\bar{P}_0$.  We couple the laws $\bar{P}_\pi$ and $\bar{P}_0$. Consider the Markov process  on $\Z \times \Z$ with the infinitesimal generator acting on functions $f$ on $\Z \times \Z$ by 
\begin{align*}
(\bar{z}_1, \bar{z}_2) & \mapsto \left(f(\bar{z}_1+1, \bar{z}_2+1) - f(\bar{z}_1, \bar{z}_2)\right) \left(\frac{\overline{\lambda}}{\bar{z}_1+1} \wedge \frac{\overline{\lambda}}{\bar{z}_2+1}\right) \\
& \qquad + \left(f(\bar{z}_1+1, \bar{z}_2) - f(\bar{z}_1, \bar{z}_2)\right) \left(\frac{\overline{\lambda}}{\bar{z}_1+1} - \frac{\overline{\lambda}}{\bar{z}_2+1}\right)^+ \\ 
& \qquad + \left(f(\bar{z}_1, \bar{z}_2+1) - f(\bar{z}_1, \bar{z}_2)\right) \left( \frac{\overline{\lambda}}{\bar{z}_2+1} - \frac{\overline{\lambda}}{\bar{z}_1+1}\right)^+ \\ 
& \qquad + \left(f(0,0) - f(\bar{z}_1, \bar{z}_2)\right) \underline{\lambda}\indf{\{\bar{z}_1 > 0, \bar{z}_2 > 0\}} \\
& \qquad + \left(f(0, \bar{z}_2) - f(\bar{z}_1, \bar{z}_2)\right) \underline{\lambda} \indf{\{\bar{z}_1 > 0, \bar{z}_2 = 0\}}\\
& \qquad + \left(f(\bar{z}_1, 0) - f(\bar{z}_1, \bar{z}_2)\right) \underline{\lambda} \indf{\{\bar{z}_1 = 0, \bar{z}_2 > 0\}}.
\end{align*}
Note that, when the initial condition has law $(\pi, \delta_0)$,  the first (resp. second) component under the above process has law $\bar{P}_\pi$ (resp. $\bar{P}_0$). Also, note that if $\bar{X}_1(0) \geq \bar{X}_2(0)$ then   $\bar{X}_1(s) \geq \bar{X}_2(s)$ for all $s$  under the above coupling. Since the first component is at least the second component under the initial law $(\pi, \delta_0)$, it follows that $\bar{E}_0(\exp\{\beta \vartheta(X_1^N(t))\}) \leq \bar{E}_\pi(\exp\{\beta \vartheta(X_1^N(t))\})$. The latter is finite for sufficiently small $\beta > 0$, thanks to the $\exp\{-\vartheta(z)\}$ decay of the probability measure $\pi$ on $\Z$. Thus we can choose $\bar{\beta} > 0$ small enough (independent of $M$) so that $\log \bar{E}_\pi(\exp\{\bar{\beta}\vartheta(X_1^N(t))\}) < 1$. Hence~\eqref{eqn:exp-tight-series} implies that
\begin{align*}
\PN_{\delta_0} (\mu^N(t) \notin \K_M)& \leq \exp\{-N(M\bar{\beta}-1)\}.
\end{align*}
Therefore, for any $M > 0$ and $T \geq 1$, we get
\begin{align}
\eta^N_T(\sim \hspace{-0.4em} \K_M) \leq \exp\{-N(M\bar{\beta}-1)\}.
\label{eqn:etaT-exp-estimate}
\end{align}
Since $\K_M$ is a compact subset of $\MZ$, this show that the family $\{\eta^N_T, T \geq 1\}$ is tight. Hence it follows that there exists an invariant probability measure $\wp^N$ for $\LN$ (see, for example,~Ethier and Kurtz~\cite[Theorem~9.3,~page~240]{ethier-kurtz}). By Assumption~\ref{assm:a1}, $\mu^N$ is an irreducible Markov process; hence $\wp^N$ is the unique invariant probability measure for $\LN$.

We now show the exponential tightness of the family $\{\wp^N, N \geq 1\}$. Let $M > 0$ be given, and choose $M^\prime = (M+1)/\bar{\beta}$.  For each $N \geq 1$, since $\wp^N$ is a weak limit of the family $\{\eta^N_T, T \geq 1\}$ as $T \to \infty$, from~(\ref{eqn:etaT-exp-estimate}) with $M$ replaced by $M^\prime$, it follows that
\begin{align}
\wp^N(\sim \hspace{-0.4em} \K_{M^\prime}) \leq  \liminf_{T \to \infty} \eta^N_T(\sim \hspace{-0.4em} \K_{M^\prime}) \leq \exp\{-NM\}.
\label{eqn:exp-tightness-mprime-choice}
\end{align}
for each $N \geq 1$. Hence,
\begin{align*}
\limsup_{N \to \infty} \frac{1}{N} \log \wp^N(\sim \hspace{-0.4em} \K_{M^\prime}) \leq -M,
\end{align*}
which establishes that the family $\{\wp^N, N \geq 1\}$ is exponential tight. This completes the proof of the proposition.
\end{proof}

\section{The LDP lower bound}
\label{section:ldp-lower-bound}
In this section we prove the LDP lower bound for the family $\{\wp^N, N\geq 1\}$. To lower bound the probability of a small neighbourhood of a point $\xi$ under $\wp^N$, we first produce  a trajectory that starts at $\K_M$ for a suitable $M > 0$, connects to $\xi^*$ with a small cost, and then reaches $\xi$ from $\xi^*$ with cost arbitrarily close to $V(\xi)$, where $V$ is the quasipotential defined in \eqref{eqn:v}. The probability of a small neighbourhood of  $\xi$ under $\wp^N$ is then lower bounded by the probability that the process $\mu^N$ remains in a small neighbourhood of the trajectory constructed above. The latter is then lower bounded using the  uniform LDP lower bound for $\mu^N$, where the uniformity is over the initial condition  lying in a given compact subset of $\MZ$.

Recall $K(\Delta)$ defined in \eqref{eqn:KD}. We begin with a lemma that allows us to connect points in $K(\Delta)$ to $\xi^*$ for small enough $\Delta$ with small cost. We omit its proof here, since it follows from a certain continuity property of $V$ which will be shown in Lemma~\ref{lemma:v-continuity}.
\begin{lemma}Given $\gamma > 0$ there exist $\Delta > 0$ and $T = T(\Delta) > 0$ such that for any $\zeta \in K(\Delta)$ there exists a trajectory  $\varphi$ on $[0,T]$ such that $\varphi(0) = \zeta$, $\varphi(T) =  \xi^*$, and $S_{[0,T]}(\varphi | \zeta) \leq \gamma$.
\label{lemma:connection-to-gase}
\end{lemma}

We now prove the LDP lower bound for the family $\{\wp^N, N\geq 1\}$.
\begin{lemma}
\label{lemma:lower-bound-specific-path}
For any $\gamma > 0$, $\delta > 0$, and $\xi \in \MZ$, there exists $N_0 \geq 1 $ such that
\begin{align}
\wp^N\{\zeta \in \MZ: d(\zeta, \xi) < \delta\} \geq \exp\{-N(V(\xi) + \gamma)\}
\label{eqn:invariant-measure-lb}
\end{align}
for all $N \geq N_0$.
\end{lemma}
\begin{proof}
Fix $\gamma > 0$, $\delta > 0$, and $\xi \in \MZ$. We may assume that $V(\xi) < \infty$; if $V(\xi) = \infty$ then~(\ref{eqn:invariant-measure-lb}) trivially holds for all $N \geq 1$. Choose some $M > 0$ and $N_1 \geq 1$ such that  $\wp^N(\K_M) \geq 1/2$ for all $N \geq N_1$; this is possible from the exponential tightness of the family $\{\wp^N, N \geq 1\}$, see Proposition~\ref{prop:invariant-measure-exp-tightness}. Using Lemma~\ref{lemma:connection-to-gase}, choose $\varepsilon > 0$ and $T_0 > 0$ such that for any $\zeta_1 \in K(\varepsilon)$ there exists a trajectory $\varphi_1$ on $[0,T_0]$ such that $\varphi_1(0) = \zeta_1, \varphi_1(T_0) = \xi^*$, and $S_{[0,T_0]}(\varphi_1|\zeta_1) \leq \gamma/4$. Since $\xi^*$ is the globally asymptotically stable equilibrium for~\eqref{eqn:mve} and since $\K_M$ is compact, for the above $\varepsilon > 0$, there exists a $T_1 > 0$ such that for any $\zeta \in \K_M$ we have $\mu_\zeta(T_1) \in K(\varepsilon)$, where $\mu_\zeta$ denotes the solution to the McKean-Vlasov equation~\eqref{eqn:mve} with initial condition $\zeta$ (see assumption~\ref{assm:zlogz-convergence}). Also, by the definition of $V(\xi)$, there exists a $T_2 > 0$ and a trajectory $\varphi_2$ such that $\varphi_2(0) = \xi^*, \varphi_2(T_2) = \xi$, and $S_{[0,T_2]}(\varphi_2 | \xi^*) \leq V(\xi) + \gamma/4$. Let $T = T_1 +T_0+ T_2$. Given $\zeta \in \K_M$, we construct a  trajectory $\varphi_\zeta$ on $[0, T]$ by using the above three trajectories as follows. Let $\varphi_\zeta(0) = \zeta$; $\varphi_\zeta(t) = \mu_\zeta(t)$ for $t \in [0,T_1]$; $\varphi_\zeta(t) = \varphi_1(t-T_1)$ for $t \in (T_1, T_1+T_0]$; and $\varphi_\zeta(t) = \varphi_2(t-(T_1+T_0))$ for $t \in (T_1+T_0, T]$. Note that $S_{[0,T]}(\varphi_\zeta | \zeta) \leq V(\xi) + \gamma/2$.  

Recall that $d$ is the metric on $\MZ$ and $\rho$ is the metric on $\DMZ$.  Note that we can choose a $\delta^\prime > 0$ (depending on $T$ and $M$) such that $\rho(\varphi, \varphi_\zeta) < \delta^\prime$ implies that $d(\varphi(T), \varphi_\zeta(T)) < \delta$ for any $\varphi	\in \DMZ$ and $\zeta \in \K_M$. Indeed, if such a choice is not possible, then there exists  a sequence $\{\zeta_n\} \in \K_M$, and a sequence of trajectories $\{\varphi_n\} \subset \DMZ$ such that $S_{[0,T]}(\varphi_{\zeta_n}| \zeta_n) \leq V(\xi) + \gamma/2$  and  $\rho(\varphi_n, \varphi_{\zeta_n}) < 1/n$ for each $n \geq 1$, but $d(\varphi_n(T), \varphi_{\zeta_n}(T)) > \delta$. By the compactness of the level sets of $S_{[0,T]}$ in Lemma \ref{lemma:compactness-S}, it follows that there exists a subsequential limit for $\{\varphi_{\zeta_{n_k}}\}_{k \geq 1}$ (say, $\varphi^*$); since $\rho(\varphi_n, \varphi_{\zeta_n}) < 1/n$, $\varphi_{n_k}$ also converges to $\varphi^*$ in $\DMZ$ as $k \to \infty$.  Furthermore, since $S_{[0,T]}(\varphi^* | \varphi^*_0) <\infty$, from Theorem \ref{thm:weak-equation}, we have that $[0,T] \ni t \mapsto \varphi^*(t)$ is continuous.  Since $\DMZ \ni \varphi   \mapsto \varphi(T)$ is continuous at all $\varphi$ such that $t \mapsto \varphi(t)$ is continuous (see, e.g, \cite[page 124]{billingsley-convergence}), it follows that $d(\varphi_{n_k}(T), \varphi^*(T)) \to 0$ as $k \to \infty$. This  contradicts the assumption $d(\varphi_n(T), \varphi_{\zeta_n}(T)) > \delta$. This shows that we can choose a $\delta^\prime > 0$ such that $\rho(\varphi, \varphi_\zeta) < \delta^\prime$ implies that $d(\varphi(T), \varphi_\zeta(T)) < \delta$ for any $\varphi	\in \DMZ$ and $\zeta \in \K_M$. Therefore, for each $N \geq N_1$, we have
\begin{align}
\wp^N\{\zeta \in \MZ: d(\zeta, \xi) < \delta\} &=  \int_{\MNZ} \PN_\zeta(d(\mu^N(T),\xi) < \delta)\wp^N(d\zeta)  \nonumber \\
&\geq \int_{\K_M \cap \MNZ} \PN_\zeta(d(\mu^N(T),\xi ) < \delta)\wp^N(d\zeta)  \nonumber \\
& \geq \int_{\K_M \cap \MNZ} \PN_\zeta(\rho(\mu^N, \varphi_\zeta) < \delta^\prime ) \wp^N(d\zeta) \nonumber \\
& \geq \frac{1}{2} \inf_{\zeta \in \K_M \cap \MNZ} \PN_\zeta(\rho(\mu^N, \varphi_\zeta) < \delta^\prime );
\label{eqn:lb-series}
\end{align}
here the first equality follows since $\wp^N$ is invariant to time shifts. By the uniform LDP lower bound in Theorem~\ref{thm:uniform-ldp-mun}, there exists $N_2 \geq N_1$ such that 
\begin{align*}
\PN_\zeta(\rho(\mu^N, \varphi) < \delta^\prime ) \geq \exp\{-N(S_{[0,T]}(\varphi | \zeta)+\gamma/4)\}
\end{align*}
for all $\zeta \in \K_M \cap \MNZ$, $\varphi \in \Phi_\zeta^{[0,T]}(V(\xi)+\gamma/2)$, and $N \geq N_2$. Noting that $S_{[0,T]}(\varphi_\zeta | \zeta) \leq V(\xi) + \gamma/2$ for any $\zeta \in \K_M \cap \MNZ$, and using the above uniform LDP lower bound,  \eqref{eqn:lb-series} becomes
\begin{align*}
\wp^N\{\zeta \in \MZ: d(\zeta, \xi) < \delta\} & \geq \frac{1}{2} \exp\{-N(V(\xi) + 3\gamma/4)\}
\end{align*}
for all $N \geq N_2$. Finally, choose $N_0 \geq N_2$ so that $1/2 \geq \exp\{-N\gamma/4\}$. Then the above becomes
\begin{align*}
\wp^N\{\zeta \in \MZ: d(\zeta, \xi) < \delta\} & \geq \exp\{-N(V(\xi) + \gamma)\}
\end{align*}
for all $N \geq N_0$. This completes the proof of LDP lower bound for the family $\{\wp^N, N\geq 1\}$.
\end{proof}

\section{Properties of the quasipotential}
\label{section:v-properties}
In this section we prove three key properties of the quasipotential $V$ defined in \eqref{eqn:v}. These three properties are (i) a  characterisation of the set of points for which $V$ is finite, (ii) a certain continuity property for $V$, and (iii) the compactness of the lower level sets of $V$. These properties play an important role in the proof of the LDP upper bound in Section~\ref{section:ldp-upper-bound}.
\subsection{A characterisation of finiteness of the quasipotential}
Recall the function $\vartheta$ defined in \eqref{eqn:theta} and the compact sets $\K_M$, $M > 0$, defined in \eqref{eqn:KM}. We start with a lemma that enables us to connect $\delta_0$, the point mass at state $0$, to a point $\xi \in \K_M$ for some $M > 0$. This connection is made using a piecewise constant velocity trajectory wherein for each $z \geq 1$, we move the mass $\xi(z)$ from state $0$ to state $z$ in $z$ steps; in the $k$th step, we move the mass $\xi(z)$ from state $k-1$ to state $k$ with unit velocity. The lemma asserts that the cost of this piecewise constant velocity trajectory is bounded above by a constant that depends only on $M$.
\begin{lemma}
\label{lemma:connection-delta0-xi}
Given $M > 0$ there exists a constant $C_M$ depending on $M$ such that for any $\xi \in \K_M$ there exists  a $T > 0$ and a trajectory $\varphi$ on $[0,T]$ such that $\varphi(0) = \delta_0$, $\varphi(T) = \xi$, and $S_{[0,T]}(\varphi | \delta_0) \leq C_M$.
\end{lemma}
\begin{proof}
Fix $M > 0$ and $\xi \in \K_M$. Fix $J \in \Z \setminus \{0\}$ and define $\Z_J = \{1,2,\ldots,J\}$, $t_z = z \xi(z)$ for $z \in \Z_J$, and $T_z = \sum_{z^\prime \in \Z_J, z^\prime \geq z} t_{z^\prime}$. Note that $T_J \leq T_{J-1} \leq \cdots \leq  T_{1}$. We shall first construct a trajectory $\varphi^J$ such that $\varphi^J(0) = \delta_0$, $\varphi^J(T_1)(z) = \xi(z)$ for each $z \in  \Z_J$, and $S_{[0,T_1]}(\varphi^J | \delta_0)$ bounded above by a constant independent of $J$.

Let $T_{J+1} = 0$. For each $z \in \Z_J$, starting with $z = J$, we move the mass $\xi(z)$ from the state $0$ to state $z$ using a piecewise unit velocity trajectory over the time duration $(T_{z+1}, T_{z+1}+t_z]$. We define this trajectory $\varphi^J$ on $[0,T_1]$ as follows. Let $\varphi^J_0 = \delta_0$. For each $z \in \Z_J$ and $1 \leq k \leq z$, when  $t \in (T_{z+1}+(k-1)\xi(z), T_{z+1}+k\xi(z)]$, let
\begin{align*}
\dot{\varphi}^J_t(l) = \left\{ \begin{array}{ll}
1 & \text{ if } l = k \\
-1 & \text{ if }l = k-1 \\
0 & \text { otherwise},
\end{array}
\right.
\end{align*}
$l \in \Z$, and define $\varphi^J_t(l) = \delta_0(l) + \int_{[0,t]} \dot{\varphi}^J_u(l) du$, $l \in \Z$, $t \in [0,T]$. 

We now calculate the cost of this trajectory. Fix $z \in \Z$ such that $\xi(z) > 0$, and let $1 \leq k \leq z$. For each $t \in (T_{z+1}+(k-1)\xi(z), T_{z+1}+k\xi(z))$ and $\alpha \in \CBZ$, note that
\begin{align*}
\langle \alpha,\dot{\varphi}^J_t & - \Lambda_{\varphi^J_t}^*\varphi^J_t \rangle- \sum_{(z,z^\prime) \in \mathcal{E}} \tau(\alpha(z^\prime) - \alpha(z)) \lambda_{z,z^\prime}(\varphi^J_t) \varphi^J_t(z)  \\
 & = (\alpha(k) - \alpha(k-1)) - \sum_{(z,z^\prime) \in \E} (\exp\{\alpha(z^\prime) - \alpha(z)\} - 1) \lambda_{z,z^\prime} (\varphi^J_t) \varphi^J_t(z).
\end{align*}
Hence,
\begin{align}
\sup_{\alpha \in \CBZ} \biggr\{\langle \alpha,\dot{\varphi}^J_t &- \Lambda_{\varphi^J_t}^*\varphi^J_t \rangle- \sum_{(z,z^\prime) \in \mathcal{E}} \tau(\alpha(z^\prime) - \alpha(z)) \lambda_{z,z^\prime}(\varphi^J_t) \varphi^J_t(z) \biggr\}  \nonumber \\
& \leq \sup_{x \in \R} (x - (\exp\{x\}-1) \lambda_{k-1,k}(\varphi^J_t) \varphi^J_t(k-1))  \nonumber \\
& \qquad  + \sup_{\alpha \in \CBZ} \left( - \sum_{(z,z^\prime) \in \E; (z,z^\prime) \neq (k-1,k)} (\exp\{\alpha(z^\prime) - \alpha(z)\} - 1) \lambda_{z,z^\prime}(\varphi^J_t) \varphi^J_t(z)\right) \nonumber \\
& \leq \log\left(\frac{1}{\varphi^J_t(k-1) \lambda_{k-1,k}(\varphi^J_t)}\right) +2 \overline{\lambda} \nonumber \\
& \leq   \log\left(\frac{1}{\varphi^J_t(k-1)}\right) + \log k  +  \log\left(\frac{1}{\underline{\lambda}}\right) +  2 \overline{\lambda},
\label{eqn:connection-series1}
\end{align}
where the last two inequalities follow from assumption~\ref{assm:a2}. Consider the first term above. For $k > 1$, integration of this quantity over the time duration $t \in (T_{z+1}+(k-1)\xi(z), T_{z+1}+k\xi(z))$ gives
\begin{align*}
\int_{(T_{z+1}+(k-1)\xi(z), T_{z+1}+k\xi(z))} \,  \log\left(\frac{1}{\varphi^J_t(k-1)} \right) dt & = -\int_{\xi(z)}^0 \log \left(\frac{1}{u}\right) \, du \\
& = (u\log u - u) \biggr|_{\xi(z)}^0\\
& = \xi(z) \log \left(\frac{1}{\xi(z)}\right) + \xi(z),
\end{align*}
where the first equality follows from the variable change $u = \varphi^J_t(k-1)$ and the facts (i) $\dot{\varphi}^J_t(k-1) = -1$, (ii) $\varphi^J_t(k-1) = \xi(z)$ when $t = T_{z+1}+(k-1)\xi(z)$, (iii) $\varphi^J_t(k-1) = 0$ when $t = T_{z+1}+k\xi(z)$, and (iv) $du = -dt$. For $k=1$, using the bound $\varphi^J_t(0) \geq \varphi^J_t(0) - (1 -  \sum_{z^\prime = z}^J \xi(z^\prime))$, we get
\begin{align*}
\int_{(T_{z+1}, T_{z+1}+\xi(z))} & \, \log\left(\frac{1}{\varphi^J_t(0)} \right)dt \\
& \leq \int_{(T_{z+1}, T_{z+1}+\xi(z))} \log\left(\frac{1}{\varphi^J_t(0) - (1 -  \sum_{z^\prime = z}^J \xi(z^\prime))} \right)dt \\
& =  -\int_{\xi(z)}^0 \log \left(\frac{1}{u}\right) \, du,
\end{align*}
where the last equality follows from the variable change $u = \varphi^J_t(0) - (1 - \sum_{z^\prime = z}^J \xi(z^\prime))$, and the facts (i) $\dot{\varphi}^J_t(0) = -1$, (ii) $\varphi^J_t(0) = 1- \sum_{z^\prime = z+1}^J \xi(z^\prime)$ when $t = T_{z+1}$ so that $\varphi^J_t(0) - (1 -  \sum_{z^\prime = z}^J \xi(z^\prime)) = \xi(z)$ when $t = T_{z+1}$, (iii)  $\varphi^J_t(0) = 1- \sum_{z^\prime = z}^J \xi(z^\prime)$ when $t = T_{z+1}+\xi(z)$ so that $\varphi^J_t(0) - (1 -  \sum_{z^\prime = z}^J \xi(z^\prime)) = 0$ when $t = T_{z+1} + \xi(z)$, and (iv) $du = -dt$. Thus, proceeding as before for the case $k > 1$, we arrive at
\begin{align*}
\int_{(T_{z+1}, T_{z+1}+\xi(z))} \log\left(\frac{1}{\varphi^J_t(0)} \right) dt \leq \xi(z) \log \left(\frac{1}{\xi(z)}\right) + \xi(z).
\end{align*}
Hence, integrating~\eqref{eqn:connection-series1} over $t \in (T_{z+1}+(k-1)\xi(z), T_{z+1}+k\xi(z))$ and  summing over $1 \leq k \leq z$, we get, for each $z \in \Z_J$,
\begin{align}
\int_{(T_{z+1}, T_{z+1}+t_z)} \, \sup_{\alpha \in \CBZ} \biggr\{\langle \alpha,\dot{\varphi}^J_t &- \Lambda_{\varphi^J_t}^*\varphi^J_t \rangle- \sum_{(z,z^\prime) \in \mathcal{E}} \tau(\alpha(z^\prime) - \alpha(z)) \lambda_{z,z^\prime}(\varphi^J_t) \varphi^J_t(z) \biggr\} dt \nonumber \\
& \leq  z\xi(z) \log \left(\frac{1}{\xi(z)}\right)  + \tilde{C}_z,  \label{eqn:bound-cost-zlogz}
\end{align}
where $\tilde{C}_z =  (z \log z+z) \xi(z) +  z \xi(z) \left(\log\left(\frac{1}{\underline{\lambda}}\right) +  2 \overline{\lambda}\right).$ Let $\tilde{C}^J = \sum_{z \in \Z_J} \tilde{C}_z$. Thus, summing the above display over $z \in \Z_J$, we arrive at
\begin{align*}
S_{[0,T_1]}(\varphi^J | \delta_0) \leq \sum_{z \in \Z_J} z \xi(z) \log  \left(\frac{1}{\xi(z)}\right) + \tilde{C}^J.
\end{align*}
Note that 
\begin{align}
\sum_{z \in \Z_J} z \xi(z) \log  \left(\frac{1}{\xi(z)}\right) & = \sum_{\stackrel{z \in \Z_J:}{\xi(z) \leq 1/z^3}}z \xi(z) \log  \left(\frac{1}{\xi(z)}\right) + \sum_{\stackrel{z \in \Z_J:}{\xi(z) > 1/z^3}} z \xi(z) \log  \left(\frac{1}{\xi(z)}\right) \nonumber \\
& \leq \frac{1}{e} +  \sum_{\stackrel{z \in \Z_J\setminus \{1\}:}{\xi(z) \leq 1/z^3}} \frac{3 \log z}{z^2} + 3 \sum_{\stackrel{z \in \Z_J:}{\xi(z) > 1/z^3}} (z \log z) \xi(z) \nonumber \\
& \leq \frac{1}{e} +  3\sum_{z \in \Z_J}\left\{\frac{\log z}{z^2}  + (z \log z) \xi(z)\right\},
\label{eqn:boundzlogz}
\end{align}
where the first inequality comes from the fact that the mapping  $x \mapsto x \log (1/x)$ is monotonically increasing for $x \in [0, 1/e]$. Hence,
\begin{align*}
S_{[0,T_1]}(\varphi^J | \delta_0) \leq \frac{1}{e} +  3\sum_{z \in \Z_J}\left\{\frac{\log z}{z^2}  + (z \log z) \xi(z)\right\} + \tilde{C}^J, \, J \geq 1.
\end{align*}

Define $T = \sum_{z \in \Z} z \xi(z)$. We now extend the trajectory $\varphi^J$  to $(T_1, T]$ by defining $\varphi^J_t = \varphi^J_{T_1}$ for $t \in (T_1, T]$. Noting that $\dot{\varphi}^J_t(z) = 0$ for all $z \in \Z$ on $t \in (T_1, T]$,	this extension suffers an additional cost of at most $2 \overline{\lambda} T$. Hence, we get 
\begin{align*}
S_{[0,T]}(\varphi^J | \delta_0) \leq \frac{1}{e} +  3\sum_{z \in \Z_J}\left\{\frac{\log z}{z^2}  + (z \log z) \xi(z)\right\} + \tilde{C}^J +  2\overline{\lambda}T.
\end{align*}
Noting that (i) the right hand side above is upper bounded by $\langle \xi, \vartheta \rangle C(\overline{\lambda}, \underline{\lambda})$, where $C(\overline{\lambda}, \underline{\lambda})$ is a constant depending on $\overline{\lambda}$ and $\underline{\lambda}$, and (ii) $\langle \xi, \vartheta \rangle \leq M$, the above display yields 
\begin{align*}
S_{[0,T]}(\varphi^J | \delta_0) \leq C(M, \overline{\lambda}, \underline{\lambda}),
\end{align*}
where $C(M, \overline{\lambda}, \underline{\lambda})$ is a constant depending on $M, \overline{\lambda}$, and $\underline{\lambda}$. Using the compactness of the level sets of $S_{[0,T]}$ (see Lemma~\ref{lemma:compactness-S}), it follows that the sequence of trajectories $\{\varphi^J, J \geq 1\}$ has a convergent subsequence. Re-indexing the original sequence, let $\varphi^J \to \varphi$ in $\DMZ$ as $J \to \infty$. By construction, for each $J \in \Z \setminus \{0\}$, $\varphi^J_T(z) = \xi(z)$ for all $z \in \Z_J$; hence $\varphi_T(z) = \xi(z)$ for all $z \in \Z$. Recall that lower semicontinuity of $S_{[0,T]}$ was proved in the course of the proof of Lemma \ref{lemma:compactness-S}. Therefore, it follows that
\begin{align*}
S_{[0,T]}(\varphi | \delta_0) \leq \liminf_{J \to \infty}  S_{[0,T]}(\varphi^J | \delta_0) \leq C(M, \overline{\lambda}, \underline{\lambda}).
\end{align*}
This completes the proof of the lemma.
\end{proof}

We are now ready to characterise the set of points $\xi$ in $\MZ$ whose $V(\xi)$ is finite.
\begin{lemma} $V(\xi) < \infty$ if and only if $\xi \in \K$. Furthermore, for any $M > 0$, there exists a constant $C_M > 0$ such that $\xi \in \K_M$ implies $V(\xi) \leq C_M$.
\label{lemma:v-finiteness}
\end{lemma}
\begin{proof}
Let $\xi \in \MZ$ be such that $V(\xi) < \infty$. Then there exists a $T > 0$ and a trajectory $\varphi$ on $[0,T]$ such that $\varphi(0) = \xi^*, \varphi(T) = \xi$, and $S_{[0,T]}(\varphi | \xi^*) \leq V(\xi) + 1$. By Theorem~\ref{thm:weak-equation}, there exists a measurable function $h_\varphi$ on $[0,T]\times \E$ such that 

\begin{align}
\langle \varphi_t, f \rangle & = \langle \varphi_0, f\rangle  + \int_{[0,t]} \sum_{(z,z^\prime)\in \E} (f(z^\prime) - f(z)) (1+h_\varphi(u,z,z^\prime)) \lambda_{z,z^\prime}(\varphi_u) \varphi_u(z) du
\label{eqn:weak-eqn-xi-km}
\end{align}
holds for all $t \in [0,T]$ and $f \in C_0(\Z)$, and $S_{[0,T]}(\varphi|\varphi(0))$ is given by
\begin{align*}
S_{[0,T]}(\varphi|\varphi(0)) = \int_{[0,T]} \sum_{(z,z^\prime)\in \E} \tau^*(h_\varphi(t,z,z^\prime)) \lambda_{z,z^\prime}(\varphi_t) \varphi_t(z) dt.
\end{align*}
For any $x \geq 0$ and $y \in \R$, using the convex duality relation $(x-1)y \leq \tau^*(x-1)+\tau(y)$, we get the inequality $xy \leq \tau^*(x-1) + (\exp\{y\}-1)$. Hence, from the above non-variational representation for $S_{[0,T]}(\varphi | \varphi(0))$,~\eqref{eqn:weak-eqn-xi-km} implies
\begin{align}
\langle \varphi_t, f \rangle & \leq \langle \xi^*, f\rangle  + \int_{[0,t]}  \sum_{(z,z^\prime)\in \E} \tau^*(h_\varphi(u,z,z^\prime)) \lambda_{z,z^\prime}(\varphi_u) \varphi_u(z) du \nonumber \\
& \qquad  +  \int_{[0,t]}  \sum_{(z,z^\prime)\in \E} (\exp\{f(z^\prime) - f(z)\}-1) \lambda_{z,z^\prime}(\varphi_u) \varphi_u(z) du \nonumber \\
& \leq \langle \xi^*, f\rangle  + V(\xi) + 1 \nonumber \\
& \qquad  +  \int_{[0,t]}  \sum_{(z,z^\prime)\in \E} (\exp\{f(z^\prime) - f(z)\}-1) \lambda_{z,z^\prime}(\varphi_u) \varphi_u(z) du.
\label{eqn:compactness-v-km-series1}
\end{align}
Recall the function $\vartheta$ on $\Z$. For $n \geq 1$, define
\begin{align*}
\vartheta_n(z) = \left\{
\begin{array}{ll}
\vartheta(z), & \text{ if } z \leq n,\\
0, & \text{ otherwise}.
\end{array}
\right.
\end{align*}
By convexity, note that $\vartheta_n(z+1) - \vartheta_n(z) \leq 1 + \log(z+1)$ and $\vartheta_n(0) - \vartheta_n(z) \leq 0$, for each $z \in \Z$. Therefore, using the upper bound for the  transition rates from assumption~\ref{assm:a2}, observe that
\begin{align*}
\int_{[0,t]}  \sum_{(z,z^\prime)\in \E} (\exp\{\vartheta_n(z^\prime) - \vartheta_n(z)\}-1) \lambda_{z,z^\prime}(\varphi_u) \varphi_u(z) du  \leq \overline{\lambda}(e-1)t,
\end{align*}
for each $t \in [0,T]$ and $n \geq 1$. It follows from~\eqref{eqn:compactness-v-km-series1} with $f$ replaced by $\vartheta_n$ that
\begin{align*}
\langle \varphi_t, \vartheta_n \rangle \leq \langle \xi^*, \vartheta_n \rangle + V(\xi) + 1 +  \overline{\lambda}(e-1)T
\end{align*}
for each $t \in [0,T]$ and $n \geq 1$. Letting $n \to \infty$ and using monotone convergence, we conclude that
\begin{align}
\sup_{t \in [0,T]}  \langle \varphi_t, \vartheta \rangle  = \sup_{t \in [0,T]} \lim_{n\to \infty}  \langle \varphi_t, \vartheta_n \rangle \leq \langle \xi^*, \vartheta \rangle + V(\xi) + 1 + \overline{\lambda}(e-1)T.
\label{eqn:compactness-v-km-series2}
\end{align}
In particular, $\langle \xi, \vartheta \rangle \leq \langle \xi^*, \vartheta \rangle + V(\xi) + 1 + \overline{\lambda}(e-1)T$. It follows that $\xi \in \K$.

Conversely, let $\xi \in \K$. Let $M> 0$ be such that $\xi \in \K_M$. By Lemma~\ref{lemma:connection-delta0-xi}, there exists a $T > 0$ and a trajectory $\varphi^{(2)}$ on $[0,T]$ such that $\varphi^{(2)}(0) = \delta_0$, $\varphi^{(2)}(T) = \xi$, and $S_{[0,T]}(\varphi^{(2)} | \delta_0) \leq C_M$ for some constant $C_M > 0$ depending on $M$.  Let $t_0 = 0$, $t_z = \sum_{z^\prime =1}^{z}\xi^*(z^\prime)$, $z \in \Z \setminus\{0\}$, and $T_1 = \sum_{z^\prime \neq 0}\xi^*(z^\prime)$.  We now construct another trajectory $\varphi^{(1)}$ on $[0,T_1]$ such that $\varphi^{(1)}(0) = \xi^*$, $\varphi^{(1)}(T_1) = \delta_0$, and $S_{[0,T_1]}(\varphi^{(1)}| \xi^*) < \infty$. This trajectory is constructed using piecewise constant velocity paths and its cost $S_{[0,T_1]}(\varphi^{(1)}| \xi^*)$ is computed using arguments similar to those used in the proof of Lemma~\ref{lemma:connection-delta0-xi}; we provide the details here for completeness.  When $t \in  (t_{z-1}, t_z]$ for some $z \in \Z \setminus \{0\}$, let
\begin{align*}
\dot{\varphi}^{(1)}_t(l) = \left\{
\begin{array}{ll}
-1, & \text{ if } l = z, \\
1, & \text{ if } l = 0, \\
0, & \text{ otherwise},
\end{array}
\right.
\end{align*}
$l \in \Z$, and define $\varphi^{(1)}_t(l) = \varphi^{(1)}_0(l) +  \int_{[0,t]} \dot{\varphi}^{(1)}_u(l) du$, $l \in \Z$, $t \in [0,T_1]$. Note that, for each $\alpha \in \CBZ$,  when $t \in (t_{z-1}, t_z)$, we have
\begin{align*}
\biggr\{ \langle \alpha, & \dot{\varphi}^{(1)}_t  - \Lambda_{\varphi^{(1)}_t}^*\varphi^{(1)}_t \rangle- \sum_{(z,z^\prime) \in \mathcal{E}} \tau(\alpha(z^\prime) - \alpha(z)) \lambda_{z,z^\prime}(\varphi^{(1)}_t) \varphi^{(1)}_t(z) \biggr\} \\
& = (\alpha(0) - \alpha(z)) - (\exp\{\alpha(0) - \alpha(z)\}-1) \lambda_{z,0}(\varphi^{(1)}_t) \varphi^{(1)}_t(z) \\
& \qquad - \sum_{(z_0,z^\prime) \in \mathcal{E}: (z_0,z^\prime) \neq (z,0)} (\exp\{\alpha(z^\prime) - \alpha(z_0)\}-1) \lambda_{z_0,z^\prime}(\varphi^{(1)}_t) \varphi^{(1)}_t(z_0) \biggr\}, 
\end{align*}
so that optimising the left hand side of the above display over $\alpha	\in \CBZ$ yields
\begin{align*}
\sup_{\alpha \in \CBZ}\biggr\{ \langle \alpha, & \dot{\varphi}^{(1)}_t  - \Lambda_{\varphi^{(1)}_t}^*\varphi^{(1)}_t \rangle- \sum_{(z,z^\prime) \in \mathcal{E}} \tau(\alpha(z^\prime) - \alpha(z)) \lambda_{z,z^\prime}(\varphi^{(1)}_t) \varphi^{(1)}_t(z) \biggr\} \\
& \leq \log\left(\frac{1}{ \varphi^{(1)}_t(z) \lambda_{z,0}(\varphi^{(1)}_t)}\right) + 2\bar{\lambda} \\
& \leq \log\left(\frac{1}{ \varphi^{(1)}_t(z)}\right) + \log\left(\frac{1}{\underline{\lambda}}\right) + 2\overline{\lambda},
\end{align*}
where the last inequality follows form the lower bound on the backward transition rates in assumption~\ref{assm:a2}. Integrating the above over $(t_{z-1}, t_z)$ and summing over $z \in \Z \setminus \{0\}$, we arrive at
\begin{align}
S_{[0,T_1]}(\varphi^{(1)} | \xi^* ) \leq \sum_{z \in \Z \setminus\{0\}} \left\{\xi^*(z) \log \frac{1}{\xi^*(z)} + \xi^*(z)\left( 1 + \log\left(\frac{1}{\underline{\lambda}}\right) + 2\overline{\lambda}\right) \right\}. \label{eqn:cost-xistar-bound}
\end{align}
Since $\xi^* \in \K$, proceeding via the steps in~\eqref{eqn:boundzlogz}, we conclude that the right hand side of the above display is finite.  We combine $\varphi^{(1)}$ and $\varphi^{(2)}$ and define a new trajectory	$\tilde{\varphi}$ on $[0,T_1+T]$ as follows: $\tilde{\varphi}(t) = \varphi^{(1)}(t)$ on $t \in [0, T_1]$; $\tilde{\varphi}(t) = \varphi^{(2)}(t-T_1)$ on $t \in (T_1, T_1+T]$. Note that $\tilde{\varphi}(0) = \xi^*$, $\tilde{\varphi}(T_1+T) = \xi$, and $S_{[0,T_1+T]}(\tilde{\varphi} | \xi^*) < \infty$. Hence $V(\xi) < \infty$.

To prove the second statement, we note that given any $M > 0$, for any $\xi \in \K_M$, the cost of the trajectory $\tilde{\varphi}$ constructed in the previous paragraph is bounded above by a constant depending only on $M$ (and not on $\xi$). This completes the proof of the lemma.
\end{proof}
\subsection{Continuity}
We now establish a certain continuity property of the quasipotential $V$. Since $V$ has compact level sets and the space $\MZ$ is not locally compact, we cannot expect $V$ to be continuous on $\MZ$. In fact, for any point $\xi \in \MZ$ with $V(\xi) < \infty$, one can produce a sequence $\{\xi_n, n\geq 1\}$ such that $\xi_n \to \xi$ in $\MZ$ as $n \to \infty$, and $\langle \xi_n, \vartheta \rangle = \infty$ for all $n \geq 1$,  so that $\inf_{n \geq 1} V(\xi_n) = \infty$.  We prove that $V$ is continuous under the convergence of $\vartheta$-moments when it is restricted to $\K$. That is, when $\xi_n, \xi \in \K$, $\xi_n \to \xi$ in $\MZ,$ and $\langle \xi_n, \vartheta \rangle \to \langle \xi, \vartheta\rangle $ as $n \to \infty$, then $V(\xi_n) \to V(\xi)$ as $n \to \infty$. Towards this, we produce a trajectory that connects $\xi$ to $\xi_n$ by first moving the mass from all the large enough states $z$ back to the state $0$, then producing a constant velocity trajectory that fills the required mass from state $0$ to all the large enough states $z$, and finally adjusting mass within a finite subset of $\Z$ to reach $\xi_n$. We show that the cost of the trajectory constructed above can be made arbitrarily small for large enough $n$.
\begin{lemma}
Let $\xi_n \in \K$, $n \geq 1$, and $\xi \in \K$. Suppose that $\xi_n \to \xi$ in $\MZ$ and $\langle \xi_n, \vartheta \rangle \to \langle \xi, \vartheta \rangle$ as $n \to \infty$. Then $V(\xi_n) \to V(\xi)$ as $n \to \infty$.
\label{lemma:v-continuity}
\end{lemma}
\begin{proof}
We first prove that $\limsup_{n \to \infty} V(\xi_n) \leq V(\xi)$.  Fix $\varepsilon > 0$. We shall move from $\xi$ to $\xi_n$ in five steps. The outline of this construction is as follows:
\begin{itemize}
\item $\varphi^{(0)}$: This trajectory  starts with $\xi$ and moves all the mass for all states $z > z_0$, for a suitable large enough $z_0$, back to state $0$. This backward movement results in a cost of $O(\varepsilon)$.
\item $\varphi^{(1)}$: Next, we move any additional mass, if required, from the states $\{1,2,\ldots, z_0\}$ back to state $0$ so that there is enough mass at state $0$ to fill up all the states beyond $z_0$. Again, this  backward movement results in a cost of $O(\varepsilon)$.
\item $\varphi^{(2)}$: Next, we construct a piecewise constant-velocity trajectory to move the mass $\sum_{z^\prime > z_0}\xi_n(z)$ from state $0$ to state $z_0+1$. After this movement, state $z_0+1$ contains all the mass required to fill up the states beyond it. This forward movement results in  a cost of $O(\varepsilon \log (1/\varepsilon))$, instead of $O(\varepsilon)$, because we move the total mass for all the states beyond $z_0$.
\item $\varphi^{(3)}$: Then, for each $z > z_0$, we move the required mass (i.e., $\xi_n(z)$) from state $0$ to state $z$ using a piece-wise constant velocity trajectory. At the end of this procedure, for each $z > z_0$, the mass at state $z$ becomes $\xi_n(z)$. This forward movement results in  a cost of $O(\varepsilon)$. 
\item $\varphi^{(4)}$: Finally, we adjust the mass within the finite set $\{1,2,\ldots, z_0\}$ to match with $\xi_n$. This also results in a cost at most $O(\varepsilon \log (1/\varepsilon))$. Again, this cost is $O(\varepsilon \log (1/\varepsilon))$ instead of $O(\varepsilon)$ because we  move, for each $z \in \{1,2,\ldots, z_0\}$, the sum of the additional  mass (under $\xi_n$ compared to $\xi$) in the states  $\{z, z+1, \ldots, z_0\}$  from state $0$ to state $z$.
\end{itemize}
Therefore, the total cost of all these trajectories is at most $O(\varepsilon \log (1/\varepsilon))$, which vanishes as $\varepsilon \to 0$. We now define these trajectories in detail and evaluate their costs.

Let $z_0 \geq 2$ be such that 
\begin{align*}
\sum_{z > z_0} \vartheta(z) \xi(z) < \varepsilon/6 \quad  \text{and} \quad  \sum_{z > z_0}\frac{\log z}{z^2} < \varepsilon.
\end{align*}
Then choose $n_1 \geq 1$ such that $\sum_{z > z_0} \vartheta(z) \xi_n(z) < \varepsilon/3$ holds for all $n \geq n_1$; this is possible since $\xi_n \to \xi$ in $\MZ$ and $\langle \xi_n, \vartheta \rangle \to \langle \xi, \vartheta \rangle$ as $n \to \infty$. Let 
\begin{align*}
t_{z_0} = 0, \quad t_z = \sum_{z^\prime = z_0+1}^z \xi(z^\prime), \, z > z_0, \quad \text{ and} \quad T_0 = \sum_{z^\prime > z_0} \xi(z^\prime).
\end{align*}
Define the trajectory $\varphi^{(0)}$ on $[0,T_0]$ as follows. When $t \in (t_{z-1}, t_{z}]$ for some $z > z_0$, let
\begin{align*}
\dot{\varphi}^{(0)}_t(l) = \left\{
\begin{array}{ll}
-1, & \text{ if } l = z, \\
1, & \text{ if } l = 0, \\
0, & \text{ otherwise},
\end{array}
\right.
\end{align*} 
$l \in \Z$, and define 
\begin{align*}
\varphi^{(0)}_t(l) = \xi(l) + \int_{[0,t]} \dot{\varphi}^{(0)}_u(l) du, \quad l \in \Z, \, t \in [0, T_0].
\end{align*}
Note that $\varphi^{(0)}_{T_0}(z) = \xi(z)$ for $1 \leq z \leq z_0$, $\varphi^{(0)}_{T_0}(z) = 0$ for $z > z_0$, and $\varphi^{(0)}_{T_0}(0) = \xi(0) + \sum_{z > z_0} \xi(z)$. Let $M = (\sup_{n \geq n_1} \langle \xi_n, \vartheta \rangle) \vee   \langle \xi, \vartheta \rangle + 1$.  Using ideas similar to those used in the proof of Lemma~\ref{lemma:v-finiteness}, it can be checked that $S_{[0,T_0]}(\varphi^{(0)} | \xi) \leq C_0(M, \overline{\lambda}, \underline{\lambda}) \varepsilon$, for some constant $C_1(M, \overline{\lambda}, \underline{\lambda})$ depending on $M$, $\overline{\lambda}$, and  $\underline{\lambda}$.  Indeed, the cost  is $O(\sum_{z > z_0}\xi(z)  \log (1/\xi(z)))$, which, using the argument used to arrive at the bound \eqref{eqn:cost-xistar-bound} and the choice of $z_0$, is bounded by 
\begin{align*}
O\left(\sum_{z > z_0}\left((z \log z)\xi(z) + \frac{\log z}{z^2} \right) \right) = O(\varepsilon).
\end{align*}

Let $\varepsilon_n = \sum_{z > z_0} \xi_n(z)$. If $\varepsilon_n > \varphi_{T_0}^{(0)}(0)$, then we move the extra mass $\varepsilon_n - \varphi_{T_0}^{(0)}(0)$ from the states $\{1,2,\ldots, z_0\}$ to state $0$ as follows. Let $T_1 = T_0 + \varepsilon_n - \varphi_{T_0}^{(0)}(0)$. When $t$ is between $T_0 + \sum_{z^\prime = z+1}^{z_0} \varphi_{T_0}^{(0)}(z^\prime)$ and $(T_0 + \sum_{z^\prime = z}^{z_0} \varphi_{T_0}^{(0)}(z^\prime)) \wedge T_1$ for some $z \leq z_0$, let
\begin{align*}
\dot{\varphi}^{(1)}_t(l) = \left\{
\begin{array}{ll}
-1, & \text{ if } l = z, \\
1, & \text{ if } l = 0,\\
0, & \text{ otherwise},
\end{array}
\right.
\end{align*} 
$l \in \Z$.  Define the trajectory $\varphi^{(1)}$  on $[0, T_1]$ as follows:  $\varphi^{(1)}_t = \varphi^{(0)}_t$ when $t \in [0, T_0]$; $\varphi^{(1)}_t(l) = \varphi^{(0)}_{T_0}(l) +  \int_{[0,t]} \dot{\varphi}^{(1)}_u(l) du$, $l \in \Z$, $t \in (T_0, T_1]$. Note that $\varphi^{(1)}$ depends on $n$, but we suppress this in the notation for ease of readability.   Again, since $\varepsilon_n$ is smaller than $\varepsilon/3$, by using calculations similar to those used in the proof of Lemma~\ref{lemma:v-finiteness}, we see that $S_{[T_0, T_1]}(\varphi^{(1)}|\varphi_{T_0}^{(0)}) \leq C_1(M,\overline{\lambda}, \underline{\lambda}) \varepsilon$ for some constant $C_1(M,\overline{\lambda}, \underline{\lambda})$ depending on $M$, $\overline{\lambda}$, and $\underline{\lambda}$. On the other hand, if $\varepsilon_n \leq \varphi_{T_0}^{(0)}(0)$, we set $T_1 = T_0$ and $\varphi^{(1)}_t = \varphi^{(0)}_t$ on $[0, T_1]$. In both cases, we have $\varphi^{(1)}_{T_1}(0) \geq \varepsilon_n $.

Let $T_2 = (z_0+1) \varepsilon_n$. We now construct another trajectory $\varphi^{(2)}$ on $[0,T_2]$ to transfer the mass $\varepsilon_n$ from state $0$ (in $\varphi^{(1)}_{T_1}$) to state $z_0 + 1$. Let $\varphi^{(2)}_0 = \varphi^{(1)}_{T_1}$. When $t \in ((z-1)\varepsilon_n, z\varepsilon_n]$ for some $z \in \{1,2,\ldots, z_0+1\}$, let
\begin{align*}
\dot{\varphi}^{(2)}_t(l) = \left\{
\begin{array}{ll}
-1, & \text{ if } l = z-1, \\
1, & \text{ if } l = z,\\
0, & \text{ otherwise},
\end{array}
\right.
\end{align*} 
$l \in \Z$, and define $\varphi^{(2)}_t(l) = \varphi^{(1)}_{T_1}(l) +  \int_{[0,t]} \dot{\varphi}^{(2)}_u(l) du$, $l \in \Z$, $t \in (0, T_2]$. Note that $|x \log (\frac{1}{x})-  y \log (\frac{1}{y})| \leq \delta + \delta \log(1/\delta)$ whenever $|x - y| \leq \delta$,  and that  $\varepsilon_n \leq \varepsilon/(z_0 \log z_0)$. Hence,  using calculations similar to those done in the proof of Lemma~\ref{lemma:connection-delta0-xi}, we see that $S_{[0,T_2]}(\varphi^{(2)}|\varphi^{(1)}_{T_1})$ can be bounded above by $C_2(M, \overline{\lambda}, \underline{\lambda})\varepsilon \log (1/\varepsilon)$ where $C_2(M, \overline{\lambda}, \underline{\lambda})$ is a constant depending on $M$, $\overline{\lambda}$, and  $\underline{\lambda}$, for each $n \geq n_1$ (recall that $\varphi^{(2)}$ depends on $n$). Indeed, the cost is bounded by the order of (see the bound in \eqref{eqn:bound-cost-zlogz})
\begin{align*}
z_0 \varepsilon_n \log \left(\frac{1}{\varepsilon_n}\right) + (z_0 \log z_0) \varepsilon_n & \leq z_0 \frac{\varepsilon}{z_0 \log z_0} \log \left(\frac{z_0 \log z_0}{\varepsilon} \right)  + \varepsilon\\
& \leq \varepsilon \log (1/\varepsilon) + 3 \varepsilon,
\end{align*}
where the first inequality uses the fact that $\varepsilon_n \leq \varepsilon/(z_0 \log z_0)$, and the second inequality uses the fact that $z_0 \geq  2$ so that $z_0 \log z_0 >  1$.

Note that $\varphi^{(2)}_{T_2}(z_0+1) = \varepsilon_n$. We now construct a trajectory that  distributes this mass $\varepsilon_n$ from the state $z_0 + 1$ to all the states $z \geq z_0 + 1$ to match with $\xi_n(z)$. Let $t^\prime_z = z \xi_n(z)$ for $z \geq z_0+2$ and $T_3 =  \sum_{z \geq z_0+2} t^\prime_z$. Similar to the construction in the proof of Lemma~\ref{lemma:connection-delta0-xi}, we can now construct a trajectory $\varphi^{(3)}$ on $[0,T_3]$ such that $\varphi^{(3)}_0 = \varphi^{(2)}_{T_2}$, $\varphi^{(3)}_{T_3}(z) = \xi_n(z)$ for each $z \geq z_0+1$, and $S_{[0,T_3]}(\varphi^{(3)}|\varphi^{(2)}_{T_2}) \leq C_3(M, \overline{\lambda}, \underline{\lambda}) \varepsilon$ for some constant $C_3(M, \overline{\lambda}, \underline{\lambda})$ depending on $M$, $\overline{\lambda}$, and  $\underline{\lambda}$, for all $n \geq n_1$. Indeed, using the bounds in  \eqref{eqn:bound-cost-zlogz} and \eqref{eqn:boundzlogz}, the total cost is bounded by the order of 
\begin{align*}
\sum_{z > z_0 + 1} \left( (z \log z) \xi_n(z) + \frac{\log z}{z^2} \right) \leq \frac{\varepsilon}{3} + \varepsilon,
\end{align*}
where the inequality follows from the choice of $z_0$.

Finally, we construct a trajectory that connects $\varphi^{(3)}_{T_3}$ to $\xi_n$ by adjusting the mass within the states $\{0,1,\ldots, z_0\}$. Note that $\varphi^{(3)}_{T_3}(z) = \xi_n(z)$ for each $z \geq z_0+1$. Let $\Z_0 \subset \{1,2,\ldots, z_0\}$ denote the set of all $z \in \{1,2,\ldots, z_0\}$ such that $\varphi^{(3)}_{T_3}(z) > \xi_n(z)$. Similar to the construction of $\varphi^{(1)}$, for each $z \in \Z_0$, we move the mass $\varphi^{(3)}_{T_3}(z) - \xi_n(z)$ from state $z$ to state $0$ using unit velocity over a time duration $\varphi^{(3)}_{T_3}(z) - \xi_n(z)$. Once these mass transfers are complete, starting with $z  =1$, we move the mass 
\begin{align*}
\sum_{z^\prime \geq z, z^\prime \notin \Z_0, z^\prime \leq z_0} (\xi_n(z^\prime) - \varphi_{T_3}^{(3)}(z^\prime))
\end{align*}
from state $z-1$ to state $z$ with at unit rate. Let 
\begin{align*}
T_4 =  \sum_{z \in \Z_0}(\varphi^{(3)}_{T_3}(z) - \xi_n(z)) + \sum_{z \notin \Z_0, z \leq z_0} (\xi_n(z) - \varphi^{(3)}_{T_3}(z)),
\end{align*}
and let $\varphi^{(4)}$ denote this piecewise constant velocity trajectory. Let $\tilde{\varepsilon}_n
 = \sum_{z \notin \Z_0, z \leq z_0} (\xi_n(z) - \varphi_{T_3}^{(3)}(z))$.
At each step of $\varphi^{(4)}$, since we move a mass of at most $\tilde{\varepsilon}_n$ from state $z-1$ to state $z$, the cost of $\varphi^{(4)}$ is at most of the order of (see \eqref{eqn:bound-cost-zlogz})
\begin{align*}
z_0 \tilde{\varepsilon}_n \log \left(\frac{1}{\tilde{\varepsilon}_n}\right) + (z_0 \log z_0) \tilde{\varepsilon}_n.
\end{align*}
Since $\tilde{\varepsilon}_n \to 0$ as $n \to \infty$, we may choose $n_2 \geq n_1$ so that $\tilde{\varepsilon}_n \leq \varepsilon/(z_0 \log z_0)$ for all $n \geq n_2$. Therefore, for $n \geq n_2$, the above display is bounded by
\begin{align*}
z_0 \frac{\varepsilon}{z_0 \log z_0} \log \left(\frac{z_0 \log z_0}{\varepsilon}\right)+\varepsilon \leq \varepsilon \log (1/\varepsilon) + 3 \varepsilon,
\end{align*}
which is $O(\varepsilon \log (1/\varepsilon))$. Therefore, $S_{[0,T_4]}(\varphi^{(4)} | \varphi^{(3)}_{T_3}) \leq C_4(M, \overline{\lambda}, \underline{\lambda}) \varepsilon \log (1/\varepsilon)$ for all $n \geq n_2$, for some constant $C_4(M, \overline{\lambda}, \underline{\lambda})$ depending on $M$, $\overline{\lambda}$, and  $\underline{\lambda}$.

Let $T = \sum_{i=1}^4 T_i$. We now append the four paths $\varphi^{(i)}, 1 \leq i \leq 4$, constructed in the previous paragraphs over the time duration $[0,T]$ to get a path $\varphi$ such that $\varphi_0 = \xi$, $\varphi_T = \xi_n$ and $S_{[0,T]}(\varphi | \xi) \leq C(M, \overline{\lambda}, \underline{\lambda}) \varepsilon \log (1/\varepsilon)$ where $C(M, \overline{\lambda}, \underline{\lambda})$ is a constant depending on $M$, $\overline{\lambda}$ and $\underline{\lambda}$. Hence, for each $n \geq n_2$, we have
\begin{align*}
V(\xi_n) \leq V(\xi) + S_{[0,T_4]}(\varphi | \xi) \leq V(\xi) + C(M, \overline{\lambda}, \underline{\lambda}) \varepsilon \log (1/\varepsilon).
\end{align*}
Therefore, $\limsup_{n \to \infty} V(\xi_n) \leq V(\xi) + C(M, \overline{\lambda}, \underline{\lambda}) \varepsilon \log (1/\varepsilon)$. Letting $\varepsilon \to 0$ and noting that $\varepsilon \log (1/\varepsilon) \to 0$, we arrive at $\limsup_{n \to \infty} V(\xi_n) \leq V(\xi)$. 

To prove $\liminf_{n \to \infty} V(\xi_n) \geq V(\xi)$, we reverse the role of $\xi_n$ and $\xi$ in the above argument. That is, we construct a trajectory $\varphi$ on $[0,T]$ such that $\varphi_0 = \xi_n$, $\varphi_T = \xi$, and $S_{[0,T]}(\varphi | \xi_n) \leq \varepsilon_n$ for all $n \geq 1$, where $\varepsilon_n \to 0$ as $n \to \infty$. Thus, we get
\begin{align*}
V(\xi) \leq V(\xi_n) + \varepsilon_n.
\end{align*}
Letting $n \to \infty$, we conclude that $\liminf_{n \to \infty} V(\xi_n) \geq V(\xi)$. This completes the proof of the lemma.
\end{proof}
\begin{remark}
\label{remark:lscV}
The choice of $n_1$ in the above proof suggests that the inequality $\limsup_{n \to\infty} V(\xi_n) \leq V(\xi)$ can be proved as long as $\xi_n \to \xi$ in $\MZ$ as $n \to \infty$ and  $\limsup_{n \to \infty} \langle \xi_n, \vartheta \rangle \leq \langle \xi, \vartheta \rangle$ holds. Similarly, the inequality $\liminf_{n \to\infty} V(\xi_n) \geq V(\xi)$ can be proved as long as  $\xi_n \to \xi$ in $\MZ$ and $\liminf_{n \to \infty} \langle \xi_n, \vartheta \rangle \geq \langle \xi, \vartheta \rangle$ holds. This observation will be later used in the proof of the compactness of the lower level sets of $V$.
\end{remark}

\subsection{Compactness of the lower level sets of the quasipotential}
Define the level sets of $V$ by
\begin{align*}
\Xi(s) \coloneqq \{\xi \in \MZ: V(\xi) \leq s\}, \, s > 0.
\end{align*}	
In this section we establish the compactness of $\Xi(s)$ for each $s > 0$.
\begin{lemma}
For each $s > 0$, $\Xi(s)$ is a compact subset of $\MZ$.
\label{lemma:v-compactness-level-sets}
\end{lemma}
\begin{proof}
We first prove an inclusion property of the level sets of $V$, namely, given $M> 0$ there exists $M^\prime > 0$ such that 
\begin{align}
 \{\xi \in \MZ: V(\xi) \leq M\} \subset \K_{M^\prime}.
\label{eqn:v-km-containment}
\end{align}
On one hand, using Proposition~\ref{prop:invariant-measure-exp-tightness} on the exponential tightness of the family $\{\wp^N, N \geq 1\}$, choose $M^\prime > 0$ (see~\eqref{eqn:exp-tightness-mprime-choice}) such that 
\begin{align*}
\limsup_{N \to \infty} \frac{1}{N} \log \wp^N(\sim \hspace{-0.4em} \K_{M^\prime}) \leq -(M+1).
\end{align*}
On the other hand, using the LDP lower bound established in Lemma~\ref{lemma:lower-bound-specific-path} and the compactness of $\K_{M^\prime}$, we have 
\begin{align*}
\liminf_{N \to \infty} \frac{1}{N} \log \wp^N(\sim \hspace{-0.4em} \K_{M^\prime}) \geq -\inf_{\xi \notin  \K_{M^\prime}} V(\xi).
\end{align*}
Combining the above two displays, we get
\begin{align*}
-\inf_{\xi \notin \K_{M^\prime}} V(\xi) \leq \liminf_{N \to \infty} \frac{1}{N} \log \wp^N(\sim \hspace{-0.4em} \K_{M^\prime}) \leq \limsup_{N \to \infty} \frac{1}{N} \log \wp^N(\sim \hspace{-0.4em} \K_{M^\prime}) \leq -(M+1).
\end{align*}
That is, $\xi \notin \K_{M^\prime}$ implies $V(\xi) \geq  M+1 > M$. This shows~\eqref{eqn:v-km-containment}. By Prohorov's theorem, $\K_M$ is a compact subset of $\MZ$; hence~\eqref{eqn:v-km-containment} shows that $\Xi(s)$ is precompact for each $s > 0$.

We now show that $\Xi(s)$ is closed in $\MZ$. Let $\xi_n \in \Xi(s)$ for each $n \geq 1$ and let $\xi_n \to \xi$ in $\MZ$ as $n \to \infty$. By Fatou's lemma, we have $\liminf_{n \to \infty} \langle \xi_n, \vartheta \rangle \geq \langle \xi, \vartheta \rangle$. Hence, by Remark~\ref{remark:lscV}, we have $\liminf_{n \to \infty} V(\xi_n) \geq V(\xi)$. Thus, $\xi \in \Xi(s)$. This completes the proof of the lemma.
\end{proof}
\section{The LDP upper bound}
\label{section:ldp-upper-bound}
Recall $\K_M$ defined in \eqref{eqn:KM} and $K(\Delta)$ defined in \eqref{eqn:KD}. For $m \in \N$, define
\begin{align*}
\mathscr{S}_m(\Delta, M) = \{\varphi \in D([0,m],\MZ) : \varphi(0) \in \K_M, \varphi(n) \notin K(\Delta) \text{ for all } n = 1,2,\ldots , m\}.
\end{align*}
That is, $\mathscr{S}_m(\Delta, M)$ denotes the set of all trajectories that start at $\K_M$ and do not intersect $K(\Delta)$ at all integer time points in $[0,m]$. We begin with a lemma that asserts that the elements of $\mathscr{S}_m(\Delta, M)$ for large enough $m$ must have non-trivial cost. The key idea used in the proof comes from the compactness of level sets of the process-level large deviations rate function $S_{[0,T]}(\cdot | \nu), \nu \in K$, for any compact subset $K$ of $\MZ$ (see Lemma~\ref{lemma:compactness-S}).
\begin{lemma}
\label{lemma:upper-bound-cost-lower-bound}
For any $s > 0$, $M > 0$, and $\Delta > 0$, there exists $m_0 \in \N$ such that 
\begin{align}
\inf\{S_{[0,m_0]}(\varphi | \varphi(0)), \varphi \in \mathscr{S}_{m_0}(\Delta, M)\} > s.
\label{eqn:upper-bound-cost-lower-bound}
\end{align}
\end{lemma}
\begin{proof}
Suppose not. Then there exist $s > 0$, $M > 0$, $\Delta > 0$, a sequence of positive numbers $\{\varepsilon_m, m \geq 1\}$ such that $\varepsilon_m \to 0$ as $m \to \infty$, and a sequence of trajectories $\{\varphi_m, m \geq 1\}$ such that $\varphi_m \in \mathscr{S}_{m}(\Delta,M)$, and $S_{[0,m]}(\varphi_m | \varphi_m(0)) \leq s + \varepsilon_m$ for each $m \geq 1$.

Note that there exists an $M_1 > 0$ such that $\varphi_m(t) \in \K_{M_1}$ for each $t \in [0,m]$ and each $m \geq 1$. Indeed, by Lemma~\ref{lemma:v-finiteness}, there exists $C_M > 0$ such that $\zeta \in K(\Delta)$ implies $V(\zeta) \leq C_M$. Thus, for each $m \geq 1$, there exist a $\bar{T}_m > 0$ and a  trajectory $\bar{\varphi}_m$ on $[0,\bar{T}_m]$ such that $\bar{\varphi}_m(0) = \xi^*$, $\bar{\varphi}_m(\bar{T}_m) = \zeta \in K(\Delta)$, and $S_{[0,\bar{T}_m]}(\bar{\varphi}_m | \xi^*) \leq C_M+1$. We extend this trajectory $\bar{\varphi}_m$ to $(\bar{T}_m, \bar{T}_m+m]$ by defining $\bar{\varphi}_m(t) = \varphi_m(t-\bar{T}_m)$ on $t \in (\bar{T}_m, \bar{T}+m]$. Note that $S_{[0, \bar{T}_m+m]}(\bar{\varphi}_m | \xi^*) \leq C_M+1+s+\varepsilon_m$, so that $V(\varphi_m(t)) \leq C_M+1+s+\varepsilon_m$ for each $t \in [0,m]$ and  each $m \geq 1$. Thus, we can find an $M_1 > 0$ such that~\eqref{eqn:v-km-containment} holds with $M$ replaced by $ C_M+s+\sup_{m\geq 1} \varepsilon_m+2$ and $M^\prime$ replaced by $M_1$. It follows that $\varphi_m(t) \in \K_{M_1}$ for each $t \in [0,m]$ and each $m \geq 1$.

For the above choice of $M_1$, using assumption~\ref{assm:zlogz-convergence}, choose $T_1 > 1$ such that $\mu_{\zeta}(t) \in K(\Delta/2)$ for each $t \geq T_1$ and each $\zeta \in \K_{M_1}$, where $\mu_\zeta$ is the solution to the McKean-Vlasov equation~\eqref{eqn:mve} with initial condition $\zeta$. Note that the closure of the set of all trajectories $\varphi$ on $[0,T_1]$ in $D([0,T_1], \MZ)$ with initial condition $\varphi(0) \in \K_{M_1}$ and $\varphi(T_1) \notin K(\Delta)$ does not contain any trajectory of the McKean-Vlasov equation~\eqref{eqn:mve}. It follows from Lemma~\ref{lemma:compactness-S} that 
\begin{align*}
\beta \coloneqq \inf\{S_{[0,T_1]}(\varphi| \varphi(0)), \varphi(0) \in \K_{M_1}, \varphi(n) \notin K(\Delta) \text{ for each } n=1,2,\ldots,\lfloor T_1 \rfloor\} > 0.
\end{align*}
Therefore, noting that $\varphi_m(t) \in \K_{M_1}$ for each $t \in [0,m]$ and $m \geq 1$, we see that
\begin{align*}
S_{[0,m]}(\varphi_m | \varphi_m(0)) & \geq \sum_{n=1}^{\lfloor m/T_1 \rfloor} S_{[(n-1)T_1,nT_1]}(\varphi_m|\varphi_m((n-1)T_1)) \\
& \geq \biggr\lfloor \frac{m}{T_1} \biggr\rfloor \beta \\
& \to \infty \text{ as } m \to \infty,
\end{align*}
which contradicts our assumption. This completes the proof of the lemma.
\end{proof}

With a slight abuse of notation, given $A \subset \MZ$, $s > 0$, and $T > 0$, define
\begin{align*}
\Phi_A^{[0,T]}(s)\coloneqq \{\varphi \in \DMZ: \varphi(0) \in A, S_{[0,T]}(\varphi|\varphi(0)) \leq s\}.
\end{align*}
We now prove a certain containment property for elements of $\MZ$ that can arise as end-points of trajectories in $\Phi_{K(\Delta)}^{[0,T]}(s)$, $s > 0$ and $\Delta > 0$, i.e., points $\xi \in \MZ$ such that there exists a trajectory $\varphi$ with $\varphi_0 \in K(\Delta)$ and $S_{[0,T]}(\varphi | \varphi_0) \leq s$. We prove that such points are not far from the lower level sets of $V$ in $\MZ$. This connection between trajectories over  finite time horizons and the level sets of the quasipotential $V$ is the key to transfer the process-level LDP upper bound in Theorem~\ref{thm:uniform-ldp-mun} to the LDP upper bound for the family of invariant measures $\{\wp^N, N \geq 1\}$.
\begin{lemma}
\label{lemma:upper-bound-inclusion}
For any $s > 0$ and $\delta > 0$ there exists $\Delta > 0$ and $T_1 \geq 1$ such that for all $ T \geq T_1$,
\begin{align}
\{\varphi(T) : \varphi \in \Phi_{K(\Delta)}^{[0,T]}(s)\} \subset \{\xi\in \MZ: d(\xi, \Xi(s)) \leq \delta\}.
\label{eqn:upper-bound-inclusion}
\end{align}
\end{lemma}
\begin{proof}
Suppose not. Then there exist $s > 0$, $\delta > 0$, sequences  $\{\Delta_n, n \geq 1\}$, $\{T_n, n \geq 1\}$ such that $\Delta_n \downarrow 0$ and $T_n \uparrow \infty$ as $n \to \infty$, and trajectories $\varphi_n \in \Phi_{K(\Delta_n)}^{[0,T_n]}(s)$ such that $d(\varphi_n(T_n), \Xi(s)) > \delta$ for each $n \geq 1$. Let $\xi_n = \varphi_n(T_n)$, $n \geq 1$. By Lemma~\ref{lemma:v-continuity}, there exists a $T^\prime > 0$ and a sequence $\{\varepsilon_n, n \geq 1\}$, with $\varepsilon_n \to 0$ as $n \to \infty$, such that for any $\zeta^\prime \in K(\Delta_n)$ there exists a trajectory $\bar{\varphi}^{\zeta^\prime}$ on $[0,T^\prime]$ such that $\bar{\varphi}^{\zeta^\prime}(0) = \xi^*, \bar{\varphi}^{\zeta^\prime}(T^\prime) = \zeta^\prime$, and  $S_{[0,T^\prime]}(\bar{\varphi}^{\zeta^\prime} | \xi^*) \leq \varepsilon_n$. For each $n \geq 1$, let $\tilde{\varphi}_n$ be the trajectory on $[0,T^\prime+T_n]$ defined as follows. Let $\tilde{\varphi}_n(0) = \xi^*$; $\tilde{\varphi}_n(t) = \bar{\varphi}^{\varphi_n(0)}(t)$ on $t \in [0,T^\prime]$; $\tilde{\varphi}_n(t) = \varphi_n(t-T^\prime)$ on $t \in (T^\prime, T^\prime+T_n]$. In particular, $\tilde{\varphi}_n(T^\prime+T_n) = \xi_n$. Clearly, $S_{[0,T^\prime+T_n]}(\tilde{\varphi}_n | \xi^*) \leq s + \varepsilon_n$. It follows that $V(\xi_n) \leq s + \varepsilon_n$. Using the compactness of the lower level sets of $V$ (see Lemma~\ref{lemma:v-compactness-level-sets}), we can find a convergent subsequence of $\{\xi_n, n \geq 1\}$; after re-indexing and denoting this convergent subsequence by $\{\xi_n, n \geq 1\}$, let $\xi_n \to \xi$ in $\MZ$ as $n \to \infty$. By assumption, $d(\xi_n, \Xi(s)) > \delta$ for each $n \geq 1$, and hence $d(\xi, \Xi(s)) \geq \delta$.  Using the lower semicontinuity of $V$, we see that
\begin{align*}
V(\xi) \leq \liminf_{n \to \infty} V(\xi_n) \leq \liminf_{n \to \infty} (s+\varepsilon_n)= s.
\end{align*}
Hence $\xi \in \Xi(s)$. This contradicts  $d(\xi, \Xi(s)) \geq \delta$, which is a  consequence of our assumption. This proves the lemma.
\end{proof}

We are now ready to prove the LDP upper bound for the family $\{\wp^N, N \geq 1\}$. The proof relies on the uniform LDP upper bound in Theorem~\ref{thm:uniform-ldp-mun},  the exponential tightness of the family $\{\wp^N, N \geq 1\}$, the containment property established in Lemma~\ref{lemma:upper-bound-inclusion}, an estimate on the probability that $\mu^N$ lies in $\mathscr{S}_m(M,\Delta)$ (which uses the process-level uniform LDP upper bound in Theorem~\ref{thm:uniform-ldp-mun} and the result of Lemma~\ref{lemma:upper-bound-cost-lower-bound}), and finally  the strong Markov property of $\mu^N$.
\begin{lemma}
\label{lemma:upper-bound}
For any $\gamma > 0$, $\delta > 0$, and $s > 0$, there exists $N_0 \geq 1$ such that 
\begin{align*}
\wp^N\{\zeta \in \MZ : d(\zeta, \Xi(s)) \geq \delta\} \leq \exp\{-N(s - \gamma)\}
\end{align*}
for all $N \geq N_0$.
\end{lemma}
\begin{proof}
Fix $\gamma > 0$, $\delta > 0$, and $s > 0$. Choose $M > 0$ and $N_1 \geq 1$ such that $\wp^N(\sim \hspace{-0.4em} \K_M) \leq \exp\{-Ns\}$ for all $N \geq N_1$; this is possible from the exponential tightness of the family $\{\wp^N, N \geq 1\}$, see Proposition~\ref{prop:invariant-measure-exp-tightness}. For the given $s > 0$ and $\delta > 0$, from Lemma~\ref{lemma:upper-bound-inclusion}, choose $\Delta > 0$ and $T_1 > 0$ such that~\eqref{eqn:upper-bound-inclusion} holds for all $T \geq T_1$. For the above choice of $\Delta > 0$ and $M > 0$, by Lemma~\ref{lemma:upper-bound-cost-lower-bound},  choose $m_0 \in \N$ such that such that~\eqref{eqn:upper-bound-cost-lower-bound} holds. By~\eqref{eqn:upper-bound-cost-lower-bound} and the compactness of $\Phi^{[0,m_0]}_{\K_M}(s)$ in $D([0,m_0], \MZ)$ (which follows from Lemma \ref{lemma:compactness-S}), the closure of $\mathscr{S}_{m_0}(\Delta, M)$ does not intersect $\Phi^{[0,m_0]}_{\K_M}(s)$. It follows that there exists a $\delta_0 > 0$ such that $\varphi \in \mathscr{S}_{m_0}(\Delta, M)$ implies $\rho(\varphi, \Phi_{\K_M}^{[0,m_0]}(s)) \geq \delta_0$. Hence by the uniform LDP upper bound in Theorem~\ref{thm:uniform-ldp-mun}, there exists $N_2 \geq N_1$ such that
\begin{align}
\PN_\zeta(\mu^N \in \mathscr{S}_{m_0}(\Delta, M)) &\leq \PN_{\zeta}(\rho(\mu^N, \Phi_{\K_M}^{[0,m_0]}(s)) \geq \delta_0) \nonumber \\
& \leq  \exp\{-N(s-\gamma/2)\}
\label{eqn:upper-bound-cost-lower-bound-probability}
\end{align}
for all $\zeta \in \K_M \cap \MNZ$ and $N \geq N_2$. Thus, with $T = m_0 + T_1$ and $N \geq N_2$, we have
\begin{align}
\wp^N \{\zeta \in \MZ &: d(\zeta, \Xi(s))\geq \delta\} \nonumber \\
& = \int_{\MNZ} \PN_\zeta(d(\mu^N(T), \Xi(s)) \geq \delta) \wp^N(d\zeta) \nonumber \\
& \leq \exp\{-Ns\} + \int_{\K_M \cap \MNZ} \PN_\zeta(d(\mu^N(T), \Xi(s)) \geq \delta) \wp^N(d\zeta) \nonumber \\
& \leq \exp\{-Ns\} + \sup_{\zeta \in \K_M \cap \MNZ} \PN_\zeta(\mu^N \in \mathscr{S}_{m_0}(\Delta, M)) \nonumber \\
& \qquad +  \int_{\K_M \cap \MNZ} \PN_\zeta(\mu^N \notin \mathscr{S}_{m_0}(\Delta, M), d(\mu^N(T), \Xi(s)) \geq \delta) \wp^N(d\zeta) \nonumber \\
& \leq \exp\{-Ns\} + \exp\{-N(s-\gamma/2)\} \nonumber \\
& \qquad +  \int_{\K_M \cap \MNZ} \PN_\zeta(\mu^N \notin \mathscr{S}_{m_0}(\Delta, M), d(\mu^N(T), \Xi(s)) \geq \delta) \wp^N(d\zeta); \label{eqn:upper-bound-series1}
\end{align}
here the first equality follows since $\wp^N$ is invariant to time shifts, the first inequality follows from the choice of $M$, and the third inequality follows from~\eqref{eqn:upper-bound-cost-lower-bound-probability}.

To bound the integrand in the third term above, let $T^\prime \geq T_1$ and $\zeta^\prime \in K(\Delta)$. Choose\footnote{The existence of such a  $\delta^\prime$ can be justified via arguments similar to those used in the proof of Lemma \ref{lemma:lower-bound-specific-path}; see the paragraph before \eqref{eqn:lb-series}} $0 < \delta^\prime < \delta $ (depending on $T$ and $s$, and not on $\zeta^\prime$ and $T^\prime$) such that $\rho(\varphi_1, \varphi_2) < \delta^\prime/2$ implies $d(\varphi_1(T^\prime), \varphi_2(T^\prime)) < \delta/2$ whenever $\varphi_1 \in D([0,T^\prime], \MZ)$ and $\varphi_2 \in \Phi^{[0,T^\prime]}_{\zeta^\prime}(s)$. Note that if a trajectory $\varphi$ on $[0,T^\prime]$ with initial condition $\varphi(0) = \zeta^\prime$ is such that $\rho(\varphi, \Phi_{\zeta^\prime}^{[0,T^\prime]}(s)) < \delta^\prime/2$, then there exists a trajectory $\varphi^\prime \in \Phi_{\zeta^\prime}^{[0,T^\prime]}(s)$ such that $\rho(\varphi, \varphi^\prime) < \delta^\prime/2$. By the choice of $\delta^\prime$, we have $d(\varphi(T^\prime), \varphi^\prime(T^\prime))<\delta/2$. By Lemma~\ref{lemma:upper-bound-inclusion}, we find that $d(\varphi^\prime(T^\prime), \Xi(s)) \leq \delta^\prime/2$.  Hence by triangle inequality $d(\varphi(T^\prime), \Xi(s)) < \delta/2 + \delta^\prime/2 < \delta$. The contrapositive of the above statement is
\begin{align*}
d(\varphi(T^\prime), \Xi(s)) \geq \delta \Rightarrow \rho(\varphi, \Phi_{\zeta^\prime}^{[0,T^\prime]}(s)) \geq \delta^\prime/2.
\end{align*}
We therefore conclude that
\begin{align}
\PN_{\zeta^\prime}(d(\mu^N(T^\prime), \Xi(s)) \geq \delta) \leq \PN_{\zeta^\prime}(\rho(\mu^N, \Phi_{\zeta^\prime}^{[0,T^\prime]}(s)) \geq \delta^\prime/2) \label{eqn:upper-bound-inclusion-probability}
\end{align}
for all $T^\prime \geq T_1$, $\zeta^\prime \in \K(\Delta)\cap \MNZ$, and $N \geq 1$.

Note that the integrand in the last term of~\eqref{eqn:upper-bound-series1} can be upper bounded by
\begin{align}
\PN_\zeta & (\mu^N \notin \mathscr{S}_{m_0}(\Delta, M), d(\mu^N(T), \Xi(s)) \geq \delta) \nonumber \\
& = \PN_\zeta(\mu^N(m) \in K(\Delta) \text{ for some }m=1,2,\ldots, m_0, \,  d(\mu^N(T), \Xi(s)) \geq \delta ) \nonumber \\
& \leq \sum_{m=1}^{m_0} \sup_{\zeta^\prime \in K(\Delta) \cap \MNZ} \PN_{\zeta^\prime} (d(\mu^N(T-m),\Xi(s)) \geq \delta) \nonumber \\
& \leq \sum_{m=1}^{m_0} \sup_{\zeta^\prime \in K(\Delta) \cap \MNZ} \PN_{\zeta^\prime}(\rho((\mu^N(t), 0 \leq t \leq T-m), \Phi_{\zeta^\prime}^{[0,T-m]}(s)) \geq \delta^\prime/2) \label{eqn:upper-bound-series2}
\end{align}
where the first inequality follows from the strong Markov property of $\mu^N$ and the second inequality follows from~(\ref{eqn:upper-bound-inclusion-probability}) by the choice of $T$. By the uniform LDP upper bound in Theorem~\ref{thm:uniform-ldp-mun}, for each $m = 1,2,\ldots m_0$, there exist $N(m) \geq N_2 $ such that
\begin{align*}
\PN_{\zeta^\prime}(\rho((\mu^N(t), 0 \leq t \leq T-m), \Phi_{\zeta^\prime}^{[0,T-m]}(s)) \geq \delta^\prime/2 ) \leq \exp\{-N(s-\gamma/2)\}
\end{align*}
for all $\zeta^\prime \in \K(\Delta) \cap \MNZ$ and $N \geq N(m)$. Put $N_3 = \max\{N(m), m = 1,2,\ldots, m_0, N_1, N_2\}$. Then~(\ref{eqn:upper-bound-series2}) yields
\begin{align*}
\PN_\zeta & (\mu^N \notin \mathscr{S}_{m_0}(\Delta, M), d(\mu^N(T), \Xi(s)) \geq \delta)  \leq m_0 \exp\{-N(s - \gamma/2)\}
\end{align*} 
for all $\zeta \in \K_M \cap \MNZ$ and $N \geq N_3$. Substitution of this back in~(\ref{eqn:upper-bound-series1}) yields
\begin{align*}
\wp^N \{\zeta \in \MZ &: d(\zeta, \Xi(s))\geq \delta\} \leq \exp\{-Ns\} + (m_0+1)\exp\{-N(s-\gamma/2)\}
\end{align*}
for all $N \geq N_3$. Finally, choose $N_0 \geq N_3$ such that $1+(m_0+1)\exp\{N\gamma/2\} \leq \exp\{N\gamma\}$ for all $N \geq N_0$. Then the above display becomes
\begin{align*}
\wp^N \{\zeta \in \MZ &: d(\zeta, \Xi(s))\geq \delta\} \leq \exp\{-N(s-\gamma)\}
\end{align*}
for all $N \geq N_0$. This completes the proof of the lemma.
\end{proof}
\section{Proof of Theorem~\ref{thm:main-result}}
\label{section:proof-completion}
We now complete the proof of Theorem~\ref{thm:main-result}.
\begin{itemize}
\item (Compactness of level sets). For any $s > 0$, by Lemma~\ref{lemma:v-compactness-level-sets}, the set $\Xi(s) = \{\xi \in \MZ: V(\xi) \leq s\}$ is a compact subset of $\MZ$;
\item (LDP lower bound). Given $\gamma > 0$, $\delta > 0$, and $\xi \in \MZ$, by Lemma~\ref{lemma:lower-bound-specific-path}, there exists $N_0 \geq 1$ such that 
\begin{align*}
\wp^N \{\zeta \in \MZ : d(\zeta, \xi) < \delta\} \geq \exp\{-N(V(\xi) + \gamma)\}
\end{align*}
for all $N \geq N_0$;
\item (LDP upper bound). Given $\gamma > 0$, $\delta > 0$, and $s > 0$, by Lemma~\ref{lemma:upper-bound}, there exists $N_0 \geq 1$ such that
\begin{align*}
\wp^N\{\zeta \in \MZ : d(\zeta, \Xi(s)) \geq \delta\} \leq \exp\{-N(s-\gamma)\}
\end{align*}
for all $N \geq N_0$.
\end{itemize} 
This completes the proof of Theorem~\ref{thm:main-result}.

\section{Two counterexamples}
\label{section:counterexamples}
In this section, for two non-interacting counterexamples described in Section~\ref{subsection:counterexamples-introduction}, we prove that the quasipotential is not equal to the relative entropy with respect to the corresponding globally asymptotically stable equilibrium. These two counterexamples are (i) a system of non-interacting M/M/1 queues, and (ii) a system of non-interacting nodes in a wireless local area network (WLAN) with constant forward transition  rates.  We detail the proofs  in the case of non-interacting M/M/1 queues. Similar arguments carry over to the case of non-interacting WLAN system with constant forward transition rates as well.
\subsection{A system of non-interacting M/M/1 queues}
Recall the system of non-interacting M/M/1 queues described in Section~\ref{subsection:counterexamples-introduction-mm1}. Recall the relative entropy from~\eqref{eqn:entropy} and the process-level large deviations rate function from~\eqref{eqn:rate-function-finite-duration-counter-example}. Also recall the function $\vartheta$ defined in \eqref{eqn:theta} and the compact sets  $\K_M$, $M > 0$,  defined in \eqref{eqn:KM}. Define the quasipotential
\begin{align*}
V_Q(\xi) \coloneqq \inf \{S^Q_{[0,T]}(\varphi | \xi^*_Q), \varphi(0) = \xi^*_Q, \varphi(T) = \xi, T > 0\}, \, \xi \in \MZ,
\end{align*}
where $S^Q$ is defined by~\eqref{eqn:rate-function-finite-duration-counter-example} with $\E$ replaced by $\EQ$ and $L_\zeta$ replaced by $L^Q$ for each $\zeta \in \MZ$.

We first prove that the quasipotential $V_Q$ is not finite outside $\K$. The key property used for this is the fact that the attractor $\xi^*_Q$ has geometric decay. As a consequence $\langle \xi^*_Q, \vartheta \rangle < \infty$. Using this property, we first show that if $\xi \notin \K$, then the associated quasipotential evaluated at $\xi$ cannot be finite. This is shown by producing a lower bound for the cost of any trajectory starting at $\xi^*_Q$ and ending at $\xi \notin \K$ from the rate function in~\eqref{eqn:rate-function-finite-duration-counter-example}.
\begin{lemma}
\label{lemma:mm1-xi-infinite}
If $\xi \in \MZ$ is such that $\xi \notin \K$, then $V_Q(\xi) = \infty$.
\end{lemma}
\begin{proof}
Fix $\xi \in \MZ$. Let $T > 0$ and $\varphi \in \DMZ$ be such that $\varphi_0 = \xi^*_Q$ and  $\varphi_T = \xi$. For each $n \geq 1$, define $f_n$ by
\begin{align*}
f_n(z) =  \left\{
\begin{array}{ll}
z, \text{ if } z \leq n\\
2n-z, \text{ if } n+1 \leq z \leq 2n,\\
0, \text{ if }z > 2n,
\end{array}
\right.
\end{align*}
and define $f_\infty(z) = z$ for each $z \in \Z$. Note that the use of $f_n$ is to approximate $f_\infty$ using  $C_0(\Z)$ functions so that we can insert them into \eqref{eqn:rate-function-finite-duration-counter-example}. We first assume that $\langle \xi , f_\infty \rangle = \infty$. In particular, $\xi \notin \K$. Using the function $f_n$ in place of $f$ in the RHS of \eqref{eqn:rate-function-finite-duration-counter-example}, we have
\begin{align*}
S_{[0,T]}^Q(\varphi | \xi^*_Q) & \geq \langle \varphi_T , f_n \rangle  - \langle \xi^*_Q , f_n \rangle - \int_{[0,T]}\langle \varphi_u, L^Q f_n \rangle - \int_{[0,T]} \sum_{(z,z^\prime) \in \EQ}\tau(f_n(z^\prime) - f_n(z)) \lambda_{z,z^\prime} \varphi_u(z) du \\
& =\langle \varphi_T , f_n \rangle  - \langle \xi^*_Q , f_n \rangle -\int_{[0,T]} \sum_{(z,z^\prime) \in \EQ} (\exp\{f_n(z^\prime) - f_n(z)\}-1) \lambda_{z,z^\prime} \varphi_u(z) du,
\end{align*}
where $\lambda_{z,z+1} = \lambda_f$, $z \in \Z$, and $\lambda_{z,z-1} = \lambda_b$, $z \in \Z \setminus \{0\}$. Noting that $f_n(z^\prime) - f_n(z)$ is either $1$, $0$ or $-1$ for each $(z,z^\prime) \in \EQ$, we have  $\sum_{(z,z^\prime) \in \EQ} (\exp\{f_n(z^\prime) - f_n(z)\}-1) \lambda_{z,z^\prime} \varphi_u(z) \leq 2(e-1)\lambda_b$ for each $u \in [0,T]$. Hence the above becomes
\begin{align*}
S_{[0,T]}^Q(\varphi | \xi^*_Q) & \geq \langle \varphi_T , f_n \rangle  - \langle \xi^*_Q , f_n \rangle -2(e-1)\lambda_bT.
\end{align*}
Note that $\langle \xi^*_Q, f_\infty \rangle < \infty$. Hence, letting $n \to \infty$ and using the monotone convergence theorem,  we conclude that $S_{[0,T]}^Q(\varphi | \xi^*_Q) = \infty$.

We now assume that $\xi \notin \K$ is such that $\langle \xi, f_\infty \rangle < \infty$. Let $T > 0$ and $\varphi \in \DMZ$ be such that $\varphi_0 = \xi^*_Q$ and $\varphi_T = \xi$. Without loss of generality, we can assume that $\sup_{t \in [0,T]} \langle \varphi_t , f_\infty \rangle < \infty$; otherwise the argument in the above paragraph shows that $S_{[0,T]}^Q(\varphi | \xi^*_Q) = \infty$. Define 
\begin{align*}
\vartheta_n(z) = \left\{
\begin{array}{ll}
\vartheta(z), \text{ if } z \leq n, \\
\vartheta(2n - z) \text{ if } n+1 \leq z \leq 2n,\\
0, \text{ if } z >2n.
\end{array}
\right.
\end{align*}
Using $\vartheta_n$ in the RHS of \eqref{eqn:rate-function-finite-duration-counter-example}, we get
\begin{align*}
S_{[0,T]}^Q(\varphi | \xi^*_Q) \geq \langle \xi, \vartheta_n \rangle -  \langle \xi^*_Q, \vartheta_n \rangle - \int_{[0,T]}\sum_{(z,z^\prime) \in \EQ} (\exp\{\vartheta_n(z^\prime) - \vartheta_n(z)\} -1) \lambda_{z,z^\prime} \varphi_u(z) du.
\end{align*}
Noting that $\vartheta_n(z^\prime) - \vartheta_n(z)$ can be upper bounded by $1 + \log(z+1)$ for each $(z,z^\prime) \in \EQ$, it follows that $\sum_{(z,z^\prime) \in \EQ} (\exp\{\vartheta_n(z^\prime) - \vartheta_n(z)\} -1) \lambda_{z,z^\prime} \varphi_u(z) \leq 2\lambda_b(e (\sup_{t \in [0,T]} \langle \varphi_t, f_\infty \rangle+1)-1)$ for each $u \in [0,T]$. Hence the above display becomes
\begin{align*}
S_{[0,T]}^Q(\varphi | \xi^*_Q) \geq \langle \xi, \vartheta_n \rangle -  \langle \xi^*_Q, \vartheta_n \rangle -2\lambda_b(e (\sup_{t \in [0,T]} \langle \varphi_t, f_\infty \rangle+1)-1)T.
\end{align*}
As before, letting $n \to \infty$, using the monotone convergence theorem, and noting that $\xi^*_Q \in \K$, we conclude that $S_{[0,T]}^Q(\varphi | \xi^*_Q) = \infty$. 

Since $\xi \notin \K$, $T > 0$, and $\varphi \in \DMZ$ such that $\varphi_0 = \xi^*_Q$ and $\varphi_T = \xi$ are arbitrary, the proof of the lemma is complete.
\end{proof}

We now prove the main result of this section, namely, the quasipotential $V_Q$ is not equal to the relative entropy $I(\cdot \| \xi^*_Q)$.
\begin{proposition} Let $\xi \in \MZ$ be such that $\langle \xi, f_\infty \rangle < \infty$ and $\xi \notin \K$. Then $I(\xi \| \xi_Q^*) < \infty$ and $V(\xi) = \infty$. In particular, $V \neq I(\cdot \| \xi_Q^*)$.
\label{thm:mm1-counterexample}
\end{proposition}
\begin{proof}
By the Donsker-Varadhan variational formula (see Donsker and Varadhan~\cite[Lemma~2.1]{donsker-varadhan-75-1}), for any $\xi \in \MZ$ and any bounded function $f$ on $\Z$, we have
\begin{align*}
I(\xi \| \xi^*_Q) \geq \langle \xi, f \rangle - \log \left( \sum_{z \in \Z} \exp\{f(z)\} \xi^*_Q(z)\right).
\end{align*}
Recall the definition of $f_n$ and $f_\infty$ from the proof of Lemma~\ref{lemma:mm1-xi-infinite}. Let $\bar{\beta} > 0$ be such that $\sum_{z \in \Z} \exp\{\bar{\beta}z\} \xi^*_Q(z) < \infty$. Replacing  $f$ by $\bar{\beta}f_n$ in the above display, letting $n \to \infty$ and using the monotone convergence theorem, we arrive at
\begin{align*}
\bar{\beta}\langle \xi, f_\infty \rangle \leq I(\xi \| \xi^*_Q) + \log \left(\sum_{z \in \Z}\exp\{\bar{\beta}z\} \xi^*_Q(z)\right).
\end{align*}
It follows that
\begin{align*}
\{\xi \in \MZ : I(\xi \| \xi^*_Q) < \infty\} \subset \{\xi \in \MZ : \langle \xi, f_\infty \rangle < \infty\}.
\end{align*}
On the other hand, since $\langle \xi_Q^*, f_\infty \rangle < \infty$, it is easy to check that $\{\xi \in \MZ : I(\xi \| \xi^*_Q) < \infty\} \supset \{\xi \in \MZ : \langle \xi, f_\infty \rangle < \infty\}$. 

Let $\xi \in \MZ$ be such that $\langle \xi, \vartheta \rangle = \infty$ and $\langle \xi, f_\infty \rangle < \infty$. Then the above yields $I(\xi \| \xi^*_Q) < \infty$. By Lemma~\ref{lemma:mm1-xi-infinite}, we see that $V_Q(\xi) = \infty$. This completes the proof of the proposition.
\end{proof}

\subsection{A non-interacting WLAN system with constant forward rates}
\label{subsection:counterexamples-wlan}
Recall the model described in Section~\ref{subsection:counterexamples-introduction-wlan}. Define the quasipotential
\begin{align*}
V_W(\xi) \coloneqq \inf \{S^W_{[0,T]}(\varphi | \xi^*_W), \varphi_0 = \xi^*_W, \varphi_T = \xi, T > 0\},\,  \xi \in \MZ,
\end{align*}
where $S^W$ is defined by~\eqref{eqn:rate-function-finite-duration-counter-example} with $\E$ replaced by $\EW$ and $L_\zeta$ replaced by $L^W$ for each $\zeta \in \MZ$. We now state the main result for this non-interacting wireless local area network.
\begin{proposition}
Let $\xi \in \MZ$ be such that $\langle \xi, f_\infty \rangle < \infty$ and $\xi \notin \K$. Then $I(\xi \| \xi^*_W) < \infty$ and $V(\xi) = \infty$. In particular, $V_W \neq I(\cdot \| \xi^*_W)$.
\label{thm:wlan-counterexample}
\end{proposition}

We start with the following lemma. The proof follows along similar lines of the proof of Lemma~\ref{lemma:mm1-xi-infinite} by noting that $\langle \xi^*_W, \vartheta  \rangle  <\infty$, and it is left to the reader.
\begin{lemma} If $\xi \in \MZ$ is such that $\xi \notin \K$, then $V_W(\xi) = \infty$.
\end{lemma}
\noindent Using the above lemma, we can now prove Proposition~\ref{thm:wlan-counterexample} along similar lines of the proof of Proposition~\ref{thm:mm1-counterexample} in the previous section.

\appendix

\section{Proofs of Section~\ref{section:prelims-results}}
\label{appendix:proofs}
\subsection{Proof of Lemma~\ref{lemma:compactness-S}}
Fix $T > 0$, $s > 0$, and $K \subset \MZ$ compact. Given $\nu \in K$, $\varphi \in \Phi_\nu^{[0,T]}(s)$ and a finite set $B \subset \Z$, choosing $f(t,z) = \indf{\{z \in B\}}$ for all $t \in [0,T]$,~(\ref{eqn:weak-equation}) yields
\begin{align*}
\varphi_t(B) - \varphi_r(B) &= \int_{[r,t]} \sum_{(z,z^\prime)\in \E} (f(z^\prime) - f(z)) (1+h_\varphi(u,z,z^\prime)) \lambda_{z,z^\prime}(\varphi_u) \varphi_u(z) du
\end{align*}
for all $0 \leq r < t \leq T$.  Note that we may take $h_\varphi \geq -1$, else the rate function would be infinite
as per \eqref{eqn:rate-function-nonvar} and the definition of $\tau^*$ in \eqref{eqn:taustar}. Therefore, we get 
\begin{align}
|\varphi_t(B) - \varphi_r(B)| & \leq \int_{[0,T]} \sum_{(z,z^\prime) \in \E}(1+h_\varphi(u,z,z^\prime)) \times \indf{\{u \in [r,t]\}} \lambda_{z,z^\prime}(\varphi_u)\varphi_u(z) du.
\label{eqn:finite-compactness}
\end{align}
Noting that
\begin{align*}
\sup\left\{\int_{[0,T]} \sum_{(z,z^\prime) \in \E}\tau^*(h_\varphi(u,z,z^\prime)) \lambda_{z,z^\prime}(\varphi_u)\varphi_u(z) du, \varphi \in \Phi_\nu^{[0,T]}(s), \nu \in K \right\} \leq s,
\end{align*}
it follows that the family $\{1+h_\varphi, \varphi \in \Phi_\nu^{[0,T]}(s), \nu \in K\}$ is uniformly integrable. That is,
\begin{align*}
\sup \left\{\int_{[0,T]} (1+h_\varphi(u,z,z^\prime)) \times \indf{\{1+h_\varphi \geq M\}}\lambda_{z,z^\prime}(\varphi_u)\varphi_u(z) du, \varphi \in \Phi_\nu^{[0,T]}(s), \nu \in K \right\} \to 0
\end{align*}
as $M \to \infty$. Hence for any $M > 0$, using the boundedness of the transition rates (from assumption~\ref{assm:a2}), \eqref{eqn:finite-compactness} yields
\begin{align*}
| & \varphi_t(B) - \varphi_r(B)|\\
& \leq 2M\overline{\lambda} (t-r) + \int_{[0,T]} \sum_{(z,z^\prime) \in \E}(1+h_\varphi(u,z,z^\prime)) \times \indf{\{ 1+h_\varphi \geq M\}} \lambda_{z,z^\prime}(\varphi_u)\varphi_u(z) du.
\end{align*}
for all $0 \leq r < t \leq T$, and $B \subset \MZ$. It follows that
\begin{align*}
& \sup_{\varphi \in \cup_{\nu \in K}\Phi_\nu^{[0,T]}(s)} \sup_{t,r:|t-r|\leq \delta} d(\varphi_t, \varphi_r) \\
& \qquad \leq 2M\overline{\lambda} \delta + \sup_{\varphi \in \cup_{\nu \in K} \Phi_\nu^{[0,T]}(s)} \sup_{t,r:|t-r|\leq \delta} \int_{[0,T]}  \sum_{(z,z^\prime) \in \E}(1+h_\varphi(u,z,z^\prime)) \\
& \qquad \qquad \times \indf{\{1+h_\varphi \geq M\}} \lambda_{z,z^\prime}(\varphi_u)\varphi_u(z) du
\end{align*}
Letting $\delta \to 0$ first and then $M \to \infty$, we arrive at
\begin{align*}
\lim_{\delta \downarrow 0 }\sup_{\varphi \in \cup_{\nu \in K}\Phi_\nu^{[0,T]}(s)} \sup_{t,r:|t-r|\leq \delta} d(\varphi_t, \varphi_r) = 0.
\end{align*}
Hence it follows that $\cup_{\nu \in K}\Phi_\nu^{[0,T]}(s)$ is precompact in $\DMZ$ (see, for example,~Billingsley~\cite[Theorem~12.3]{billingsley-convergence}).

To show that $\cup_{\nu \in K}\Phi_\nu^{[0,T]}(s)$  is closed, let $\{\varphi_n, n\geq 1\} \subset \cup_{\nu \in K}\Phi_\nu^{[0,T]}(s)$ and suppose that $\varphi_n \to \bar{\varphi}$ in $\DMZ$. Note that, for any $f \in C_0^1([0,T] \times \MZ)$, the mapping
\begin{align*}
\varphi & \mapsto \Biggr\{\langle \varphi_T , f_T \rangle  - \langle \varphi_0 , f_0 \rangle - \int_{[0,T]}\langle \varphi_u , \partial_u f_u \rangle du \nonumber \\
& \qquad  - \int_{[0,T]}\langle \varphi_u, L_{\varphi_u} f_u \rangle du - \int_{[0,T]} \sum_{(z,z^\prime) \in \E}\tau(f_u(z^\prime) - f_u(z)) \lambda_{z,z^\prime}(\varphi_u) \varphi_u(z) du \Biggr\}
\end{align*}
is continuous on $\DMZ$, and hence, the mapping
\begin{align*}
\varphi & \mapsto \sup_{f \in C_0^1([0,T]\times\Z)} \Biggr\{\langle \varphi_T , f_T \rangle  - \langle \varphi_0 , f_0 \rangle - \int_{[0,T]}\langle \varphi_u , \partial_u f_u \rangle du \nonumber \\
& \qquad  - \int_{[0,T]}\langle \varphi_u, L_{\varphi_u} f_u \rangle du - \int_{[0,T]} \sum_{(z,z^\prime) \in \E}\tau(f_u(z^\prime) - f_u(z)) \lambda_{z,z^\prime}(\varphi_u) \varphi_u(z) du \Biggr\}
\end{align*}
is lower semicontinuous on $\DMZ$ (see, for example,~Berge~\cite[Theorem~1,~page~115]{berge-topology}). Hence, it follows that
\begin{align*}
S_{[0,T]}(\bar{\varphi}|\bar{\varphi}(0)) \leq \liminf_{n \to \infty} S_{[0,T]}(\varphi_n|\varphi_n(0)) \leq s,
\end{align*}
and it follows that $\cup_{\nu \in K} \Phi_\nu^{[0,T]}(s)$ is closed. Consequently, $\cup_{\nu \in K} \Phi_\nu^{[0,T]}(s)$ is a compact subset of $\DMZ$.
\qed

\subsection{Proof of Theorem~\ref{thm:uniform-ldp-mun}}
In this Section, we prove Theorem \ref{thm:uniform-ldp-mun}. In the case  of finite state space (i.e., when $\Z$ is a finite set), the LDP for the family $\{\mu^N_{\nu_N}, N \geq 1\}$, whenever $\nu_N \to \nu$ in $\MZ$ as $N \to \infty$, was proved in \cite[Theorem 3.1]{borkar-sundaresan-12} under suitable assumptions. The main assumption required in the proof of \cite[Theorem 3.1]{borkar-sundaresan-12} was the boundedness of the ``total outgoing jump rate'' across all the states, which also holds in our countable state space case under Assumptions \ref{assm:a1}--\ref{assm:a3}. So, to prove the LDP for the family $\{\mu^N_{\nu_N}, N \geq 1\}$, whenever $\nu_N \to \nu$ in $\MZ$ as $N \to \infty$, one can go through the steps in \cite[Section 5]{borkar-sundaresan-12} verbatim; we reproduce the important steps here for the sake of completeness. Once this LDP is proved, we then show the uniform LDP over the class of compact subsets of $\MZ$ using \cite[Proposition 1.12, 1.14]{budhiraja-dupuis-book}.

\subsubsection{LDP for $\{\mu_{\nu_N}^N, N \geq 1\}$ when $\nu_N \to \nu$ in $\MZ$}
We first introduce some notation.
Let $\{(X^N_n(t), t \in [0,T]), 1 \leq n \leq N\}$ denote the joint evolution of the states of all the particles. This is a Markov process on $\Z^N$ with the infinitesimal generator acting on functions $f$ on $\Z^N$ given by
\begin{align*}
(z_1, \ldots, z_N) \mapsto \sum_{n=1}^N \sum_{z_n^\prime \in \{z_n+1, 0\}}(f(z_1, \ldots, z_n^\prime, \ldots, z_N) - f(z_1, \ldots, z_N)) \lambda_{z_n, z_n^\prime}(\emp(z_1, \ldots, z_N)),
\end{align*}
where $\emp(z_1, \ldots, z_N):=\frac{1}{N} \sum_{n=1}^N \delta_{z_n} \in \MNZ$. Define the empirical measure
\begin{align*}
\Theta^N : = \frac{1}{N} \sum_{n = 1}^N \delta_{X^N_n(\cdot)};
\end{align*}
$\Theta^N$ is a $\MDZ$-valued random variable. Let $\sigma: \MDZ \to \DMZ$ denote the canonical projection map. Note that $\mu^N(t) = \sigma_t(\Theta^N)$, $t \in [0,T]$. Similarly, let $\{(\bar{X}^N_n(t), t \in [0,T]), 1 \leq n  \leq N\}$ denote the evolution of the independent particles, where each particle executes a Markov process with the   infinitesimal generator $\bar{L}$ defined in \eqref{eqn:Lbar}. Define the corresponding empirical measure $\bar{\Theta}^N$ by
\begin{align*}
\bar{\Theta}^N : = \frac{1}{N} \sum_{n = 1}^N \delta_{\bar{X}^N_n(\cdot)}.
\end{align*}
Let $\PQN_{\nu_N}$ (resp. $\PbarQN_{\nu_N}$) denote the law of $\Theta^N$ (resp. $\bar{\Theta}^N$) with initial condition $\nu_N \in \MNZ$ (i.e., $\frac{1}{N} \sum_{n=1}^N\delta_{X^N_n(0)} = \nu_N$). These are probability measures on $\MDZ$, i.e., $\PQN_{\nu_N}, \PbarQN_{\nu_N} \in  \M_1(\MDZ)$.

Note that $\PQN_{\nu_N} \ll \PbarQN_{\nu_N}$. For $x \in \DZ$ and $\mu \in \DMZ$, define
\begin{align*}
h(x; \mu): = \sum_{t \in [0,T]}\indf{\{x(t) \neq x(t-)\} }  \log \left(\frac{\lambda_{x(t-), x(t)}(\mu_t)}{\widetilde{\lambda}_{x(t-), x(t)}}\right) - \int_{[0,T]}  \sum_{\stackrel{z^\prime \in \Z:}{(x(t-), z^\prime) \in \E}}  \left(\lambda_{x(t-), z^\prime}(\mu_t)  - \widetilde{\lambda}_{x(t-), z^\prime} \right) dt,
\end{align*}
where $\{\widetilde{\lambda}_{z,z^\prime}, (z,z^\prime) \in \E\}$ are the non-interacting rates defined by
\begin{align*}
\widetilde{\lambda}_{z,z^\prime} := \begin{cases}
\overline{\lambda}/(z+1) & \text{ if } z^\prime = z+1, \\
\underline{\lambda} & \text{ if } z^\prime = 0, z \geq 1.
\end{cases}
\end{align*}
Also, define
\begin{align}
h(Q) := \int_{\DZ} h( \cdot ; \sigma(Q)) \, dQ, \quad Q \in \MDZ.
\label{eqn:h}
\end{align}
Using Girsanov's theorem, it is straightforward to check that
\begin{align*}
\frac{d\PQN_{\nu_N}}{d\PbarQN_{\nu_N}}(Q) = \exp\{Nh(Q)\}, \quad Q \in \MDZ.
\end{align*}

We now introduce some notation related to path spaces. Define $\psi : \DZ \to \{0,1,\ldots\}$  by
\begin{align*}
\psi(x) = \sum_{t \in [0,T]} \indf{\{x(t) \neq x(t-)\}};
\end{align*}
$\psi(x)$ is the number of discontinuities in $x$. Since $\Z$ is a countable set, it follows that $\psi(x) < \infty$ for all $x \in \DZ$ (\cite[Chapter 3, Lemma 1]{billingsley-convergence}). Define
\begin{align*}
\X: = \{x \in \DZ : \psi(x) < \infty,  (x(t-), x(t)) \in \E \text{ whenever } x(t) \neq x(t-), \,  t \in [0,T] \},
\end{align*}
and equip $\X$ with the subspace topology. Since $\Z$ is countable, we have that $\psi$ is continuous on $\X$. Define
\begin{align*}
\| f\|_\psi := \sup_{x \in \X} \frac{|f(x)|}{1+\psi(x)}, \, \text{ for } f : \X \to \R.
\end{align*}
Then, define
\begin{align*}
\C_\psi(\X) : = \{f : \X \to \R \text{ such that } f \text{ is continuous and } \|f\|_\psi < \infty\},
\end{align*}
and
\begin{align*}
\Msi := \left\{Q \in \M_1(\X):  \int_{\X} \psi \, dQ < \infty \right\}.
\end{align*}
$\Msi$ is a subset of $\C_\psi(\X)^*$, the algebraic dual of $\C_\psi(\X)$, and we equip it with the weak* topology. This is the coarsest topology on $\Msi$ where we say $Q_N \to Q$ in $\Msi$ as $N \to \infty$ if and only if
\begin{align*}
\int_\X f \, dQ_N \to \int_\X f \, dQ \, \, \text{ as } N \to \infty, \quad \text{ for all } f \in \C_\psi(\X).
\end{align*}

Recall $\bar{P}_z$, $z \in \Z$, from Section \ref{section:exp-tightness}. For each $\nu \in \MZ$, define $J : \M_1(\X) \to [0,\infty]$ by
\begin{align}
J(Q) : =\sup_{f \in \C_\psi(\X)} \left[ \int_{\X} f \, dQ -  \sum_{z \in \Z} \nu(z) \log \int_{\X} \exp\{f\} \, d\bar{P}_z\right].
\label{eqn:JQ}
\end{align}
By \cite[Lemma 5.3]{borkar-sundaresan-12}, we also have
\begin{align}
J(Q)  =\sup_{f \in \C_b(\X)} \left[ \int_{\X} f \, dQ -  \sum_{z \in \Z} \nu(z) \log \int_{\X} \exp\{f\} \, d\bar{P}_z\right],
\label{eqn:JQ-b}
\end{align}
where $\C_b(\X)$ is the space of bounded and continuous functions on $\X$ equipped with the supremum norm.

We first state a lemma for the LDP for  the family $\{\PbarQN_{\nu_N}, N \geq 1\}$ on $\Msi$  whenever $\nu_N \to \nu$ in $\MZ$ as $N \to \infty$. Its proof follows verbatim from \cite[Lemma 5.1]{borkar-sundaresan-12}.
\begin{lemma}[{LDP for the non-interacting system; \cite[Lemma 5.1]{borkar-sundaresan-12}}] Let $\nu_N \to \nu$ in $\MZ$ as $N \to \infty$. Then the family $\{\PbarQN_{\nu_N}, N \geq 1\}$ satisfies the LDP on $\Msi$ with rate function $J$ defined in \eqref{eqn:JQ}.
\label{lemma:nonin}
\end{lemma}

Next, we provide two necessary conditions for the finiteness of $J$ defined in \eqref{eqn:JQ}.
\begin{lemma}[{Finiteness of $J$; \cite[Lemma 5.2]{borkar-sundaresan-12}}] If $J(Q) < \infty$, then  we have $Q \in \Msi$ and $Q \circ \projinv= \nu$.
\label{lemma:Q-Msi}
\end{lemma}
\begin{proof}
Let $Q$ be such that $J(Q) < \infty$. The proof of $Q \circ \projinv= \nu$ follows verbatim from \cite[Lemma 5.2]{borkar-sundaresan-12}. For the first assertion, since $\psi \in \C_\psi(\X)$, from the definition of $J$ in \eqref{eqn:JQ}, we have
\begin{align}
J(Q) \geq \int_\X \psi \, dQ - \sum_{z \in \Z} \nu(z) \log \int_\X \exp\{\psi\} \, d\bar{P}_z.
\label{eqn:J-psi}
\end{align}
Note that, for each $z \in \Z$, under $\bar{P}_z$ (see \eqref{eqn:Lbar}), the number of jumps on $[0,T]$ is stochastically dominated by a Poisson random variable with parameter $(\underline{\lambda}+\overline{\lambda})T \leq 2\overline{\lambda}T$. Therefore,
\begin{align*}
\int_\X \exp\{\psi\} \, d\bar{P}_z \leq \sum_{k \geq 0} \exp\{k\} \frac{ \exp\{-2\overline{\lambda}T\}(2\overline{\lambda}T)^k}{k!} = c_1 < \infty,
\end{align*}
where $c_1$ is some constant independent of $z$. Therefore,
\begin{align*}
\sum_{z \in \Z} \nu(z)  \log \int_\X \exp\{\psi\} d\bar{P}_z < \infty.
\end{align*}
Hence, from \eqref{eqn:J-psi}, using $J(Q) < \infty$, we conclude that
\begin{align*}
\int_\X \psi \, dQ < \infty.
\end{align*}
It follows that $Q \in \Msi$.
\end{proof}

The next lemma is required to prove the continuity of $h$ on $\Msi$.
\begin{lemma}[{see \cite[Lemma 5.7]{borkar-sundaresan-12} for the finite state space case}] Suppose that $Q \in \M_1(\X)$ is such that $J(Q) < \infty$. Then,
\begin{align*}
\lim_{\alpha \to 0}\sup_{t \in [0,T]}\int_\X \sup_{u \in [t-\alpha, t+\alpha] \cap [0,T]} \indf{\{X(u) \neq X(u-)\}}\, dQ(X) = 0.
\end{align*}
\label{lemma:exp-zero}
\end{lemma}
\begin{proof}
Let $P \in \M_1(\X)$ denote the mixture distribution defined by $dP := \sum_{z \in \Z} \nu(z) d\bar{P}_z$. Since $J(Q) < \infty$, it follows that $Q \ll P$. Indeed, using Jensen's inequality, we have,
\begin{align*}
\sum_{z \in \Z} \log \int_\X \exp\{f\} \, dP_z \leq \log \int_\X \exp\{f\} \, dP\, \quad \text{ for any } f \in \C_b(\X), 
\end{align*}
and hence, from \eqref{eqn:JQ-b} and the Donsker-Varadhan variational formula for $I(Q \| P)$, we conclude that 
\begin{align}
I(Q \| P) \leq J(Q).
\label{eqn:IJQ}
\end{align}
Since $J(Q) < \infty$, the above implies that $I(Q \| P) < \infty$.  This shows $Q \ll P$. Hence, with $K_{t, \alpha} = \{x \in \X : x(u) \neq x(u-) \text{ for some } u \in [t-\alpha, t+\alpha] \cap [0,T]\}$, we have
\begin{align}
\int_\X \sup_{u \in [t-\alpha, t+\alpha] \cap [0,T]} \indf{\{X(u) \neq X(u-)\}} \, dQ(X)  & = Q(K_{t, \alpha}) \nonumber \\
& = \int_\X \left(\frac{dQ}{dP}\right) \indf{\{K_{t, \alpha}\}} \, dP \nonumber \\
& \leq \left\|\left(\frac{dQ}{dP}\right) \right\|_{\tau^*, P}  \|\indf{\{K_{t, \alpha}\}} \|_{\tau, P},  \label{eqn:bound-tautaustar}
\end{align}
where the last inequality follows from the H\"older's inequality in Orlicz spaces. Here, $\|\cdot\|_{\tau, P}$ is the Orlicz norm defined by
\begin{align*}
\|f \|_{\tau, P}: = \inf \left\{ a > 0 : \int_\X \tau\left(\frac{|f(x)|}{a}\right) \, dP(x) \leq 1 \right\}.
\end{align*}
Similarly, $\|f \|_{\tau^*, P}$ is defined as above with $\tau$ replaced by $\tau^*$. 

Consider $\left\|\left(\frac{dQ}{dP}\right) \right\|_{\tau^*, P}$. Note that, there exists a $u_0 \geq 1$ such that $\tau^*(u) \leq 2 u \log u$ for all $u \geq u_0$. Therefore,
\begin{align*}
\int_\X \tau^*\left( \frac{dQ}{dP}\right) \, dP & \leq \tau^*(u_0) + 2 \int_\X \left( \frac{dQ}{dP}\right) \log \left( \frac{dQ}{dP}\right) \, dP \\
& = \tau^*(u_0) + 2I(Q \| P) \\
& \leq \tau^*(u_0) + 2 J(Q)\\
& < \infty,
\end{align*}
where the second inequality follows from \eqref{eqn:IJQ} and the third inequality follow from the assumption that $J(Q) < \infty$. Since $\tau^*(u/a) \leq \tau^*(u)/a$ for $a \geq 1$ (by Jensen's inequality), this shows 
\begin{align}
\left\|\left(\frac{dQ}{dP}\right) \right\|_{\tau^*, P} < c_2 < \infty \text{ for some } c_2 \text{ that does not depend on } t.
\label{eqn:bound-tau}
\end{align}

Next, consider  $ \|\indf{\{K_{t, \alpha}\}} \|_{\tau^*, P}$. Note that, under $P$, the number of jumps in $[t-\alpha, t+\alpha] \cap [0,T]$  is stochastically dominated by a Poisson random variable with parameter $2\alpha(\overline{\lambda} + \underline{\lambda}) \leq 4 \alpha \overline{\lambda}$. Therefore, $P(K_{t, \alpha}) \leq 1 - \exp\{-4\alpha\overline{\lambda}\} \leq 4\alpha\overline{\lambda}$. Since $\tau(\indf{\{K_{t, \alpha}\}}/a) = \tau(1/a) \indf{\{K_{t, \alpha}\}}$ for any $a > 0$, we have
\begin{align*}
\int_\X \tau\left(\indf{\{K_{t, \alpha}\}}/a\right) \, dP  = \tau(1/a) P(K_{t, \alpha}) \leq \tau(1/a) 4\alpha\overline{\lambda}.
\end{align*}
Therefore, if we choose $a = 1/(\tau^{-1}(1/4\alpha\overline{\lambda}))$, the right-hand side of the above display becomes $1$. This shows
\begin{align*}
\|\indf{\{K_{t, \alpha}\}} \|_{\tau^*, P} \leq \frac{1}{\tau^{-1}(1/4\alpha\overline{\lambda})}
\end{align*}
for all $t$. Hence, by \eqref{eqn:bound-tautaustar}, \eqref{eqn:bound-tau}, and the previous display, we get
\begin{align*}
\sup_{t \in [0,T]}\int_\X \sup_{u \in [t-\alpha, t+\alpha] \cap [0,T]} \indf{\{X(u) \neq X(u-)\}} \, dQ(X)  \leq \frac{c_2}{\tau^{-1}(1/4\alpha\overline{\lambda})} \to 0 \quad \text{ as  }\alpha \to 0.
\end{align*}
This completes the proof of the lemma.
\end{proof}

Next, we argue the continuity of the projection map $\sigma$.  
\begin{lemma}[{Continuity of $\sigma$; \cite[Lemma 5.8]{bordenave-etal-12}}] Let $Q \in \M_1(\X)$ be such that $J(Q) < \infty$. Then $\sigma : \MDZ \to \DMZ$ is continuous at $Q$.
\label{lemma:sigma-cont}
\end{lemma}
\begin{proof}
Let $Q \in \M_1(\X)$ be such that $J(Q) < \infty$. By Lemma \ref{lemma:Q-Msi}, it follows that $Q \in \Msi$. In \cite[Lemma 2.8]{leonard-95}, for the case when $\nu = \delta_{z_0}$ for some $z_0 \in \Z$, it was shown that $\sigma : \MDZ \to \DMZ$ is continuous at $Q$ whenever $Q \in \Msi$\footnote{ This continuity was shown in \cite{leonard-95} when $\MDZ$ is equipped with the usual weak topology and $\DMZ$ is equipped with the stronger uniform topology. Since the Skorohod topology on $\DMZ$ is coarser than the uniform topology, it follows that $\sigma$ is continuous}. For general $\nu \in \MZ$, by using the result of Lemma \ref{lemma:exp-zero} and following the proof of \cite[Lemma 2.8]{leonard-95} verbatim, we arrive at the continuity of $\sigma$.
\end{proof}

Finally, we have that $h$ is continuous on $\Msi$.

\begin{lemma} Assume \ref{assm:a1}, \ref{assm:a2}, and \ref{assm:a3}. Then the mapping $h$ defined in \eqref{eqn:h} is continuous on $\Msi$.
\label{lemma:h-cont}
\end{lemma}
\begin{proof}
Using Lemma \ref{lemma:exp-zero},  Lemma \ref{lemma:sigma-cont} and  Assumptions \ref{assm:a1}--\ref{assm:a3}, the proof of \cite[Lemma 2.9]{leonard-95} holds verbatim.
\end{proof}

The above lemmas give us the LDP for the family $\{\mu^N_{\nu_N}, N \geq 1\}$ on $\DMZ$ whenever $\nu_N \to \nu$ in $\MZ$ as $N \to \infty$.
\begin{proposition}
\label{prop:finite-ldp}
Assume \ref{assm:a1}, \ref{assm:a2}, and \ref{assm:a3}.  Suppose that $\nu_N \to \nu$ in $\MZ$ as $N \to \infty$. Then, the family $\{\mu^N_{\nu_N}, N \geq 1\}$ satisfies the LDP on $\DMZ$ with rate function $S_{[0,T]}(\cdot | \nu)$ defined in \eqref{eqn:rate-function-finite-duration}.
\end{proposition}
\begin{proof}
Let $\nu_N \to \nu$ in $\MZ$ as $N \to \infty$. By Lemma \ref{lemma:nonin}, we have that $\{\PbarQN_{\nu_N}, N \geq 1\}$ satisfies the LDP on $\Msi$ with rate function $J$. Since $h$ is continuous on the set $\{Q \in \Msi: J(Q) < \infty\}$ (by Lemma \ref{lemma:h-cont}), from Varadhan's lemma, one can conclude that (see \cite[Proof of Theorem 3.1]{borkar-sundaresan-12}) the family $\{\PQN_{\nu_N}\}$ satisfies the LDP on $\Msi$ with rate function $Q \mapsto J(Q) - h(Q)$.  
By Lemma \ref{lemma:sigma-cont}, since  $\sigma$ is continuous (with the usual weak topology on $\MDZ$) at $Q$  when $J(Q) < \infty$, it follows that the restriction of $\sigma$ to $\Msi$  is also continuous (with respect to the stronger topology on $\Msi$) at $Q$ when $J(Q) < \infty$.  Therefore, using the generalized contraction principle (e.g., \cite[Theorem 4.2.23]{dembo-zeitouni}), the LDP for the family $\{\mu^N_{\nu_N}, N \geq 1\}$ on $\DMZ$ follows. The rate function for this LDP can be shown to admit the form given in \eqref{eqn:rate-function-finite-duration} (see, e.g., \cite[Proof of Theorem 3.1]{leonard-95}).
\end{proof}

\subsubsection{Uniform LDP for $\{\mu_{\nu_N}^N, N \geq 1\}$ over the class of compact subsets of $\MZ$}
Proposition \ref{prop:finite-ldp} establishes the LDP for the family $\{\mu^N_{\nu_N}, N \geq 1\}$, whenever $\nu_N \to \nu$ in $\MZ$ as $N \to \infty$. We now extend this to the uniform LDP on the class of compact subsets of $\MZ$. Towards this, we rely  on \cite[Proposition 1.12, 1.14]{budhiraja-dupuis-book}.  Although our definition of the  uniform LDP (Definition \ref{def:uniform-ldp}) has initial conditions lying in $A \cap \MNZ$ (unlike the definition of uniform LDP in  \cite[Definition 1.13]{budhiraja-dupuis-book} where the initial conditions do not depend on the parameter $N$), we can use straightforward modifications of the arguments in \cite[Proposition 1.12, 1.14]{budhiraja-dupuis-book} to prove the desired  uniform LDP. We provide an outline of these arguments here.

We first provide a definition of the uniform Laplace principle over the class of compact subsets of $\MZ$. Recall the definition of the rate function $S_{[0,T]}$ in \eqref{eqn:rate-function-finite-duration}. For $\nu \in \MZ$ and $g \in C_b(\DMZ)$, define
\begin{align*}
F(\nu, g) := - \inf_{\varphi \in \DMZ} \left[g(\varphi) + S_{[0,T]}(\varphi | \nu)\right].
\end{align*}

\begin{definition} 
\label{def:ulaplace}
We say that the family $\{\mu^N_{\nu_N}, N \geq 1\}$ of $\DMZ$-valued random variables defined on a probability space $(\Omega, \mathcal{F}, P)$ satisfies the uniform Laplace principle over the class $\mathcal{A}$ of subsets of $\MZ$ with the family of rate functions $\{S_{[0,T]}(\cdot | \nu), \nu \in \MZ\}$, $S_{[0,T]}(\cdot | \nu) : \DMZ \to [0, +\infty]$, $\nu \in \MZ$, if
\begin{itemize}
\item (Compactness of level sets). For each $K \subset \MZ$ compact and $s \geq 0$, $\bigcup_{\nu \in K}\Phi_\nu(s)$ is a compact subset of $D([0,T],\MZ)$, where $\Phi_\nu(s) \coloneqq \{\varphi \in D([0,T],\MZ): \varphi_0 = \nu, S_{[0,T]}(\varphi | \nu) \leq s\}$;
\item (Uniform Laplace asymptotics) For any $A \in \mathcal{A}$ and $g \in C_b(\DMZ)$, we have
\begin{align*}
\lim_{N \to \infty} \sup_{\nu_N \in A \cap \MNZ} \left|\frac{1}{N} \log E_{\nu_N}\left[ \exp\{-Ng(\mu^N_{\nu_N})\} \right] - F(\nu_N, g)  \right| = 0.
\end{align*}
\end{itemize}
\end{definition}
\noindent This is a modification of \cite[Definition 1.11]{budhiraja-dupuis-book} to the case when the initial conditions are only allowed to lie in $A \cap \MNZ$. We have the following result.
\begin{lemma}[{see \cite[Proposition 1.12]{budhiraja-dupuis-book}}] Assume \ref{assm:a1}, \ref{assm:a2}, and \ref{assm:a3}. 
Then the family $\{\mu^N_{\nu_N}, N \geq 1\}$ satisfies the uniform Laplace principle over the class of compact subsets of $\MZ$ with the family of rate functions $\{S_{[0,T]}(\cdot | \nu), \nu \in \MZ\}$, $S_{[0,T]}(\cdot | \nu) : \DMZ \to [0, +\infty]$, $\nu \in \MZ$.
\label{lemma:unif-laplace}
\end{lemma}
\begin{proof}
By Lemma \ref{lemma:compactness-S}, we have that for each $K \subset \MZ$ compact and $s \geq 0$, $\bigcup_{\nu \in K}\Phi_\nu(s)$ is a compact subset of $D([0,T],\MZ)$, where $\Phi_\nu(s) \coloneqq \{\varphi \in D([0,T],\MZ): \varphi_0 = \nu, I_\nu(\varphi) \leq s\}$.

To show the uniform Laplace asymptotics, let $g \in C_b(\DMZ)$. By Proposition \ref{prop:finite-ldp}, whenever $\nu_N \to \nu$ in  $\MZ$ as $N \to \infty$ we have that the family $\{\mu^N_{\nu_N}, N \geq 1\}$ satisfies the LDP on $\DMZ$ with rate function $S_{[0,T]}(\cdot | \nu)$. Therefore, by Varadhan's lemma (e.g., \cite[Theorem 4.3.1]{dembo-zeitouni}), we have
\begin{align}
\lim_{N \to \infty}\frac{1}{N} \log E_{\nu_N}\left[ \exp\{-Ng(\mu^N_{\nu_N})\} \right]  = F(\nu, g).
\label{eqn:ldp-eqn1}
\end{align}
Define 
\begin{align*}
F^N(\nu_N^\prime, g) : =\frac{1}{N} \log E_{\nu_N^\prime}\left[ \exp\{-Ng(\mu^N_{\nu_N^\prime})\} \right], \quad \nu_N^\prime \in \MNZ.
\end{align*}
Using \eqref{eqn:ldp-eqn1}, we now show that the mapping  $\nu \mapsto F(\nu, g)$ is continuous. To show this continuity, it suffices to show that, given any $\varepsilon > 0$ there exists $\delta > 0$ such that for all $\nu^\prime \in \MZ$ such that $d(\nu^\prime, \nu) < \delta$ and $\nu^\prime_N \in \MNZ$ such that $\nu_N^\prime \to \nu^\prime$ as $N \to \infty$, we have
\begin{align*}
|F^N(\nu_N^\prime,g) - F(\nu, g)| < \varepsilon \quad \text{ for all large enough }   N.
\end{align*}
Indeed, if this is true, sending $N \to \infty$ in the above display and using \eqref{eqn:ldp-eqn1}, we arrive at $|F(\nu^\prime, g) - F(\nu, g) | < \varepsilon$, which shows the continuity of $\nu \mapsto F(\nu, g)$. We now show the above statement using contraposition. Suppose the above statement is not true. Then there exists  $\varepsilon > 0$ and a sequence $\{\nu_N\}$ with  $\nu_N \in \MNZ$ and $\nu_N \to \nu$ as $N \to \infty$ such that $|F^N(\nu_N, g) - F(\nu, g)| > \varepsilon$. Using \eqref{eqn:ldp-eqn1}, we get $|F(\nu, g) - F(\nu, g)| > \varepsilon > 0$, which is a contradiction. This establishes the continuity of the mapping $\MZ \ni \nu \mapsto F(\nu, g)$.

Since $\nu \mapsto F(\nu, g)$ is continuous, using \eqref{eqn:ldp-eqn1},  by the same  arguments in \cite[Proposition 1.12]{budhiraja-dupuis-book}, one can show that  for any compact subset $K$ of $\MZ$,  $\sup_{\nu_N \in K \cap \MNZ} | F^N(\nu_N, g) - F(\nu, g)| \to 0$ as $N \to \infty$. This shows that the family $\{\mu^N_{\nu_N}, N \geq 1\}$ satisfies the uniform Laplace principle over the class of compact subsets of $\MZ$ with the family of rate functions $\{S_{[0,T]}(\cdot | \nu), \nu \in \MZ\}$.
\end{proof}

We can now complete the proof of Theorem \ref{thm:uniform-ldp-mun} using the arguments in \cite[Proposition 1.14]{budhiraja-dupuis-book}.

\begin{proof}[Proof of Theorem \ref{thm:uniform-ldp-mun}]
By Lemma \ref{lemma:unif-laplace}, the family $\{\mu^N_{\nu_N}, N \geq 1\}$ satisfies the uniform Laplace principle $\DMZ$ over the class of compact subsets of $\MZ$ with the family of rate functions $\{S_{[0,T]}(\cdot | \nu), \nu \in \MZ\}$. Restricting the initial conditions to $\MNZ$ and  following the proof of \cite[Proposition 1.14]{budhiraja-dupuis-book} verbatim, we conclude that the family $\{\mu^N_{\nu_N}, N \geq 1\}$ satisfies the uniform LDP on $\MDZ$ over the class of compact subsets of $\MZ$ with the family of rate functions $\{S_{[0,T]}(\cdot | \nu), \nu \in \MZ\}$.
\end{proof}

\section*{Acknowledgements}
The authors were supported by a grant from the Indo-French Centre for Applied Mathematics on a project titled ``Metastability phenomena in algorithms and engineered systems". The first author was supported in part by a fellowship grant from the Centre for Networked Intelligence (a Cisco CSR initiative), Indian Institute of Science, Bangalore; and in part by Office of Naval Research under the Vannevar Bush Faculty Fellowship N0014-21-1-2887. The authors thank two anonymous referees for carefully reading the manuscript and providing valuable comments that improved the paper.

\bibliographystyle{abbrv}
\bibliography{Report}

\noindent \scriptsize{\textsc{S. Yasodharan\newline
Division of Applied Mathematics \newline
Brown University\newline
Providence RI 02912, USA\newline}}
email: \texttt{sarath@alum.iisc.ac.in}

\noindent \scriptsize{\textsc{R. Sundaresan\newline
Department of Electrical Communication Engineering \newline
Indian Institute of Science\newline
Bangalore 560\,012, India\newline}}
email: \texttt{rajeshs@iisc.ac.in}
\end{document}